\documentclass[11pt]{article}
\usepackage{eurosym}
\usepackage{amsfonts}
\usepackage{amsmath}
\usepackage{geometry}
\usepackage{amssymb}
\usepackage{mathabx}
\usepackage{graphicx}
\usepackage{color}
\usepackage{dsfont}
\usepackage{bbm}
\usepackage[dvipsnames]{xcolor}
\usepackage{mathrsfs}
\usepackage[author={Oana}]{pdfcomment}
\usepackage[colorinlistoftodos]{todonotes}
\usepackage{titling}
\usepackage{lipsum}
\usepackage{bbm}
\usepackage{soul}
\usepackage{amssymb}
\usepackage{amsthm}
\usepackage{calc}
\usepackage{accents}
\usepackage{stmaryrd}
\usepackage{physics}
\usepackage{bm}

\newtheorem{theorem}{Theorem}

\newtheorem{assumptions}[theorem]{Assumptions}

\newtheorem{corollary}[theorem]{Corollary}

\newtheorem{definition}[theorem]{Definition}

\newtheorem{lemma}[theorem]{Lemma}

\newtheorem{proposition}[theorem]{Proposition}
\newtheorem{remark}[theorem]{Remark}

\def\cB{\mathcal{B}}

\def\cL{\mathcal{L}}
\def\T{\mathbb{T}}
\def\R{\mathbb{R}}
\def\cP{\mathcal{P}}
\def\txtd{{\textnormal{d}}}

\begin{document}

\date{}
\title{\textbf{Well-posedness and Hurst parameter estimation for fluid equations driven by fractional transport noise}}
\author{Alexandra Blessing Neam\c tu \thanks{Department of Mathematics and Statistics, University of Konstanz, Germany, e-mail: alexandra.blessing@uni-konstanz.de}  \qquad Dan Crisan\thanks{Department of Mathematics, Imperial College London, United Kingdom, e-mail: d.crisan@imperial.ac.uk} \qquad Oana Lang\thanks{Department of Mathematics, Babe\c{s}-Bolyai University, Cluj-Napoca,
Romania, e-mail: oana.lang@ubbcluj.ro}}

\maketitle

\begin{abstract}
We study a two-dimensional incompressible vorticity equation on the torus driven by transport-type fractional Brownian noise with Hurst parameter $H \in (1/2,1)$. 
The model captures persistent, long-range correlated forcing consistent with inertial-range scaling laws and fractional Brownian approximations of turbulent fluctuations.
A central ingredient of our approach is a version of the sewing lemma adapted to a class of integrands that includes, but is not limited to, transport-type structures. This result provides a flexible tool for constructing the Young integral and serves as a basis for analysing a wider class of stochastic partial differential equations. Using this approach, we establish existence and uniqueness of solutions via a fixed point argument and investigate statistical properties of the flow. In particular, we study quadratic functionals of the solution and derive an estimator for the Hurst parameter $H$.
\end{abstract}

\tableofcontents

\section{Introduction}
Our work is motivated by classical and modern statistical theories of two-dimensional turbulence, in particular the dual-cascade framework initiated by Kraichnan \cite{kraichnan}, which predicts self-similar inertial-range behaviour and inverse energy transfer in the presence of conserved vorticity and energy. In such a setting, velocity and vorticity fluctuations are expected to exhibit power-law scaling. Under Taylor's frozen turbulence hypothesis (\cite{taylor}) explained in detail below, this scaling is often heuristically associated with fractional Brownian motion with Hurst parameter $H \approx 1/3$. While this regime lies beyond the scope of the present work, it serves as a guiding heuristic. Here, we take a first step by focusing instead on the more regular case $H \in (1/2,1)$, where the analysis is more tractable.

Consider a fluid flow which evolves with velocity $u$ and let $\omega = \curl u$ be the corresponding vorticity. Assume that $u(x,t,w), \omega(x,t,w)$, for $x \in \mathbb{T}^2$ and $t \geq 0$, are defined on a probability space $(\Omega, \mathcal{F}, \mathbb{P})$.
We study the two-dimensional incompressible vorticity equation perturbed by transport-type fractional Brownian noise
\begin{equation}\label{eq:mainStrat}
d\omega_t+u_{t}\cdot \nabla \omega_tdt+\displaystyle
\mathcal{L}_{\xi}\omega_t dW_{t}^{H}=\Delta\omega_tdt
\end{equation}
with initial condition $\omega _{0}$. We have $\nabla \cdot u = 0$, $\omega_t=curl\ u_t$, while $\xi$ is a time-independent divergence-free vector field (so $\nabla \cdot \xi = 0$), and the operator $\mathcal{L}$ is given by $\mathcal{L}_{\xi}\omega_t := \xi \cdot \nabla \omega_t$. Here $W^H$ is a fractional Brownian motion with Hurst parameter $H >1/2$.

By considering a two-dimensional vorticity equation driven by fractional Brownian noise with $H>1/2$, our approach provides a mathematically tractable stochastic representation of persistent, long-range-correlated forcing consistent with the phenomenological descriptions developed in \cite{kolmogorov}, \cite{kraichnan}, \cite{taylor}. The analysis of existence, uniqueness, and statistical estimation of the Hurst parameter developed here establishes a rigorous foundation for linking stochastic partial differential equation models with classical turbulence theory and stochastic parametrizations, while offering a framework for quantifying memory effects and scale-invariant variability in two-dimensional turbulent flows.

In Kolmogorov's view (\emph{K41 theory}), three-dimensional turbulence can be explained as a predominantly one-way cascade of energy across scales: energy enters at large scales, it is then transferred to smaller scales, and eventually, at very small scales, energy is transformed into heat, under the effect of viscosity. Moreover, Kolmogorov derives an average rate $\epsilon$ at which this dissipation occurs. The small-scale statistics in locally isotropic fluids depends on two factors: the average dissipation rate $\epsilon$ and the viscosity $\nu$. At intermediate scales (\emph{inertial range}), fluid statistics depends uniquely on $\epsilon$, energy being transferred without loss. Based on these two hypotheses of similarity, the \emph{$r^{2/3}$ law is derived}. This is one of the most famous results in turbulence theory and it says that the second-order velocity structure function scales as: 
\begin{equation}
\mathbb{E}\left[|u(x+r) - u(x)|^2\right] \sim \epsilon^{2/3}r^{2/3}
\end{equation}
where $\mathbb{E}$ is the expected value with respect to the probability measure $\mathbb{P}$.
In \cite{kraichnan}, Kraichnan investigated the structure of inertial ranges in two-dimensional turbulence, highlighting the fundamental role played by the simultaneous conservation of kinetic energy and mean-square vorticity in the inviscid limit. Unlike the three-dimensional case, where energy predominantly cascades toward small scales, Kraichnan showed that the two-dimensional case is characterised by a two-way energy transfer: an inverse cascade of energy toward large scales represented through an energy spectrum $E(k)\sim k^{-5/3}$, and a direct cascade of enstrophy toward small scales associated with a $k^{-3}$ spectrum. Using a detailed analysis of triadic interactions in Fourier space and similarity arguments, he established that each inertial range transports only one invariant, while the other remains asymptotically conserved. Kraichnan's theory forms the basis for a so-called \emph{phenomenological correspondence} between structure functions and energy spectra, extending Kolmogorov’s ideas to the two-dimensional setting.

Overall, the approach introduced in \cite{kolmogorov} and \cite{kraichnan} establishes a \emph{phenomenological correspondence} (see e.g. \cite{schochtel}) between the second-order structure function
\[
\mathbb{E}\left[|u(x+r,t)-u(x,t)|^2\right])=C|r|^{\alpha-1}
\]
and the energy spectrum $E(\cdot)$ of the velocity field $u$,
\[
E(k)=\widetilde{C}\,k^{-\alpha}
\]
where $C,\widetilde{C}>0$ are constants, $1<\alpha<3$, $r \in \mathbb{R}^2$, and $k>0$ in the inertial subrange.
In particular, when $\alpha = 5/3$, one can recover the \emph{Kolmogorov two-thirds law} and the \emph{Kolmogorov five-thirds law} (also known as the Kolmogorov energy spectrum), respectively.

While stochastic triad models provide a reduced-order dynamical framework that preserves key structural features such as helicity or energy conservation, they necessarily operate at the level of finitely many interacting Fourier modes. As such, they capture localized mechanisms of nonlinear energy exchange but do not fully address the influence of temporally correlated, multi-scale forcing on the full vorticity field. 

In \cite{ourtriadpaper} the authors have investigated triad interactions driven by transport type noise (SALT noise actually). In traditional turbulence theory, interactions among three Fourier modes (“triads”) are the building blocks of energy transfer across scales. In \cite{ourtriadpaper} we have looked at stochastic parametrisations (Stochastic Advection by Lie Transport/SALT and Location Uncertainty/LU) projected onto such triads, investigating helicity-preserving or energy-preserving triad models. These models retain fundamental nonlinear interactions, but add stochastic forcing representing unresolved scales or uncertainty.

Extending this perspective to the two-dimensional vorticity equation driven by fractional Brownian motion with Hurst parameter $H>1/2$ is natural. Fractional Brownian forcing introduces long-range temporal dependence and persistent correlations, features that are increasingly recognized as relevant in geophysical and large-scale turbulent flows. In contrast to white-in-time stochastic perturbations, the regularity regime $H>1/2$ permits pathwise analytical treatment while modeling memory effects that cannot be captured within classical Markovian frameworks. Thus, moving from stochastic triad interactions to a full SPDE driven by fractional Brownian noise bridges reduced stochastic parametrizations and infinite-dimensional turbulence models, providing a mathematically rigorous setting in which nonlinear cascade mechanisms and correlated stochastic forcing coexist. This step brings the modeling framework closer to the statistical and scaling structures observed in two-dimensional turbulence, while maintaining analytical tractability necessary for establishing existence, uniqueness, and parameter inference results.

A connection between the classical theory of turbulence and fractional Brownian motion has been established in the literature through the \emph{Taylor’s frozen turbulence hypothesis}, see e.g. \cite{schochtel}, \cite{taylor}.

Taylor’s frozen turbulence hypothesis establishes a correspondence between temporal fluctuations observed at fixed spatial locations and the underlying spatial structure of turbulent flows under strong mean advection, enabling the interpretation of time series data in terms of inertial-range energy spectra. Building on the theory developed in \cite{kraichnan}, Taylor \cite{taylor} shows that spatial cascade dynamics and scaling laws can be translated into temporally correlated stochastic behavior, providing a natural justification for fractional Brownian motion approximations of turbulent forcing. 

Taylor’s hypothesis states that, for any scalar-valued fluid-mechanical quantity $\xi$ (for example, $u_i$, $i=1,2$), we have (see \cite{schochtel})
\[
\frac{\partial \xi}{\partial t}
=
-|\bar{U}|\,\frac{\partial \xi}{\partial \bar{x}}.
\]

The frozen turbulence hypothesis allows one to express the statistical properties of spatial increments $u(x+r,t)-u(x,t)$ in terms of temporal increments $u(x,t)-u(x,t+s)$ at a fixed time $t$. The above relation implies that
\[
u_i(x,t+s)=u_i(x-\bar{U}s,t),
\]
and therefore, using the relation for spatial increments, we obtain
\[
\mathbb{E}\bigl[|u(x,t)-u(x,t+s)|^2\bigr]
=
C|\bar{U}|^{\alpha-1}s^{\alpha-1}
\]
in the inertial subrange along the time axis.

Comparing this scaling with the increment structure of fractional Brownian motion,
\[
\mathbb{E}\bigl[|W_{t+s}^H - W_t^H|^2\bigr] = s^{2H}
\]
we identify the relation
\[
2H = \alpha - 1.
\]
In particular, the Kolmogorov-type scaling $\alpha = \tfrac{5}{3}$ corresponds formally to $H = \tfrac{1}{3}$. This correspondence suggests that it is natural to model the random velocity field $u$ using noise driven by fractional Brownian motion $(W_t^H)_{t\in\mathbb{R}}$ with Hurst parameter $H \in (0,1)$.
To duplicate the order of the time increments obtained by the Taylor's frozen turbulence hypothesis, $H$ must be chosen to be equal to $(\alpha-1)/2$.


Inspired by these considerations, we provide an SPDE-formulation to these phenomenological ideas, in which long-range temporal correlations and scale-invariant variability are incorporated through fractional Brownian noise, and their impact on the dynamics of the vorticity equation is analyzed in terms of existence, uniqueness, and statistical parameter estimation.~For other fluid dynamical models driven by fractional Brownian motion, we refer to~\cite{antonio,flandoli3,flandoli1,flandoli2,XM}.

Our approach is applicable to general class of equations of the form 

\begin{equation}
d\omega _{t}+\left( \mathcal{D+E}\right) \omega _{t}dt+\mathcal{F}\omega
_{t}dW_{t}^{H}=0,  \label{general:intro}
\end{equation}%
where $\mathcal{D}$ is a nonlinear operator, $\mathcal{F}$ is a (possibly) nonlinear operator, $\mathcal{E}$ is an (unbounded) linear operator, see section \ref{sect:general} for details and assumptions. For our arguments, it is essential to assume that $H>1/2$. This is due to the fact that the stochastic convolution improves the spatial regularity by a parameter which is strictly less than the H\"older regularity of the fractional Brownian motion, see Corollary \ref{i:reg}. Since $H>1/2$, this allows us to incorporate transport-type noise. In the Young regime, locally monotone stochastic partial differential equations with linear or Lipschitz continuous nonlinear multiplicative noise were treated in \cite{YR}. For $H\in(1/3,1/2]$ we refer to \cite{BG,deya, martina,martina2,antoine,James,XM} that consider rough partial differential equations in the framework of unbounded rough drivers.
We choose to work with the mild formulation of \eqref{eq:mainStrat} which enables us to incorporate possibly nonlinear diffusion coefficients (see \eqref{general:intro}) and to exploit optimal regularity results of the solution. Moreover, the equivalence between mild and weak solution is used for the estimation of the Hurst parameter.~Therefore, our approach provides Hurst parameter estimations for stochastic partial differential equations with transport-type noise as well as nonlinear multiplicative noise.
 
\newpage
\noindent\emph{Outline of the paper}
\vspace{3mm}

 In Section \ref{sect:preliminaries} we introduce the main notations and preliminaries.  In Section \ref{sect:sewinglemma} we develop a version of the sewing lemma adapted to our setting, where the integrands in the Young integral are quite general and, in particular, include transport-type structures but are not restricted to them. This level of generality is a key ingredient in allowing us to analyse more general classes of SPDEs later in the paper. The presentation is self-contained and does not rely on auxiliary results. In Section \ref{sect:vorticityeqn} we study the stochastic vorticity equation and establish existence and uniqueness of solutions using a fixed point argument. We also analyse the properties of the nonlinear term and the fractional noise term. In Section \ref{sect:hurst} we investigate statistical properties of the solution and develop a methodology for estimating the Hurst parameter based on quadratic variations. In Section \ref{sect:genproofs} we extend the fixed point argument to a more general class of stochastic evolution equations. Finally, in Section \ref{sect:applications} we present applications of our framework to several fluid models.
Appendix \ref{a:young} contains an alternative proof based on rough paths techniques, included for completeness.

\vspace{4mm}
\noindent\emph{Contributions of the paper}
\vspace{3mm}

The present work develops an analytical framework for transport-type stochastic perturbations driven by fractional Brownian motion with Hurst parameter $H>1/2$. We establish well-posedness for the two-dimensional stochastic vorticity equation with fractional noise by combining a fixed point argument with the analysis of the nonlinear transport structure. A key technical ingredient is the introduction of a version of the sewing lemma adapted to gradient-type integrands, which allows for a rigorous construction of the stochastic integral appearing in the equation. We further analyse quadratic functionals of the solution and develop a methodology for estimating the Hurst parameter from the dynamics, thereby linking statistical properties of the flow to the roughness
of the driving noise. Our results connect stochastic fluid models with turbulence-inspired scaling laws, based on the relation between Kolmogorov-type scaling and fractional noise models.

\section{Preliminaries and notations}\label{sect:preliminaries}
We consider a scale of Banach spaces $(\cB_\alpha)_{\alpha\in\R}$ where $\alpha$ indicates the spatial regularity.~In section~\ref{sect:sewinglemma} we work with the fractional power spaces corresponding to the Laplace operator, namely $\cB_\alpha=D(\Delta^\alpha)$.~In particular, 
we will work with the following functional spaces:
\begin{itemize}
\item $\mathcal{B}_{\alpha }=W^{2\alpha ,2}(\mathbb{T}^{2})=H^{2\alpha}(\mathbb{T}^{2})$ with the norm $%
\left\vert \left\vert x\right\vert \right\vert _{\alpha }:=\left\vert
\left\vert x\right\vert \right\vert _{H^{2\alpha}(\mathbb{T}^{2})}$, where $W^{2\alpha,2}(\mathbb{T}^2)=H^{2\alpha}(\mathbb{T}^2)$ is the standard Sobolev space on the two-dimensional torus $\mathbb{T}^2$. We denote the dual of $\mathcal{B}_{\alpha}$ by $\cB^*_{\alpha}$.
We consider $\mathcal{B}_{\infty}:=\bigcap_{\alpha \geq 0} W^{2 \alpha,2}\left(\mathbb{T}^2\right)=C^{\infty}\left(\mathbb{T}^2\right)$. 
We also use $\cB_{\alpha-1/2}=H^{2\alpha-1}(\T^2)$.

\item $C([0,T],\mathcal{B}_{\alpha })$ endowed with the norm $\left\vert
\left\vert \cdot \right\vert \right\vert _{0,\alpha }$. That is, for $y\in
C([0,T],\mathcal{B}_{\alpha })$, we have that 
\begin{equation*}
\left\vert \left\vert y\right\vert \right\vert _{0,\alpha }=\sup_{t\in \left[
0,T\right] }\left\vert \left\vert y_{s}\right\vert \right\vert _{\alpha }.
\end{equation*}

\item $C^{\gamma }([0,T],\mathcal{B}_{\beta })$ endowed with the norm $%
\left\vert \left\vert \cdot \right\vert \right\vert _{\gamma ,\beta }$. That
is, for $y\in C^{\gamma }([0,T],\mathcal{B}_{\beta })$, we have that 
\begin{equation*}
\left\vert \left\vert y\right\vert \right\vert _{\gamma ,\beta
}=\sup_{s,t\in \left[ 0,T\right] }\frac{\left\vert \left\vert
y_{t}-y_{s}\right\vert \right\vert _{\beta }}{\left\vert t-s\right\vert
^{\gamma }}=\sup_{s,t\in \left[ 0,T\right] }\frac{\left\vert \left\vert
y_{ts}\right\vert \right\vert _{\beta }}{\left\vert t-s\right\vert ^{\gamma }%
}
\end{equation*}
using the notation $y_{st}=y_{t}-y_{s}$. It follows that 
\begin{equation*}
\left\vert \left\vert y_{t}-y_{s}\right\vert \right\vert _{\beta }\leq
\left\vert \left\vert y\right\vert \right\vert _{\gamma ,\beta }\left\vert
t-s\right\vert ^{\gamma }.
\end{equation*}

\item We introduce the space  $V^T:=C([0,T],\mathcal{B}_{\alpha })\cap C^{\gamma }([0,T],%
\mathcal{B}_{\alpha-\gamma })$ endowed with the norm%
\begin{equation*}
y\in V^T,~~~\left\vert \left\vert y\right\vert \right\vert _{V^T}=\max \left\{
\left\vert \left\vert y\right\vert \right\vert _{0,\alpha },\left\vert
\left\vert y\right\vert \right\vert _{\gamma ,\alpha-\gamma }\right\} .
\end{equation*}
\item We use the notation $A\lesssim B$ if there exists a constant $C>0$ such that $A\leq C B$. 
\end{itemize}
{\bf Smoothing properties of analytic semigroups}

We denote by $(S_t)_{t\geq 0}$ the analytic semigroup generated by the Laplace operator. It is well-known that we can view the semigroup $(S_t)_{t\geq 0}$ as a linear mapping between the spaces $\cB_\alpha$. We consider $\sigma\in[0,1]$ and $S:[0,T]\to \cL(\cB_{\alpha},\cB_{ \alpha+\sigma})$. Then we can obtain the following standard bounds for the corresponding operator norms (\cite{Amann2,Pazy}):
\begin{align}
            & \|S_t x\|_{\cB_\alpha} \lesssim \|x\|_{\cB_\alpha}\label{semigroup} \\
     &\|S_tx\|_{\cB_{\alpha+\sigma}}\lesssim t^{-\sigma}\|x\|_{\cB_{\alpha}}\label{hg:2}\\
           & \|(S_t-I) x\|_{\cB_{\alpha}}\lesssim t^\sigma \|x\|_{\cB_{\alpha+\sigma}}.\label{hg:1}
         \end{align}
We further use the following properties of fractional Sobolev spaces.
\begin{lemma}\label{sobolev} 
  \begin{itemize}
      \item Let $s>\frac{d}{2}$. Then $H^{s}(\T^d)$ is an algebra, meaning that if $f\in H^s(\T^d)$ and $g\in H^s(\T^d)$ the product $fg\in H^s(\T^d)$ and
\begin{equation*}
\Vert fg\Vert _{H^{s}}\lesssim \Vert f\Vert _{H^{s}}\Vert g\Vert _{H^{s}}.
\end{equation*}
\item Let $s,r,t>0$ such that $s+r>t+\frac{d}{2}$. If $f\in H^s(\T^d)$, $g\in H^r(\T^d)$, then $fg\in H^t(\T^d)$ and 
\begin{align}\label{sobolev:product}
    \| fg\|_{H^t} \lesssim \|f\|_{H^s} \|g\|_{H^r}.
\end{align}
  \end{itemize}
\end{lemma}

An immediate consequence of \eqref{sobolev:product} reads as 
\begin{corollary}
    Let $s>\frac{d}{2}$ and $f\in H^s(\T^d)$ and $g\in H^{s+1}(\T^{d})$. Then $f\cdot \nabla g\in H^s(\T^d)$ and
    \[ \| f\cdot \nabla g \|_{H^s} \lesssim C \|f\|_{H^s} \| g\|_{H^{s+1}}. \]
\end{corollary}

\subsection{Definitions of solutions}

We consider the two-dimensional incompressible vorticity equation perturbed by transport-type fractional Brownian noise
\begin{equation}\label{eq2}
d\omega_t+u_{t}\cdot \nabla \omega_tdt+\displaystyle
\mathcal{L}_{\xi}\omega_t dW_{t}^{H}=\Delta\omega_tdt
\end{equation}
with initial condition $\omega _{0}$, where $u$ is the velocity of the incompressible fluid, meaning that $\nabla \cdot u = 0$.~Furthermore, $\omega_t=curl\ u_t$ is the corresponding fluid vorticity, $\xi$ is a time-independent divergence-free vector field, the operator $\mathcal{L}$ is given by $\mathcal{L}_{\xi}\omega_t := \xi \cdot \nabla \omega_t$, and $W^H$ is a fractional Brownian motion with Hurst parameter $H >1/2$.

\begin{definition}
Let $\xi\in \cB_\infty$.~We call a process $\omega\in C([0,T];\cB_\alpha)$ which satisfies  the variation of constants formula 
\begin{align*}
    \omega_t = S_{t}\omega_0 - \int_0^t S_{t-s}(u_s\cdot \nabla \omega_s)ds - \int_0^t S_{t-s}(\xi \cdot \nabla \omega_s) dW^H_s
\end{align*}
a \emph{mild solution} for~\eqref{eq2}.
\end{definition}

\begin{definition}
Let $\xi \in \mathcal{B}_\infty$. We say that a process
$\omega \in C([0,T];\mathcal{B}_\alpha)
$
is a \emph{weak solution} to~\eqref{eq2} if for every test function
$\varphi \in \mathcal{B}_\infty = C^\infty(\mathbb{T}^2)$ and every $t \in [0,T]$, it holds almost surely that
\begin{align*}
\langle \omega_t,\varphi\rangle
&= \langle \omega_0,\varphi\rangle + \int_0^t\langle \omega_s, \Delta \varphi\rangle -\int_0^t \langle \omega_s, (u_s \cdot \nabla)\varphi\rangle ds - \int_0^t \langle \omega_s, (\xi \cdot \nabla)\varphi\rangle dW^H_s
\end{align*}
where $\langle \cdot,\cdot\rangle$ denotes the duality between distributions and smooth test functions on $\mathbb{T}^2$.
\end{definition}

\begin{remark}
    The techniques developed in this work are applicable to the case when we have finitely many vector fields $\xi$ and finitely many independent fractional Brownian motions. 
\end{remark}

\subsection{Stochastic fluid equations in a general setting}\label{sect:general}

Although we treat in detail the two-dimensional vorticity case, we provide below a general framework under which our methodology holds, provided certain conditions are fulfilled.  This includes a general form of the fluid equation, as well as the assumptions one needs to impose on all linear and nonlinear operators in order for all procedures to apply.~We consider the evolution equation%
\begin{equation}
d\omega _{t}+\left( \mathcal{D+E}\right) \omega _{t}dt+\mathcal{F}\omega
_{t}dW_{t}^{H}=0  \label{general}
\end{equation}
under the following assumptions on the coefficients.

\begin{assumptions}{\em (Differential operator $\mathcal{E}$)}\label{linear}
We consider a family of 
interpolation spaces endowed with the norms $(\|\cdot\|_{\alpha})_{\alpha\in \R}$, such that $\cB_\beta \hookrightarrow \cB_\alpha$ for $\alpha<\beta$ and the following interpolation inequality holds
\begin{align}\label{interpolation:ineq}
	\|x\|^{\alpha_3-\alpha_1}_{\alpha_2} \lesssim \|x\|^{\alpha_3-\alpha_2}_{\alpha_1} \|x\|^{\alpha_2-\alpha_1}_{\alpha_3},
\end{align}
for $\alpha_1\leq \alpha_2\leq \alpha_3$ and $x\in  \cB_{\alpha_3}$.
We assume that $\mathcal{E}$ generates an analytic semigroup on $\cB_\alpha$.~In particular, it is well-known that the estimates~\eqref{semigroup}--\eqref{hg:1} are valid in this general case, see~\cite{Amann2}.

\end{assumptions}

For similar smoothing properties for non-autonomous or quasilinear parabolic evolution equations, we refer to~\cite{GHN,Tim,AN}. 

\begin{assumptions}{\em (Nonlinear drift)}\label{ass:d}
We assume that $\mathcal{D}:\cB_\alpha\to \cB_{\alpha-\beta}$ for $\beta\in[0,1)$ and there exist constants $C,p, q \geq 0$ such that 
\begin{eqnarray*}
\left\vert \left\vert \mathcal{D\omega }\right\vert \right\vert _{\mathcal{B}%
_{\alpha -\beta }} &\leq &C\left\vert \left\vert \mathcal{\omega }%
\right\vert \right\vert _{\mathcal{B}_{\alpha }}^{q} \\
\left\vert \left\vert \mathcal{D\omega }^{1}\mathcal{-D\omega }%
^{2}\right\vert \right\vert _{\mathcal{B}_{\alpha -\beta }} &\leq &C\max
\left( \left\vert \left\vert \mathcal{\omega }^{1}\right\vert \right\vert _{%
\mathcal{B}_{\alpha }}^{p},\left\vert \left\vert \mathcal{\omega }%
^{2}\right\vert \right\vert _{\mathcal{B}_{\alpha }}^{p}\right) \left\vert
\left\vert \mathcal{\omega }^{1}\mathcal{-\omega }^{2}\right\vert
\right\vert _{\mathcal{B}_{\alpha }}.
\end{eqnarray*}
\end{assumptions}

\begin{assumptions}{\em (Nonlinear diffusion coefficient)}\label{ass:f}
We assume that there exists $\theta\in[0,\gamma)$ such that $\theta + \gamma \in [0,\alpha)$ and $\mathcal{F}:\cB_\alpha\to \cB_{\alpha-\theta}$, $\mathcal{F}:\cB_{\alpha-\gamma}\to \cB_{\alpha-\gamma-\theta}$ is twice Fr\'echet differentiable with bounded derivatives. In particular,  
\begin{eqnarray*}
\left\vert \left\vert \mathcal{F\omega }\right\vert \right\vert _{\mathcal{B}%
_{\alpha -\theta }} &\leq &C\left\vert \left\vert \mathcal{\omega }%
\right\vert \right\vert _{\mathcal{B}_{\alpha }} \\
\left\vert \left\vert \mathcal{F\omega }^{1}\mathcal{-F\omega }%
^{2}\right\vert \right\vert _{\mathcal{B}_{\alpha -\theta} } &\leq &C\left\vert
\left\vert \mathcal{\omega }^{1}\mathcal{-\omega }^{2}\right\vert
\right\vert _{\mathcal{B}_{\alpha }}.
\end{eqnarray*}
\end{assumptions}
We provide the proofs for this general case in section \ref{sect:genproofs} below.

\section{Sewing lemma revisited}\label{sect:sewinglemma}

Since the stochastic convolution improves the
spatial regularity by an amount $\sigma <H$, with $H>\tfrac{1}{2}$ (see Corollary~\eqref{i:reg}), it is possible to analyse the transport-type noise term using a Young integration
approach. This can be constructed by similar techniques to the rough noise case considered in \cite{GH,GHN,GubinelliTindel,HN19}. We present a self-contained proof of this construction, yielding optimal bounds on the integral and refer to Appendix \ref{a:young} for a more rough path flavoured argument.~Moreover, based on this construction we can directly solve the equation \eqref{eq:mainStrat} in $V^T$ and not in a larger space and then use regularizing properties of analytic semigroups to conclude that this indeed belongs to $V^T$, as frequently done in rough path theory \cite{GH,GHN,AN}.  \\
 
In the following, we fix an arbitrary process 
\begin{equation*}
Y\in C\bigl(\lbrack 0,T],\mathcal{B}_{\alpha -\tfrac{1}{2}}\bigr)\cap
C^{\gamma }\bigl(\lbrack 0,T],\mathcal{B}_{\alpha -\gamma -\tfrac{1}{2}}%
\bigr),
\end{equation*}%
and exploit the spatial smoothing properties of the semigroup $%
(S_{t})_{t\geq 0}$ to show that the process $I=\left\{ I_{t},t\in \left[ 0,T%
\right] \right\} $ 
\begin{equation*}
I_{t}:=\int_{0}^{t}S_{t-r}Y_{r}\,\mathrm{d}W_{r}^{H}
\end{equation*}%
is well defined and belongs to the space 
\begin{equation*}
C\bigl(\lbrack 0,T],\mathcal{B}_{\alpha }\bigr)\cap C^{\gamma }\bigl(\lbrack
0,T],\mathcal{B}_{\alpha -\gamma }\bigr).
\end{equation*}%
This abstract result applies in particular to the process $\xi \cdot \nabla
\omega $, which satisfies the above regularity assumptions whenever $\xi\in\cB_\infty$ and
\begin{equation*}
\omega \in C\bigl(\lbrack 0,T],\mathcal{B}_{\alpha }\bigr)\cap C^{\gamma }%
\bigl(\lbrack 0,T],\mathcal{B}_{\alpha -\gamma }\bigr).
\end{equation*}

A key advantage of this approach is its flexibility. By relying only on
Young integration combined with semigroup smoothing, it allows us to treat a
much broader class of equations, not necessarily restricted to
transport-type operators. The assumptions are formulated directly in terms
of time regularity and spatial smoothing, making the method robust and
straightforward to verify (see Assumption \ref{ass:f} for general condition of the noise
operators).

By contrast, an approach based on rough path theory would impose
substantially stronger structural and regularity requirements. In that
framework, the operator appearing in the stochastic integral would need to
satisfy additional smoothness assumptions in to control the remainder terms
arising from discretizations of the equation and coupling the construction
of the stochastic integral with the solution theory of the vorticity
equation itself, following the classical Davie method. These additional
technical requirements restrict the class of admissible equations and
obscure the central role played here by the smoothing properties of the
semigroup.

To define $I_{t}$, we introduce the following discrete version of the
integral 
\begin{equation*}
I_{t}^{k}=\sum_{\left[ \frac{n}{2^{k}},\frac{n+1}{2^{k}}\right] \subset
\lbrack 0,t]}S_{t-\frac{n}{2^{k}}}Y_{\frac{n}{2^{k}}}\left( W_{\frac{n+1}{%
2^{k}}}^{H}-W_{\frac{n}{2^{k}}}^{H}\right) ,~~k\geq 0\text{ }.
\end{equation*}

\begin{theorem}[Young integral]\label{s}
Let $Y\in C\left( [0,T],\mathcal{B}_{\alpha -\frac{1}{2}}\right) \cap
C^{\gamma }\left( [0,T],\mathcal{B}_{\alpha -\gamma -\frac{1}{2}}\right) $.
Then the sequence of processes $\left( I_{t}^{k}\right) _{k}$ is Cauchy in
the space $\mathcal{B}_{\alpha }$ and we define the integral 
\begin{equation*}
I_{t}=\int_{0}^{t}S_{t-r}Y_{r}~{\mathnormal{d}}W_{r}^{H}:=\lim_{k\rightarrow
\infty }I_{t}^{k}
\end{equation*}%
to be the limit of the sequence. Moreover $I\in C\left( [0,T],\mathcal{B}%
_{\alpha }\right) \cap C^{\gamma }\left( [0,T],\mathcal{B}_{\alpha -\gamma
}\right) $.
\end{theorem}

\noindent\textbf{Proof. }

\noindent \textbf{Step 1. Convergence of the }$\left( I_{t}^{k}\right) _{k}$%
\textbf{\ sequence in the space }$\cB_{\alpha }$\textbf{.}

Let start by analysing the first term of the sequence $\left( I^{k}\right)
_{k}$. Note that the index $k$ has to be sufficiently high so that we have at
least one dyadic interval $\left[ 0,\frac{1}{2^{k_{0}}}\right] \ $in $\left[
0,t\right] $ so $k\geq k_{0}$, where $\frac{1}{2^{k_{0}}}\leq t$. \ Of
course if $t\geq 1,$ then~$k_{0}=0$. The first term of the sequence is therefore $I^{k_{0}}$. Observe that, for $0\leq t\leq T$, we have%
\begin{equation*}
\begin{aligned}
\Vert I_{t}^{k_{0}}\Vert _{\mathcal{B}_{\alpha }} &=\Vert S_{t}Y_{0}\left(
W_{\frac{1}{2^{k_{0}}}}^{H}-W_{0}^{H}\right) \Vert _{\mathcal{B}_{\alpha }}
\leq K_{W}^{\gamma }\left( \frac{1}{2^{k_{0}}}\right) ^{\gamma }\Vert
S_{t}Y_{0}\Vert _{\mathcal{B}_{\alpha }} \\
&\leq CK_{W}^{\gamma }\left( \frac{1}{2^{k_{0}}}\right) ^{\gamma -\frac{1}{2%
}}\Vert Y_{0}\Vert _{\mathcal{B}_{\alpha -\frac{1}{2}}} \leq CK_{W}^{\gamma }T^{\gamma -\frac{1}{2}}\Vert Y_{0}\Vert _{\mathcal{B}%
_{\alpha -\frac{1}{2}}}.
\end{aligned}
\end{equation*}%
So indeed, $I_{t}^{k_{0}}\in \mathcal{B}_{\alpha }$. We prove next that $%
\left( I_{t}^{k}\right) _{k}$ is a Cauchy sequence in $\mathcal{B}_{\alpha }$%
-norm by controlling the difference between consecutive terms of the
sequence. We have that 
\begin{equation}
I_{t}^{k}-I_{t}^{k+1}=\sum_{\left[ \frac{n}{2^{k}},\frac{n+1}{2^{k}}\right]
\subset \lbrack 0,t]}a_{n}^{k}+R_{t}^{k+1},  \label{representation}
\end{equation}%
where 
\begin{equation*}
a_{n}^{k}:=\left( S_{t-\frac{2n}{2^{k+1}}}Y_{\frac{2n}{2^{n+1}}}-S_{t-\frac{%
\left( 2n+1\right) }{2^{k+1}}}Y_{\frac{\left( 2n+1\right) }{2^{k+1}}}\right)
\left( W_{\frac{2\left( n+1\right) }{2^{k+1}}}^{H}-W_{\frac{\left(
2n+1\right) }{2^{k+1}}}^{H}\right) .
\end{equation*}%
and $R_{t}^{k+1}$ is a term that only appears in the sum when that last
interval (closest to $t$) cannot be paired with one of the previous
intervals. That is 
\begin{equation*}
R_{t}^{k+1}=\left\{ 
\begin{array}{cc}
0 & if~~\left[ t2^{k+1}\right] ~~is~even \\ 
S_{t-\frac{\left[ t2^{k+1}\right] -1}{2^{k+1}}}Y_{\frac{\left[ t2^{k+1}%
\right] -1}{2^{k+1}}}\left( W_{\frac{\left[ t2^{k+1}\right] }{2^{k}}}^{H}-W_{%
\frac{\left[ t2^{k+1}\right] -1}{2^{k}}}^{H}\right) & if~~\left[ t2^{k+1}%
\right] ~~is~odd%
\end{array}%
\right. .
\end{equation*}%
and therefore 
\begin{eqnarray*}
\left\vert \left\vert R_{t}^{k+1}\right\vert \right\vert _{\mathcal{B}%
_{\alpha }} &\leq &C\left( t-\frac{\left[ t2^{k+1}\right] -1}{2^{k+1}}%
\right) ^{-\frac{1}{2}}K_{W}^{\gamma }\left\vert \left\vert Y\right\vert
\right\vert _{0,\alpha -\frac{1}{2}}\left( \frac{1}{2^{k+1}}\right) ^{\gamma
} \\
&\leq &C\left( \frac{\{t2^{k+1}\}+1}{2^{k+1}}\right) ^{-\frac{1}{2}%
}K_{W}^{\gamma }\left\vert \left\vert Y\right\vert \right\vert _{0,\alpha -%
\frac{1}{2}}\left( \frac{1}{2^{k+1}}\right) ^{\gamma } \\
&\leq &CK_{W}^{\gamma }\left\vert \left\vert Y\right\vert \right\vert
_{0,\alpha -\frac{1}{2}}\left( \frac{1}{2^{k+1}}\right) ^{\gamma -\frac{1}{2}%
} \\
&\leq &CK_{W}^{\gamma }\left\vert \left\vert Y\right\vert \right\vert
_{0,\alpha -\frac{1}{2}}\left( \frac{1}{2^{k+1}}\right) ^{\gamma -\frac{1}{2}%
} \\
&\leq &CK_{W}^{\gamma }\left\vert \left\vert Y\right\vert \right\vert
_{0,\alpha -\frac{1}{2}}t^{\gamma -\frac{1}{2}}\left( \frac{1}{2^{k+1-k_{0}}}%
\right) ^{\gamma -\frac{1}{2}}.
\end{eqnarray*}%
Note that 
\begin{equation*}
a_{n}^{k}=b_{n}^{k}+c_{n}^{k},
\end{equation*}%
where 
\begin{eqnarray*}
b_{n}^{k} &:&=\left( S_{t-\frac{2n}{2^{k+1}}}-S_{t-\frac{\left( 2n+1\right) 
}{2^{k+1}}}\right) Y_{\frac{2n}{2^{k+1}}}\left( W_{\frac{2\left( n+1\right) 
}{2^{k+1}}}^{H}-W_{\frac{\left( 2n+1\right) }{2^{k+1}}}^{H}\right) \\
c_{n}^{k} &:&=S_{t-\frac{2n+1}{2^{k+1}}}\left( Y_{\frac{2n}{2^{k+1}}}-Y_{%
\frac{\left( 2n+1\right) }{2^{k+1}}}\right) \left( W_{\frac{2\left(
n+1\right) }{2^{k+1}}}^{H}-W_{\frac{\left( 2n+1\right) }{2^{k+1}}%
}^{H}\right) .
\end{eqnarray*}%
We have 
\begin{eqnarray*}
\Vert c_{n}^{k}\Vert _{\mathcal{B}_{\alpha }} &=&\left\Vert S_{t-\frac{%
\left( 2n+1\right) }{2^{k+1}}}\left( Y_{\frac{n}{2^{k}}}-Y_{\frac{\left(
2n+1\right) }{2^{k+1}}}\right) \right\Vert _{\mathcal{B}_{\alpha
}}K_{W}^{\gamma }\frac{1}{2^{(k+1)\gamma }} \\
&\leq &\left( t-\frac{\left( 2n+1\right) }{2^{k+1}}\right) ^{-\frac{1}{2}%
-\gamma }\left\Vert Y_{\frac{n}{2^{k}}}-Y_{\frac{\left( 2n+1\right) }{2^{k+1}%
}}\right\Vert _{\alpha -\gamma -\frac{1}{2}}K_{W}^{\gamma }\frac{1}{%
2^{(k+1)\gamma }} \\
&\leq &\left( t-\frac{\left( 2n+1\right) }{2^{k+1}}\right) ^{-\frac{1}{2}%
-\gamma }K_{W}^{\gamma }\left\vert \left\vert Y\right\vert \right\vert
_{\alpha ,\alpha -\gamma -\frac{1}{2}}\frac{1}{2^{(k+1)\gamma }}\frac{1}{%
2^{(k+1)\gamma }} \\
&\leq &\left( t-\frac{\left( 2n+1\right) }{2^{k+1}}\right) ^{-\frac{1}{2}%
-\gamma +\left( \gamma -\frac{1}{2}+\epsilon \right) }K_{W}^{\gamma
}\left\vert \left\vert Y\right\vert \right\vert _{\alpha ,\alpha -\gamma -%
\frac{1}{2}}\frac{1}{2^{2\left( k+1\right) \gamma }}\frac{1}{2^{-\left(
k+1\right) \left( \gamma -\frac{1}{2}+\epsilon \right) }} \\
&=&\left( t-\frac{\left( 2n+1\right) }{2^{k+1}}\right) ^{-\frac{1}{2}-\gamma
+\left( \gamma -\frac{1}{2}+\epsilon \right) }K_{W}^{\gamma }\left\vert
\left\vert Y\right\vert \right\vert _{\alpha ,\alpha -\gamma -\frac{1}{2}}%
\frac{1}{2^{\left( k+1\right) \left( \gamma +\frac{1}{2}-\epsilon \right) }},
\end{eqnarray*}%
where we chose $0<\epsilon <\gamma -\frac{1}{2}$ and used the fact that 
\begin{equation*}
\frac{\left( t-\frac{\left( 2n+1\right) }{2^{k+1}}\right) ^{\left( \gamma -%
\frac{1}{2}+\epsilon \right) }}{\frac{1}{2^{\left( k+1\right) \left( \gamma -%
\frac{1}{2}+\epsilon \right) }}}\geq 1.
\end{equation*}%
Summing up over $n$ we obtain%
\begin{eqnarray*}
\sum_{\left[ t_{n}^{k},t_{n+1}^{k}\right] \subset \left[ 0,t\right] }\Vert
c_{n}^{k}\Vert _{\mathcal{B}_{\alpha }} &\leq &K_{W}^{\gamma }\left\vert
\left\vert Y\right\vert \right\vert _{\alpha ,\alpha -\gamma -\frac{1}{2}}%
\frac{1}{2^{\left( k+1\right) \left( \gamma +\frac{1}{2}-\epsilon \right) -k}%
}\sum_{n=0}^{\left[ t2^{k}\right] -1}\left( t-\frac{\left( 2n+1\right) }{%
2^{k+1}}\right) ^{-1+\epsilon }\frac{1}{2^{k}} \\
&\leq &\frac{K_{W}^{\gamma }\left\vert \left\vert Y\right\vert \right\vert
_{\alpha ,\alpha -\gamma -\frac{1}{2}}}{2^{\gamma +\frac{1}{2}-\epsilon }}%
\left( \frac{1}{2^{\gamma +\frac{1}{2}-1-\epsilon }}\right)
^{k}\int_{0}^{t}\left( t-u\right) ^{-1+\epsilon }du \\
&\leq &t^{\epsilon }\frac{K_{W}^{\gamma }\left\vert \left\vert Y\right\vert
\right\vert _{\alpha ,\alpha -\gamma -\frac{1}{2}}}{\epsilon 2^{\gamma +%
\frac{1}{2}-\epsilon }}\left( \frac{1}{2^{\gamma -\frac{1}{2}-\epsilon }}%
\right) ^{k}.
\end{eqnarray*}%
Now we estimate $\Vert b_{n}^{k}\Vert _{\alpha }$. We have, similarly with
the computations above: 
\begin{eqnarray*}
\Vert b_{n}^{k}\Vert _{\mathcal{B}_{\alpha }} &\leq &\left\Vert \left( S_{t-%
\frac{n}{2^{k}}}-S_{t-\frac{2n+1}{2^{k+1}}}\right) Y_{\frac{n}{2^{k}}%
}\right\Vert _{\mathcal{B}_{\alpha }}K_{W}^{\gamma }\frac{1}{2^{\left(
k+1\right) \gamma }} \\
&\leq &\left( \frac{1}{2^{k+1}}\right) ^{\frac{1}{2}}\left( t-\frac{\left(
2n+1\right) }{2^{k+1}}\right) ^{-1}\Vert Y\Vert _{0,\alpha -\frac{1}{2}%
}K_{W}^{\gamma }\frac{1}{2^{\left( k+1\right) \gamma }} \\
&\leq &\Vert Y\Vert _{0,\alpha -\frac{1}{2}}K_{W}^{\gamma }\left( t-\frac{%
\left( 2n+1\right) }{2^{k+1}}\right) ^{-1+\epsilon }\frac{1}{2^{\left(
k+1\right) \left( \gamma +\frac{1}{2}-\epsilon \right) }}
\end{eqnarray*}%
from which we deduce that 
\begin{equation*}
\sum_{n=0}^{\left[ t2^{k}\right] -1}\Vert b_{n}^{k}\Vert _{\mathcal{B}%
_{\alpha }}\leq t^{\epsilon }\frac{\Vert Y\Vert _{0,\alpha -\frac{1}{2}%
}K_{W}^{\gamma }}{\epsilon 2^{\gamma +\frac{1}{2}-\epsilon }}\left( \frac{1}{%
2^{\gamma -\frac{1}{2}-\epsilon }}\right) ^{k}.
\end{equation*}%
It follows that 
\begin{eqnarray*}
\Vert I_{t}^{k}-I_{t}^{k+1}\Vert _{\alpha } &\leq &\sum_{n=0}^{\left[ t2^{k}%
\right] -1}\Vert b_{n}^{k}\Vert _{\alpha }+\Vert c_{n}^{k}\Vert _{\alpha
}+\Vert R_{n}^{k}\Vert _{\alpha } \\
&\leq &t^{\epsilon }\frac{\left( \Vert Y\Vert _{0,\alpha -\frac{1}{2}%
}+\left\vert \left\vert Y\right\vert \right\vert _{\alpha ,\alpha -\gamma -%
\frac{1}{2}}\right) K_{W}^{\gamma }}{\epsilon 2^{\gamma +\frac{1}{2}%
-\epsilon }}\left( \frac{1}{2^{\gamma -\frac{1}{2}-\epsilon }}\right) ^{k} \\
&&+CK_{W}^{\gamma }\left\vert \left\vert Y\right\vert \right\vert _{0,\alpha
-\frac{1}{2}}\left( \frac{1}{2^{k+1}}\right) ^{\gamma -\frac{1}{2}}.
\end{eqnarray*}%
Hence the sequence $\left( I_{t}^{k}\right) $ converges in the $%
\Vert \cdot \Vert _{\alpha }$ norm and 
\begin{eqnarray}
\left\vert \left\vert \lim_{k\rightarrow \infty }I_{t}^{k}\right\vert
\right\vert _{\mathcal{B}_{\alpha }} &\leq &\left\vert \left\vert
I_{t}^{k_{0}}\right\vert \right\vert _{\mathcal{B}_{\alpha
}}+\sum_{k=k_{0}}^{\infty }\Vert I_{t}^{k}-I_{t}^{k+1}\Vert _{\mathcal{B}%
_{\alpha }}  \notag \\
&\leq &CK_{W}^{\gamma }\left\vert \left\vert Y\right\vert \right\vert
_{0,\alpha -\frac{1}{2}}t^{\gamma -\frac{1}{2}}+t^{\epsilon }\frac{\left(
\Vert Y\Vert _{0,\alpha -\frac{1}{2}}+\left\vert \left\vert Y\right\vert
\right\vert _{\alpha ,\alpha -\gamma -\frac{1}{2}}\right) K_{W}^{\gamma }}{%
\epsilon 2^{\gamma +\frac{1}{2}-\epsilon }}\sum_{k\geq k_{0}}^{\infty
}\left( \frac{1}{2^{\gamma -\frac{1}{2}-\epsilon }}\right) ^{k}  \notag \\
&&+CK_{W}^{\gamma }\left\vert \left\vert Y\right\vert \right\vert _{0,\alpha
-\frac{1}{2}}\left( \frac{1}{2^{k+1}}\right) ^{\gamma -\frac{1}{2}}  \notag
\\
&\leq &C\left( \Vert Y\Vert _{0,\alpha -\frac{1}{2}}+\left\vert \left\vert
Y\right\vert \right\vert _{\alpha ,\alpha -\gamma -\frac{1}{2}}\right)
K_{W}^{\gamma }t^{\gamma -\frac{1}{2}}  \label{impbound}
\end{eqnarray}%
as 
\begin{eqnarray*}
\sum_{k\geq k_{0}}^{\infty }\left( \frac{1}{2^{\gamma -\frac{1}{2}-\epsilon }%
}\right) ^{k} &=&\left( \frac{1}{2^{\gamma -\frac{1}{2}-\epsilon }}\right)
^{k_{0}}\sum_{k\geq 0}^{\infty }\left( \frac{1}{2^{\gamma -\frac{1}{2}%
-\epsilon }}\right) ^{k}\leq \left( \frac{1}{2^{k_{0}}}\right) ^{\gamma -%
\frac{1}{2}-\epsilon }\sum_{k\geq 0}^{\infty }\left( \frac{1}{2^{\gamma -%
\frac{1}{2}-\epsilon }}\right) ^{k}\leq Ct^{\gamma -\frac{1}{2}-\epsilon } \\
\sum_{k\geq k_{0}}^{\infty }\left( \frac{1}{2^{k+1}}\right) ^{\gamma -\frac{1%
}{2}} &=&\left( \frac{1}{2}\right) ^{\gamma -\frac{1}{2}}\left( \frac{1}{%
2^{\gamma -\frac{1}{2}}}\right) ^{k_{0}}\sum_{k\geq 0}^{\infty }\left( \frac{%
1}{2^{\gamma -\frac{1}{2}}}\right) ^{k}\leq Ct^{\gamma -\frac{1}{2}}.
\end{eqnarray*}%
As a result, we can indeed define%
\begin{equation*}
I_{t}:=\int_{0}^{t}S_{t-r}Y_{r}~{\mathnormal{d}}W_{r}^{H}
\end{equation*}%
to be the limit of the sequence $I_{t}^{k}$ in $\mathcal{B}_{\alpha }$%
. Moreover, we can also obtain the following immediate extensions and
properties of the integral:

\begin{itemize}
\item For any $0\leq s\leq t\leq p\leq T$, we define the integral 
\begin{equation*}
\int_{s}^{t}S_{p-r}Y_{r}~{\mathnormal{d}}W_{r}^{H}
\end{equation*}%
by the same convergence arguments and, similar to (\ref{impbound}), we have
the following bound on the integral 

\begin{equation}
\left\vert \left\vert \int_{s}^{t}S_{p-r}Y_{r}~{\mathnormal{d}}%
W_{r}^{H}\right\vert \right\vert _{\mathcal{B}_{\alpha }}\leq C\left( \Vert
Y\Vert _{0,\alpha -\frac{1}{2}}+\left\vert \left\vert Y\right\vert
\right\vert _{\alpha ,\alpha -\gamma -\frac{1}{2}}\right) K_{W}^{\gamma
}(t-s)^{\gamma -\frac{1}{2}}.  \label{bddd13}
\end{equation}

\item Using the continuity of the semigroup $S$ we can deduce that 

\begin{equation*}
\int_{s}^{t}S_{p-r}Y_{r}~{\mathnormal{d}}W_{r}^{H}=S_{p-t}\left(
\int_{s}^{t}S_{t-r}Y_{r}~{\mathnormal{d}}W_{r}^{H}\right)
\end{equation*}

\item For any $0\leq s\leq t\leq p\leq T$, \ by the same convergence
arguments and, similar to (\ref{impbound}), \ we deduce that 
\begin{equation*}
\begin{aligned}
& \left\vert \left\vert \int_{s}^{t}S_{p-r}Y_{r}~{\mathnormal{d}}%
W_{r}^{H}-\int_{s}^{t}S_{t-r}Y_{r}~{\mathnormal{d}}W_{r}^{H}\right\vert
\right\vert _{\mathcal{B}_{\alpha }} =\left\vert \left\vert
\int_{s}^{t}\left( S_{p-t}-I\right) S_{t-r}Y_{r}~{\mathnormal{d}}%
W_{r}^{H}\right\vert \right\vert _{\mathcal{B}_{\alpha }}  \notag \\
&\leq \left( p-t\right) ^{^{\varepsilon }}\left( \Vert Y\Vert _{0,\alpha -%
\frac{1}{2}}+\left\vert \left\vert Y\right\vert \right\vert _{\alpha ,\alpha
-\gamma -\frac{1}{2}}\right) K_{W}^{\gamma }(t-s)^{\gamma -\frac{1}{2}%
-\varepsilon }  \label{bdd14}
\end{aligned}
\end{equation*}%
for any $0<\varepsilon <\gamma -\frac{1}{2}$.
\end{itemize}

Up until not we have only been concerned with the rigurous definition of the
integral $I_{t}$ and its various extensions. Next we analize the properties
of the integral as a process.

\noindent \textbf{Step 2. }$I\in C\left( [0,T],\mathcal{B}_{\alpha }\right) $%
\textbf{. }

To show this we need to control $\left\vert \left\vert
I_{t}-I_{s}\right\vert \right\vert _{\mathcal{B}_{\alpha }}$. First, observe
that for any $0\leq s\leq t$, where $s$ is a dyadic number, we have that%
\begin{equation}
\int_{0}^{t}S_{t-r}Y_{r}~{\mathnormal{d}}W_{r}^{H}=\int_{s}^{t}S_{t-r}Y_{r}~{%
\mathnormal{d}}W_{r}^{H}+\int_{0}^{s}S_{t-r}Y_{r}~{\mathnormal{d}}W_{r}^{H}
\label{adit}
\end{equation}%
Note that all the terms in (\ref{adit}) are well defined by the arguments
in the first step and the identity holds true by the same convergence
method. \ Identity (\ref{adit}) can then be extended to arbitrary
intermediate points $s\in \left( 0,t\right) $, that is $s$ is not necessarily
dyadic, by~taking a sequence $\left( s_{n}\right) $ of dyadic numbers which
converges to $s$ and then observing that 
\begin{eqnarray}
\int_{0}^{t}S_{t-r}Y_{r}~{\mathnormal{d}}W_{r}^{H} &=&\lim_{n\rightarrow
\infty }\int_{s_{n}}^{t}S_{t-r}Y_{r}~{\mathnormal{d}}W_{r}^{H}+\lim_{n%
\rightarrow \infty }\int_{0}^{s_{n}}S_{t-r}Y_{r}~{\mathnormal{d}}W_{r}^{H} 
\notag \\
&=&\int_{s}^{t}S_{t-r}Y_{r}~{\mathnormal{d}}W_{r}^{H}+%
\int_{0}^{s}S_{t-r}Y_{r}~{\mathnormal{d}}W_{r}^{H},  \label{addit2}
\end{eqnarray}%
where the two limits are justified by using the bound (\ref{bddd13}). We
note that we couldn't easily deduce (\ref{addit2}) from the earlier
calculations because it would have been harder to keep track of the
`nuisance' term $R_{t}^{k}$. We are now ready to show the continuity of the
integral in $\mathcal{B}_{\alpha }$-norm. For this observe that, from (\ref%
{addit2}), we deduce that 
\begin{eqnarray*}
\left\vert \left\vert I_{t}-I_{s}\right\vert \right\vert _{\mathcal{B}%
_{\alpha }} &\leq &\left\vert \left\vert \int_{s}^{t}S_{t-r}Y_{r}~{%
\mathnormal{d}}W_{r}^{H}\right\vert \right\vert _{\mathcal{B}_{\alpha
}}+\left\vert \left\vert \left( S_{t-s}-I\right) \left(
\int_{0}^{s}S_{t-r}Y_{r}~{\mathnormal{d}}W_{r}^{H}\right) \right\vert
\right\vert _{\mathcal{B}_{\alpha }} \\
&\leq &C\left( \Vert Y\Vert _{0,\alpha -\frac{1}{2}}+\left\vert \left\vert
Y\right\vert \right\vert _{\alpha ,\alpha -\gamma -\frac{1}{2}}\right)
\left( 1+T^{\gamma -\frac{1}{2}-\varepsilon }\right) K_{W}^{\gamma }\left(
(t-s)^{\gamma -\frac{1}{2}}+(t-s)^{\varepsilon }\right) ,
\end{eqnarray*}%
which gives us the continuity in $\mathcal{B}_{\alpha }$. In fact the
function is $\varepsilon $-H\"{o}lder continuous in $\mathcal{B}_{\alpha }$
for any $\varepsilon \in \left( 0,\gamma -\frac{1}{2}\right) $.

\textbf{Step 3 }$I\in C^{\gamma }\left( [0,T],\mathcal{B}_{\alpha -\gamma
}\right) $\textbf{.}

We move now to the $\gamma $-H\"{o}lder continuity in $\mathcal{B}_{\alpha
-\gamma }$. Again we will use (\ref{addit2}) to deduce that 
\begin{equation}
\left\vert \left\vert I_{t}-I_{s}\right\vert \right\vert _{\mathcal{B}%
_{\alpha -\gamma }}\leq \left\vert \left\vert \int_{s}^{t}S_{t-r}Y_{r}~{%
\mathnormal{d}}W_{r}^{H}\right\vert \right\vert _{\mathcal{B}_{\alpha
-\gamma }}+\left\vert \left\vert \left( S_{t-s}-I\right) \left(
\int_{0}^{s}S_{t-r}Y_{r}~{\mathnormal{d}}W_{r}^{H}\right) \right\vert
\right\vert _{\mathcal{B}_{\alpha -\gamma }}.  \label{int3}
\end{equation}%
Observe that 
\begin{eqnarray*}
\left\vert \left\vert \left( S_{t-s}-I\right) \left(
\int_{0}^{s}S_{t-r}Y_{r}~{\mathnormal{d}}W_{r}^{H}\right) \right\vert
\right\vert _{\mathcal{B}_{\alpha -\gamma }} &\leq &\left( t-s\right)
^{\gamma }\left\vert \left\vert \left( S_{t-s}-I\right) \left(
\int_{0}^{s}S_{t-r}Y_{r}~{\mathnormal{d}}W_{r}^{H}\right) \right\vert
\right\vert _{\mathcal{B}_{\alpha }} \\
&\leq &\left( t-s\right) ^{\gamma }C\left( \Vert Y\Vert _{0,\alpha -\frac{1}{%
2}}+\left\vert \left\vert Y\right\vert \right\vert _{\alpha ,\alpha -\gamma -%
\frac{1}{2}}\right) K_{W}^{\gamma }T^{\gamma -\frac{1}{2}}
\end{eqnarray*}%
which gives the correct control of the second term in (\ref{int3}). It only
remains to control the first term in (\ref{int3}). For this we need to
repeat the same calculations in Step 1, but looking at the norm in $\mathcal{%
B}_{\alpha -\gamma }$ instead of the norm in $\mathcal{B}_{\alpha }$. Just
as before, it is we will do this for $s=0$. Let us look at the first term $%
I_{t}^{k_{0}}$ of the sequence $\left( I_{t}^{k}\right) _{k}$. We have that 
\begin{eqnarray}
\Vert I_{t}^{k_{0}}\Vert _{\mathcal{B}_{\alpha -\gamma }} &=&\Vert
S_{t}Y_{0}(W_{\frac{1}{2^{k_{0}}}}^{H}-W_{0}^{H})\Vert _{\mathcal{B}_{\alpha
-\gamma }}  \notag \\
&\leq &K_{W}^{\gamma }\left( \frac{1}{2^{k_{0}}}\right) ^{\gamma }\Vert
S_{t}Y_{0}\Vert _{\mathcal{B}_{\alpha -\gamma }}  \notag \\
&\leq &CK_{W}^{\gamma }\left( \frac{1}{2^{k_{0}}}\right) ^{\gamma }  \Vert
Y_{0}\Vert _{\mathcal{B}_{\alpha -\frac{1}{2}}}  \notag \\
&\leq &CK_{W}^{\gamma }t^{\gamma }\Vert Y_{0}\Vert _{\mathcal{B}_{\alpha -%
\frac{1}{2}}}.  \label{bdd15}
\end{eqnarray}%

The control (\ref{bdd15}) is enough to deduce (after controlling the
differences $\Vert I_{t}^{k+1}-I_{t}^{k}\Vert _{\mathcal{B}_{\alpha -\gamma
}}$) that $I\in C^{\gamma }\left( [0,T],\mathcal{B}_{\alpha -\gamma }\right) 
$. Moreover, since $W^{H}$ is $(\gamma +\varepsilon )$-H\"{o}lder continuous
for any $\varepsilon <H-\gamma $, we can deduce that 
\begin{equation}
\Vert I_{t}^{k_{0}}\Vert _{\mathcal{B}_{\alpha -\gamma }}\leq CK_{W}^{\gamma
}\Vert Y_{0}\Vert _{\mathcal{B}_{\alpha -\frac{1}{2}}}t^{\gamma +\varepsilon
}\leq CK_{W}^{\gamma }T^{\varepsilon }\Vert Y_{0}\Vert _{\mathcal{B}_{\alpha
-\frac{1}{2}}}t^{\gamma }.  \label{tarziu2}
\end{equation}%

While the control (\ref{tarziu2}) is not explicitly required for the
construction of the Young integral, the fact that it belongs to the space $%
C^{\gamma }\left( [0,T],\mathcal{B}_{\alpha -\gamma }\right) $ will be
used in the construction of the solution of the vorticity equation. 
To
complete the control of in the $\mathcal{B}_{\alpha -\gamma }$-norm, we use
as above, the representation (\ref{representation}) and control all the
terms in the $\mathcal{B}_{\alpha -\gamma }$ norm. We have

\begin{eqnarray}
\left\vert \left\vert R_{t}^{k+1}\right\vert \right\vert _{\mathcal{B}%
_{\alpha -\gamma }} &\leq &CK_{W}^{\gamma }\left\vert \left\vert
Y\right\vert \right\vert _{0,\alpha -\frac{1}{2}}\left( \frac{1}{2^{k+1}}%
\right) ^{\gamma }  \notag \\
&\leq &CK_{W}^{\gamma }\left\vert \left\vert Y\right\vert \right\vert
_{0,\alpha -\frac{1}{2}}\left( \frac{1}{2^{k+1}}\right) ^{\gamma }  \notag \\
&\leq &CK_{W}^{\gamma }\left\vert \left\vert Y\right\vert \right\vert
_{0,\alpha -\frac{1}{2}}\left( \frac{1}{2^{k+1}}\right) ^{\gamma }  \notag \\
&\leq &CK_{W}^{\gamma }\left\vert \left\vert Y\right\vert \right\vert
_{0,\alpha -\frac{1}{2}}\left( \frac{1}{2^{k+1}}\right) ^{\gamma }  \notag \\
&\leq &CK_{W}^{\gamma }\left\vert \left\vert Y\right\vert \right\vert
_{0,\alpha -\frac{1}{2}}t^{\gamma }\left( \frac{1}{2^{k+1-k_{0}}}\right)
^{\gamma }.  \label{bbb1}
\end{eqnarray}%
The control (\ref{bbb1}) is enough to show that $I\in C^{\gamma
}\left( [0,T],\mathcal{B}_{\alpha -\gamma }\right) $. Moreover, since $W^{H}$
is $(\gamma +\varepsilon )$-H\"{o}lder continuous for any $\varepsilon
<H-\gamma $, we can deduce that 
\begin{equation}
\left\vert \left\vert R_{t}^{k+1}\right\vert \right\vert _{\mathcal{B}%
_{\alpha -\gamma }}\leq CK_{W}^{\gamma }\left\vert \left\vert Y\right\vert
\right\vert _{0,\alpha -\frac{1}{2}}t^{\gamma +\varepsilon }\left( \frac{1}{%
2^{k+1-k_{0}}}\right) ^{\gamma +\varepsilon }\leq CK_{W}^{\gamma }\left\vert
\left\vert Y\right\vert \right\vert _{0,\alpha -\frac{1}{2}}T^{\varepsilon
}t^{\gamma }\left( \frac{1}{2^{k+1-k_{0}}}\right) ^{\gamma +\varepsilon }
\label{tarziu3}
\end{equation}%
which is not explicitly required for the construction of the Young integral
and showing that it is in the space $C^{\gamma }\left( [0,T],\mathcal{B}%
_{\alpha -\gamma }\right) $, however it will be used in the construction of
the solution of the vorticity equation. We control next the terms $c_{n}^{k}$
and $b_{n}^{k}$. First observe that 
\begin{eqnarray*}
\Vert c_{n}^{k}\Vert _{\mathcal{B}_{\alpha -\gamma }} &=&\left\Vert S_{t-%
\frac{\left( 2n+1\right) }{2^{k+1}}}\left( Y_{\frac{n}{2^{k}}}-Y_{\frac{%
\left( 2n+1\right) }{2^{k+1}}}\right) \right\Vert _{\mathcal{B}_{\alpha
-\gamma }}K_{W}^{\gamma }\frac{1}{2^{(k+1)\gamma }} \\
&\leq &\left( t-\frac{\left( 2n+1\right) }{2^{k+1}}\right) ^{-\frac{1}{2}%
}\left\Vert Y_{\frac{n}{2^{k}}}-Y_{\frac{\left( 2n+1\right) }{2^{k+1}}%
}\right\Vert _{\alpha -\gamma -\frac{1}{2}}K_{W}^{\gamma }\frac{1}{%
2^{(k+1)\gamma }} \\
&\leq &\left( t-\frac{\left( 2n+1\right) }{2^{k+1}}\right) ^{-\frac{1}{2}%
}K_{W}^{\gamma }\left\vert \left\vert Y\right\vert \right\vert _{\alpha
,\alpha -\gamma -\frac{1}{2}}\frac{1}{2^{(k+1)\gamma }}\frac{1}{%
2^{(k+1)\gamma }} \\
&\leq &\left( t-\frac{\left( 2n+1\right) }{2^{k+1}}\right) ^{\gamma
-1}K_{W}^{\gamma }\left\vert \left\vert Y\right\vert \right\vert _{\alpha
,\alpha -\gamma -\frac{1}{2}}\frac{1}{2^{2\left( k+1\right) \gamma -\left(
k+1\right) \left( \gamma -\frac{1}{2}\right) }} \\
&\leq &K_{W}^{\gamma }\left\vert \left\vert Y\right\vert \right\vert
_{\alpha ,\alpha -\gamma -\frac{1}{2}}\left( t-\frac{\left( 2n+1\right) }{%
2^{k+1}}\right) ^{\gamma -1}\frac{1}{2^{\left( k+1\right) \left( \gamma +%
\frac{1}{2}\right) }} \\
&\leq &K_{W}^{\gamma }\left\vert \left\vert Y\right\vert \right\vert
_{\alpha ,\alpha -\gamma -\frac{1}{2}}\left( t-\frac{\left( 2n+1\right) }{%
2^{k+1}}\right) ^{\gamma -1}\frac{1}{2^{k}}\frac{1}{2^{k\left( \gamma -\frac{%
1}{2}\right) }}\frac{1}{2^{\left( \gamma +\frac{1}{2}\right) }}
\end{eqnarray*}%
where used the fact that 
\begin{equation*}
\frac{\left( t-\frac{\left( 2n+1\right) }{2^{k+1}}\right) ^{\gamma -\frac{1}{%
2}}}{\frac{1}{2^{\left( k+1\right) \left( \gamma -\frac{1}{2}\right) }}}\geq
1.
\end{equation*}%
Summing up over $n$ we obtain%
\begin{eqnarray}
\sum_{\left[ t_{n}^{k},t_{n+1}^{k}\right] \subset \left[ 0,t\right] }\Vert
c_{n}^{k}\Vert _{\mathcal{B}_{\alpha -\gamma }} &\leq &\frac{K_{W}^{\gamma
}\left\vert \left\vert Y\right\vert \right\vert _{\alpha ,\alpha -\gamma -%
\frac{1}{2}}}{2^{\left( \gamma +\frac{1}{2}\right) }}t^{\left( \gamma -\frac{%
1}{2}\right) }\frac{1}{2^{\left( k-k_{0}\right) \left( \gamma -\frac{1}{2}%
\right) }}\int_{0}^{t}\left( t-u\right) ^{\gamma -1}du  \notag \\
&\leq &\frac{K_{W}^{\gamma }\left\vert \left\vert Y\right\vert \right\vert
_{\alpha ,\alpha -\gamma -\frac{1}{2}}}{2^{\left( \gamma +\frac{1}{2}\right)
}}T^{\left( \gamma -\frac{1}{2}\right) }\frac{1}{2^{\left( k-k_{0}\right)
\left( \gamma -\frac{1}{2}\right) }}t^{\gamma }  .\label{bbb3}
\end{eqnarray}%
Now we estimate $\Vert b_{n}^{k}\Vert _{\mathcal{B}_{\alpha -\gamma }}$. We
have, similarly with the computations above: 
\begin{eqnarray*}
\Vert b_{n}^{k}\Vert _{\mathcal{B}_{\alpha -\gamma }} &\leq &\left\Vert
\left( S_{t-\frac{n}{2^{k}}}-S_{t-\frac{2n+1}{2^{k+1}}}\right) Y_{\frac{n}{%
2^{k}}}\right\Vert _{\mathcal{B}_{\alpha -\gamma }}K_{W}^{\gamma }\frac{1}{%
2^{\left( k+1\right) \gamma }} \\
&\leq &K_{W}^{\gamma }\frac{1}{2^{\left( k+1\right) \gamma }}\frac{1}{%
2^{\left( k+1\right) \gamma }}\left\Vert S_{t-\frac{2n+1}{2^{k+1}}}Y_{\frac{n%
}{2^{k}}}\right\Vert _{\mathcal{B}_{\alpha }} \\
&\leq &K_{W}^{\gamma }\frac{1}{2^{2\left( k+1\right) \gamma }}\left( t-\frac{%
\left( 2n+1\right) }{2^{k+1}}\right) ^{-\frac{1}{2}}\left\Vert Y_{\frac{n}{%
2^{k}}}\right\Vert _{\alpha -\frac{1}{2}} \\
&\leq &K_{W}^{\gamma }\left\Vert Y_{\frac{n}{2^{k}}}\right\Vert _{\alpha -%
\frac{1}{2}}\frac{1}{2^{2\left( k+1\right) \gamma -\left( k+1\right) \left(
\gamma -\frac{1}{2}\right) }}\left( t-\frac{\left( 2n+1\right) }{2^{k+1}}%
\right) ^{-\frac{1}{2}}\left( t-\frac{\left( 2n+1\right) }{2^{k+1}}\right)
^{\gamma -\frac{1}{2}}
\end{eqnarray*}%
which gives%
\begin{equation}
\sum_{n=0}^{\left[ t2^{k}\right] -1}\Vert b_{n}^{k}\Vert _{\mathcal{B}%
_{\alpha }}\leq K_{W}^{\gamma }\left\Vert Y\right\Vert _{0,\alpha -\frac{1}{2%
}}T^{\left( \gamma -\frac{1}{2}\right) }\frac{1}{2^{\left( k-k_{0}\right)
\left( \gamma -\frac{1}{2}\right) }}t^{\gamma } . \label{bbb2}
\end{equation}%
It follows that 
\begin{equation*}
\begin{aligned}
\Vert I_{t}^{k}-I_{t}^{k+1}\Vert _{\mathcal{B}_{\alpha -\gamma }} &\leq
\sum_{n=0}^{\left[ t2^{k}\right] -1}\Vert b_{n}^{k}\Vert _{\mathcal{B}%
_{\alpha -\gamma }}+\Vert c_{n}^{k}\Vert _{\mathcal{B}_{\alpha -\gamma
}}+\Vert R_{n}^{k}\Vert _{\mathcal{B}_{\alpha -\gamma }}\leq \\
& \left( \Vert
Y\Vert _{0,\alpha -\frac{1}{2}}+\left\vert \left\vert Y\right\vert
\right\vert _{\alpha ,\alpha -\gamma -\frac{1}{2}}\right) K_{W}^{\gamma
}\left( 1+T^{\left( \gamma -\frac{1}{2}\right) }\right) \frac{1}{2^{\left(
k-k_{0}\right) \left( \gamma -\frac{1}{2}\right) }}t^{\gamma }  \label{bbb4}
\end{aligned}
\end{equation*}
and, therefore, 
\begin{eqnarray}
\left\vert \left\vert \lim_{k\rightarrow \infty }I_{t}^{k}\right\vert
\right\vert _{\mathcal{B}_{\alpha -\gamma }} &\leq &\left\vert \left\vert
I_{t}^{k_{0}}\right\vert \right\vert _{\mathcal{B}_{\alpha -\gamma
}}+\sum_{k=k_{0}}^{\infty }\Vert I_{t}^{k}-I_{t}^{k+1}\Vert _{\mathcal{B}%
_{\alpha -\gamma }}  \notag \\
&\leq &CK_{W}^{\gamma }\Vert Y_{0}\Vert _{\mathcal{B}_{\alpha -\frac{1}{2}%
}}t^{\gamma }+\left( \Vert Y\Vert _{0,\alpha -\frac{1}{2}}+\left\vert
\left\vert Y\right\vert \right\vert _{\alpha ,\alpha -\gamma -\frac{1}{2}%
}\right) K_{W}^{\gamma }\left( 1+T^{\left( \gamma -\frac{1}{2}\right)
}\right) t^{\gamma }\sum_{k\geq k_{0}}^{\infty }\left( \frac{1}{2^{\gamma -%
\frac{1}{2}-\epsilon }}\right) ^{k-k_{0}}  \notag \\
&\leq &C_{T}(W^{H})\left( \Vert Y\Vert _{0,\alpha -\frac{1}{2}}+\left\vert
\left\vert Y\right\vert \right\vert _{\alpha ,\alpha -\gamma -\frac{1}{2}%
}\right) t^{\gamma },  \label{holder}
\end{eqnarray}
where $C_{T}(W^{H})=CK_{W}^{\gamma }(1+T^{\left( \gamma -\frac{1}{2}\right)
})$. This concludes the proof of the last step of the theorem.

\vspace{3mm}
\begin{remark}\label{rem:sewing} $\left. \right. $\\
1. As announced, we can immediately apply Theorem \ref{s} to the
process $Y=\xi \cdot \nabla \omega $, which satisfies the required
regularity assumptions whenever $\omega \in C\left( [0,T],\mathcal{B}%
_{\alpha }\right) \cap C^{\gamma }([0,T],\mathcal{B}_{\alpha -\gamma })$.\
The integral 
\begin{equation}
I_{t}\left( \omega \right) =\int_{0}^{t}S_{t-r}\left( \xi \cdot \nabla
\omega \right) ~{\mathnormal{d}}W_{r}^{H}  \label{Itomega}
\end{equation}%
is well defined. Moreover there exist $C_{T}(\xi ,W^{H})$ such that 
\begin{equation}
\left\vert \left\vert I\left( \omega \right) \right\vert \right\vert
_{V^{T}}\leq C_{T}^{1}(\xi ,W^{H})\left\vert \left\vert \omega \right\vert
\right\vert _{V^{T}}.  \label{Itomega2}
\end{equation}%
More generally, following the same arguments, the integral 
\begin{equation}\label{Itomega4}
I_{t}\left( \omega \right) =\int_{0}^{t}S_{t-r}\left( \mathcal{F}\left(
\omega \right) \right) ~{\mathnormal{d}}W_{r}^{H}  
\end{equation}%
is well defined, for any (possibly) nonlinear operator $\mathcal{F}$ satisfying Assumption \ref{ass:f}.

\vspace{2mm}
2. The integral defined in (\ref{Itomega}) is Lipschitz on the
space $V^{T}$. To see this apply Theorem~\ref{s} to the process $Y^{12}:=\xi
\cdot \nabla \left( \omega ^{1}-\omega ^{2}\right) $ for any two processes $%
\omega ^{1},\omega ^{2}\in V^{T}$ and deduce that 
\begin{equation}
\left\vert \left\vert I\left( \omega ^{1}\right) -I\left( \omega ^{2}\right)
\right\vert \right\vert _{V_{T}}\leq C_{T}^{1}(\xi ,W^{H})\left\vert
\left\vert \omega ^{1}-\omega ^{2}\right\vert \right\vert _{V^{T}}.
\label{Itoomega3}
\end{equation}%
The same holds true, more generally, for the integral (\ref{Itomega4})
provided 
\begin{equation}\label{cond:f}
\left\vert \left\vert \mathcal{F\omega }^{1}\mathcal{-F\omega }%
^{2}\right\vert \right\vert _{V^T}\leq
C\left\vert \left\vert \mathcal{\omega }^{1}\mathcal{-\omega }%
^{2}\right\vert \right\vert _{V^{T}}
\end{equation}%
for some $C>0$. This condition holds under the assumption that $\mathcal{F}$ is two times continuously Fr\'echet differentiable with bounded derivatives, as imposed in Assumption \ref{ass:f}.
\vspace{2mm}

3. In the next section we will be looking to find a ball $%
B_{R}(T)\subset V^{T}$ such that for any $\omega \in $ $B_{R}(T),~\
\left\vert \left\vert I\left( \omega \right) \right\vert \right\vert
_{V^{T}}\leq \frac{R}{3}$. This is possible provided the constant $%
C_{T}^{1}(\xi ,W^{H})$ appearing in (\ref{Itomega2}) can be chosen to be
less than $\frac{1}{3}$.\ For this it is sufficient to have a constant that
vanishes as $T$ decreases to $0$. As seen from the arguments in Theorem \ref{s} this holds true, but one has to exploit the additional H\"{o}lder
continuity (beyond $\gamma $) that $W^{H}$ has. In particular following from
the same arguments one can deduce that (\ref{Itomega2}) holds with 
\begin{equation}
C_{T}^{1}(\xi ,W^{H})=C\left\vert \left\vert \xi \right\vert \right\vert _{%
\mathcal{B}_{\infty}}(K_{W}^{\gamma +\varepsilon }T^{\varepsilon
}+K_{W}^{\gamma }T^{\left( \gamma -\theta \right) })  \label{Itoomega5}
\end{equation}%
where $C$ is a universal constant, $\varepsilon \in (0,H-\gamma )$, $%
K_{W}^{\gamma +\varepsilon }\,$is the $\left( \gamma +\varepsilon \right) -$H\"{o}%
lder constant of $W^{H}$ on the interval $\left[ 0,T\right] $ and $%
K_{W}^{\gamma }\,$is the $\gamma -$H\"{o}lder constant of $W^{H}$ on the
interval $\left[ 0,T\right] $. Of course (\ref{Itoomega5}), enables us to
conclude that there exists $T=T\left( R\right) $ such that for any $T\leq $ $%
T\left( R\right) $, $C_{T}^{1}(\xi ,W^{H})\leq \frac{1}{3}$ and therefore 
\begin{equation}
\left\vert \left\vert I\left( \omega \right) \right\vert \right\vert
_{V^{T}}\leq \frac{R}{3}  \label{itoomega7}
\end{equation}%
for any $\omega \in $ $B_{R}(T)$. Moreover we will be looking to show that $%
I $ has contractive properties. More precisely we will want that for any $%
\omega ^{1},\omega ^{2}\in B_{R}(T)\subset V^{T},\left\vert \left\vert
I\left( \omega ^{1}\right) -I\left( \omega ^{2}\right) \right\vert
\right\vert _{V^{T}}$ can be controlled in terms of $\left\vert \left\vert
\omega ^{1}-\omega ^{2}\right\vert \right\vert _{V^{T}}$. For the same $%
C_{T}^{1}(\xi ,W^{H})$ as in (\ref{Itoomega5}) and applying (\ref{Itoomega3}%
) we deduce that \ 
\begin{equation}
\left\vert \left\vert I\left( \omega ^{1}\right) -I\left( \omega ^{2}\right)
\right\vert \right\vert _{V^{T}}\leq \frac{1}{3}\left\vert \left\vert \omega
^{1}-\omega ^{2}\right\vert \right\vert _{V^{T}}.  \label{itoomega6}
\end{equation}
\end{remark}

\section{Stochastic vorticity equation}\label{sect:vorticityeqn}
We proceed now with the analysis of the two-dimensional incompressible vorticity equation perturbed by transport-type fractional Brownian noise
\begin{equation*}
d\omega_t+u_{t}\cdot \nabla \omega_tdt+\displaystyle
\mathcal{L}_{\xi}\omega_t dW_{t}^{H}=\Delta\omega_tdt
\end{equation*}
with initial condition $\omega _{0} \in \mathcal{B}_{\alpha}, \alpha > d/2$, where $u$ is the velocity of the incompressible fluid (so $\nabla \cdot u = 0$), $\omega_t=curl\ u_t$ is the corresponding fluid vorticity, $\xi$ is a time-independent divergence-free vector field, the operator $\mathcal{L}$ is given by $\mathcal{L}_{\xi}\omega_t := \xi \cdot \nabla \omega_t$, and $W^H$ is a fractional Brownian motion with Hurst parameter $H >1/2$.

\subsection{Fixed point argument}\label{sect:fixedpoint}
As before, for \[V^T=C([0,T],\cB_\alpha)\cap C^\gamma([0,T],\cB_{\alpha-\gamma})\] we consider the map 
\begin{equation}
\Lambda : V^T \rightarrow V^T
\end{equation}
and show that for $T$ sufficiently small, there exists $\omega$ such that
\begin{equation}
\Lambda (\omega) = \omega,
\end{equation}
where
\begin{equation}\label{eq:mainmild}
\Lambda(\omega) = S_t\omega_0 - \displaystyle\int_0^t S_{t-s}(u_s\cdot \nabla \omega_s) ds - \displaystyle\int_0^t S_{t-s}(\xi \cdot \nabla\omega_s) dW_s^H.
\end{equation}

\begin{theorem}\label{thm:mild}
Let $\xi\in \cB_\infty$.~The equation \eqref{eq:mainmild} has a unique mild solution $\omega\in V^T$ for $T$ sufficiently small. 
\end{theorem}

\begin{proof}
In order to apply Banach's fixed point argument, we have to show that $\Lambda : V^T \rightarrow V^T$ is a contraction by choosing $T$ small enough. We introduce the approximating sequence $(\omega^n)_{n\geq 1}$ given by

\begin{equation}
\begin{aligned}
&\omega_0 = x \in V^T\\
& \omega_t^n = S_t\omega_0 - \displaystyle\int_0^tS_{t-s}u_s^{n-1}\cdot \nabla\omega_s^{n-1}ds - \displaystyle\int_0^t S_{t-s}\xi \cdot \nabla\omega_s^{n-1}dW_s^H
\end{aligned}
\end{equation}
that is
\begin{equation}
\omega^n = \Lambda (\omega^{n-1})
\end{equation}
i.e the link between the integral form and the map $\Lambda$ is given by
\begin{equation}\label{lambdaeq}
(\Lambda(\omega^{n-1}))_t=S_t\omega_0 - \displaystyle\int_0^tS_{t-s}u_s^{n-1}\cdot \nabla\omega_s^{n-1}ds - \displaystyle\int_0^t S_{t-s}\xi \cdot \nabla\omega_s^{n-1}dW_s^H.
\end{equation}
Note that $\omega^0 \in \mathcal{B}_{\alpha}$ exists by hypothesis and then $u^0 := K\omega^0$ is well-defined due to the properties of the Biot-Savart kernel $K$. Then by induction $\omega^1 \in \mathcal{B}_{\alpha}$ is well-defined and so on.
Let
\begin{equation}
B_R (T):= \left\{ y\in V^T, y_0=\{t\rightarrow S_t\omega_0\}, \quad \hbox{and} \quad \|y-y_0\|_{V^T}\leq R   \right\}.
\end{equation}
Note that, from the properties of the semigroup $S_t$, the centre of the ball $y_0=\{t\rightarrow S_t\omega_0\}$ indeed belongs to $V^T$.
We show that, for $T$ sufficiently small,
\begin{equation}
  \Lambda: B_R(T) \rightarrow B_R(T)  
\end{equation}
and there exists $L < 1$ such that
\begin{equation}
\|\Lambda(\omega^{n,1}) - \Lambda(\omega^{n,2}) \|_{V^T} \leq L \|\omega^{n,1} - \omega^{n,2}\|_{V^T}, \quad  \hbox{where} \quad \omega^{n,1}, \omega^{n,2} \in B_R(T)
\end{equation}
i.e. that
\begin{eqnarray}
\left\|\Lambda\left(\omega^{n,1}\right)-\Lambda\left(\omega^{n,2}\right)\right\|_{\left.C(0, T] ; \mathcal{B}_a\right)}+\left\|\Lambda\left(\omega^{n,1}\right)-\Lambda\left(\omega^{n,2}\right)\right\|_{C^{\gamma}\left([0, T] ; \mathcal{B}_{a-\gamma}\right)} \leq L\left\|\omega^{n,1}-\omega^{n,2}\right\|_{V^T}.
\end{eqnarray}
This is shown using estimates on the terms that appear in the integral form of the equation \eqref{lambdaeq}, using properties of the semigroup, see Proposition \ref{prop:contraction} below.
Then there exists $\omega^{*}$ such that 
\begin{equation}
\omega^{*} = \displaystyle\lim_{n\rightarrow\infty}\omega^n
\end{equation}
and
\begin{equation}
\Lambda(\omega^{*}) = \omega^{*}.
\end{equation}
That is, $\omega^{*}$ is a solution for \eqref{eq:mainmild}. 

As usual, uniqueness follows from the contraction property of the mapping used in the fixed point argument. The contraction ensures that the fixed point is unique in the chosen functional space.
\end{proof}

\begin{proposition}\label{prop:contraction}
The map $\Lambda$ is a contraction on $V^T$ choosing $T$ small enough.
That is \begin{equation}\label{lambdaball}
  \Lambda: B_R(T) \rightarrow B_R(T)  
\end{equation}
and there exists $L < 1$ such that
\begin{equation}
\|\Lambda(\omega^{n,1}) - \Lambda(\omega^{n,2}) \|_{V^T} \leq L \|\omega^{n,1} - \omega^{n,2}\|_{V^T}, \quad  \hbox{where} \quad \omega^{n,1}, \omega^{n,2} \in B_R(T).
\end{equation}
i.e.
\begin{eqnarray}
\left\|\Lambda\left(\omega^{n,1}\right)-\Lambda\left(\omega^{n,2}\right)\right\|_{\left.C(0, T] ; \mathcal{B}_a\right)}+\left\|\Lambda\left(\omega^{n,1}\right)-\Lambda\left(\omega^{n,2}\right)\right\|_{C^{\gamma}\left([0, T] ; \mathcal{B}_{a-\gamma}\right)} \leq L\left\|\omega^{n,1}-\omega^{n,2}\right\|_{V^T}.
\end{eqnarray}
\end{proposition}
\begin{proof}
The equation for the difference is given by
\begin{eqnarray}
\Lambda\left(\omega^{n, 1}\right)-\Lambda\left(\omega^{n, 2}\right) &=& -\int_0^t S_{t-s}\left(u_s^{n-1,1} \cdot \nabla \omega_s^{n-1,1}-u_s^{n-1,2} \cdot \nabla \omega_s^{n-1,2}\right) ds \\ 
& - & \int_0^t S_{t-s}\left(\xi \cdot \nabla (\omega_s^{n-1,1} - \omega_s^{n-1,2} \right) dW_s^H.
\end{eqnarray}
For the estimates corresponding to the space $C([0,T], \mathcal{B}_{\alpha})$ see below. For the H\"older space estimates, see the next two subsections. 

\noindent\textbf{Step 1: the nonlinear term.} Let us denote:
\begin{eqnarray}
A_t:=\int_0^t S_{t-s}\left(u_s^{n-1,1} \cdot \nabla \omega_s^{n-1,1}-u_s^{n-1,2} \cdot \nabla \omega_s^{n-1,2}\right) d s.
\end{eqnarray}
We can write
\begin{eqnarray}
u^{n-1,1} \cdot \nabla \omega^{n-1,1}-u^{n-1,2} \cdot \nabla \omega^{n-1,2}=\left(u^{n-1,1}-u^{n-1,2}\right) \cdot \nabla \omega^{n-1,1}+u^{n-1,2} \cdot \nabla\left(\omega^{n-1,1}-\omega^{n-1,2}\right).
\end{eqnarray}

We estimate each term separately in the \( \mathcal{B}_{\alpha - \delta} \) norm (for some small \( \delta > 0 \)), and then apply the semigroup smoothing.
Let \( \omega^{n-1,i} \in V^T \), i.e. both \( \omega^{n-1,1} \) and \( \omega^{n-1,2} \) belong to
\[
C([0, T]; \mathcal{B}_\alpha) \cap C^\gamma([0, T]; \mathcal{B}_{\alpha - \gamma}).
\]
Since \( u^{n-1,i} = K \ast \omega^{n-1,i} \), due to the properties of the Biot–Savart kernel \( K \), we have
\[
\|u^{n-1,1}_t - u^{n-1,2}_t\|_{\mathcal{B}_\alpha} \leq C_1 \|\omega^{n-1,1}_t - \omega^{n-1,2}_t\|_{\mathcal{B}_\alpha}.
\]

Then we can write
\[
\left\| \left(u^{n-1,1} - u^{n-1,2} \right) \cdot \nabla \omega^{n-1,1} \right\|_{\mathcal{B}_{\alpha - \delta}} 
\leq C_2 \|u^{n-1,1} - u^{n-1,2} \|_{\mathcal{B}_\alpha} \cdot \|\omega^{n-1,1} \|_{\mathcal{B}_\alpha}
\leq C_3 \|\omega^{n-1,1} - \omega^{n-1,2} \|_{\mathcal{B}_\alpha} \cdot \|\omega^{n-1,1} \|_{\mathcal{B}_\alpha}.
\]
for $\delta\ge\frac{1}{2}$. 
Likewise, 
\[
\left\| u^{n-1,2} \cdot \nabla \left(\omega^{n-1,1} - \omega^{n-1,2} \right) \right\|_{\mathcal{B}_{\alpha - \delta}} 
\leq C_4 \|u^{n-1,2} \|_{\mathcal{B}_\alpha} \cdot \|\omega^{n-1,1} - \omega^{n-1,2} \|_{\mathcal{B}_\alpha}.
\]
also for $\delta\ge\frac{1}{2}$.
Overall, the difference of nonlinearities satisfies:
\[
\| u^{n-1,1} \cdot \nabla \omega^{n-1,1} - u^{n-1,2} \cdot \nabla \omega^{n-1,2} \|_{\mathcal{B}_{\alpha - \delta}} \leq C(R) \cdot \| \omega^{n-1,1} - \omega^{n-1,2} \|_{\mathcal{B}_\alpha},
\]
since \( \|\omega^{n-1,i}\|\in B_R(T) \).
Now we can use the smoothing property of the semigroup:
\[
\|S_{t-s} f \|_{\mathcal{B}_\alpha} \leq C_5 (t - s)^{-\delta} \|f\|_{\mathcal{B}_{\alpha - \delta}}.
\]
So we have 
\[
\|A_t\|_{\mathcal{B}_\alpha}
= \left\| \int_0^t S_{t-s} \left( u_s^{n-1,1} \cdot \nabla \omega_s^{n-1,1} - u_s^{n-1,2} \cdot \nabla \omega_s^{n-1,2} \right) ds \right\|_{\mathcal{B}_\alpha}
\lesssim C(R) \int_0^t (t - s)^{-\delta} \|\omega_s^{n-1,1} - \omega_s^{n-1,2} \|_{\mathcal{B}_{\alpha}} ds.
\]
for $\delta\ge 1$
From here, using that \( \omega^{n-1,1} - \omega^{n-1,2} \in C([0,T]; \mathcal{B}_\alpha) \), we conclude:
\[
\|A\|_{C([0,T]; \mathcal{B}_\alpha)} \leq C_T \|\omega^{n-1,1} - \omega^{n-1,2} \|_{C([0,T]; \mathcal{B}_\alpha)},
\]
where \( C_T \sim T^{1 - \delta} \) and \( C_T \to 0 \) as \( T \to 0 \).
For the Hölder part, similar semigroup estimates yield, see next subsection,
\[
\| A_t - A_s \|_{\mathcal{B}_{\alpha - \gamma}} \leq C_6(T) |t - s|^\gamma \|\omega^{n-1,1} - \omega^{n-1,2} \|_{V^T},
\]
by similar arguments.
Hence:
\[
\| A \|_{C^\gamma([0,T]; \mathcal{B}_{\alpha - \gamma})} \leq C_T \|\omega^{n-1,1} - \omega^{n-1,2} \|_{V^T}.
\]
Altogether,
\[
\|A\|_{V^T} = \|A\|_{C([0,T]; \mathcal{B}_\alpha)} + \|A\|_{C^\gamma([0,T]; \mathcal{B}_{\alpha - \gamma})}
\leq C_T \|\omega^{n-1,1} - \omega^{n-1,2} \|_{V^T}.
\]
Choosing \( T > 0 \) small enough, we get \( C_T < 1 \), and this term becomes a contraction.

\noindent\textbf{Step 2: The Young integral.}
Let
\begin{eqnarray}
B_t:=\int_0^t S_{t-s} \xi \cdot \nabla\left(\omega_s^{n-1,1}-\omega_s^{n-1,2}\right) d W_s^H
\end{eqnarray}
and
\begin{eqnarray}
g_s:=S_{t-s} \xi \cdot \nabla\left(\omega_s^{n-1,1}-\omega_s^{n-1,2}\right).
\end{eqnarray}
We want to show that
\begin{eqnarray}\label{eq:estimb}
\|B\|_{C\left([0, T] ; \mathcal{B}_\alpha\right)}+\|B\|_{C^\gamma\left([0, T] ; \mathcal{B}_{\alpha-\gamma}\right)} \leq C(T)\left\|\omega^{n-1,1}-\omega^{n-1,2}\right\|_{V^T}.
\end{eqnarray}
For $\xi \in \mathcal{B}_{\infty}$ we can write
\begin{eqnarray}
\left\|\xi \cdot \nabla\left(\omega^{n-1,1}-\omega^{n-1,2}\right)\right\|_{\mathcal{B}_{\alpha-1/2}}\leq C_7\left\|\omega^{n-1,1}-\omega^{n-1,2}\right\|_{\mathcal{B}_\alpha} 
\end{eqnarray}
and by the same arguments as above
\begin{eqnarray}
\left\|S_{t-s}\left(\xi \cdot \nabla\left(\omega^{n-1,1}-\omega^{n-1,2}\right)\right)\right\|_{\mathcal{B}_\alpha} \leq \frac{C_8}{(t-s)^{1/2}}\left\|\omega^{n-1,1}-\omega^{n-1,2}\right\|_{\mathcal{B}_\alpha} .
\end{eqnarray}
Likewise, using Theorem~\ref{s} (see also Corollary \ref{i:reg}), for $s_1, s_2 >0$, we have:

\begin{eqnarray}
\left\|g_{s_1}-g_{s_2}\right\|_{\mathcal{B}_{\alpha-\gamma}} \lesssim\left|s_1-s_2\right|^{\gamma-1/2}\left\|\omega^{n-1,1}-\omega^{n-1,2}\right\|_{V^T}.
\end{eqnarray}
This proves \eqref{eq:estimb}.
Overall, we have shown that $\Lambda$ is a contraction on $V^T$. 
\end{proof}

\subsubsection{Properties of the nonlinear term}\label{sect:nonlinear}

We show that 
\begin{equation*}
t\rightarrow \int_{0}^{t}S_{t-s}u_{s}^{n-1}\cdot \nabla \omega
_{s}^{n-1}ds\quad \in V^T
\end{equation*}%
when $\omega ^{n-1}\in V^T.$ We use the fact that 
\begin{equation*}
\left\vert \left\vert u_{s}^{n-1}\cdot \nabla \omega _{s}^{n-1}\right\vert
\right\vert _{\alpha -\frac{1}{2}}\leq C\left\vert \left\vert \omega
_{s}^{n-1}\right\vert \right\vert _{\alpha }^{2}
\end{equation*}%
by the properties of the Biot-Savart law (see e.g. \cite{CL1}). 
So on a ball of radius $R$ in $V$\ then 
\begin{equation*}
\left\vert \left\vert u_{s}^{n-1}\cdot \nabla \omega _{s}^{n-1}\right\vert
\right\vert _{\alpha -\frac{1}{2}}\leq CR^{2}
\end{equation*}%
We get that 
\begin{eqnarray*}
&&\int_{0}^{t}S_{t-r}u_{r}^{n-1}\cdot \nabla \omega
_{r}^{n-1}dr-\int_{0}^{s}S_{s-r}u_{r}^{n-1}\cdot \nabla \omega _{r}^{n-1}dr
\\
&=&\int_{s}^{t}S_{t-r}u_{r}^{n-1}\cdot \nabla \omega
_{r}^{n-1}dr+\int_{0}^{s}\left( S_{t-s}-I\right) S_{s-r}u_{r}^{n-1}\cdot
\nabla \omega _{r}^{n-1}ds
\end{eqnarray*}
end first estimate each term in $\cB_\alpha$.
This entails
\begin{eqnarray*}
\left\vert \left\vert \int_{s}^{t}S_{t-r}u_{r}^{n-1}\cdot \nabla \omega
_{r}^{n-1}dr\right\vert \right\vert _{\alpha } &\leq &\int_{s}^{t}\left\vert
\left\vert S_{t-r}\left( u_{r}^{n-1}\cdot \nabla \omega _{r}^{n-1}\right)
\right\vert \right\vert _{\alpha }dr \\
&\leq &\int_{s}^{t}\left( t-r\right) ^{-\frac{1}{2}}\left\vert \left\vert
u_{r}^{n-1}\cdot \nabla \omega _{r}^{n-1}\right\vert \right\vert _{\alpha -%
\frac{1}{2}}dr \\
&\leq &CR^{2}\int_{s}^{t}\left( t-r\right) ^{-\frac{1}{2}}dr \\
&\leq &CR^{2}\left( t-s\right) ^{\frac{1}{2}}.
\end{eqnarray*}%
Moreover for $0<\sigma <1/2$ 
\begin{equation*}
\begin{aligned}
\left\vert \left\vert \int_{0}^{s}\left( S_{t-s}-I\right)
S_{s-r}u_{r}^{n-1}\cdot \nabla \omega _{r}^{n-1}ds\right\vert \right\vert
_{\alpha } &\leq C\left( t-s\right) ^{\sigma }\int_{0}^{s}\left\vert
\left\vert S_{s-r}u_{r}^{n-1}\cdot \nabla \omega _{r}^{n-1}\right\vert
\right\vert _{\alpha +\sigma }dr \\
&\leq C\left( t-s\right) ^{\sigma }\int_{0}^{s}\left( s-r\right) ^{-\frac{1%
}{2}-\sigma }\left\vert \left\vert u_{r}^{n-1}\cdot \nabla \omega
_{r}^{n-1}\right\vert \right\vert _{\alpha -\frac{1}{2}}dr \\
&\leq C\left( t-s\right) ^{\sigma }R^{2}\int_{0}^{s}\left( s-r\right) ^{-%
\frac{1}{2}-\sigma }dr \\
&\leq C\left( t-s\right) ^{\sigma }R^{2}s^{\frac{1}{2}-\sigma }.
\end{aligned}
\end{equation*}%
So from the above we deduce the continuity in the $\left\vert \left\vert
\cdot \right\vert \right\vert _{\alpha }$ norm of 
\begin{equation*}
t\rightarrow \int_{0}^{t}S_{t-s}u_{s}^{n-1}\cdot \nabla \omega _{s}^{n-1}ds.
\end{equation*}%
Let's move on to the H\"older continuity. We have that 
\begin{eqnarray*}
\left\vert \left\vert \int_{s}^{t}S_{t-r}u_{r}^{n-1}\cdot \nabla \omega
_{r}^{n-1}dr\right\vert \right\vert _{\alpha -\gamma } &\leq
&\int_{s}^{t}\left\vert \left\vert S_{t-r}\left( u_{r}^{n-1}\cdot \nabla
\omega _{r}^{n-1}\right) \right\vert \right\vert _{\alpha -\gamma }dr \\
&\leq &\int_{s}^{t} (t-s)^{\gamma-1/2} \left\vert \left\vert u_{r}^{n-1}\cdot \nabla \omega
_{r}^{n-1}\right\vert \right\vert _{\alpha -\frac{1}{2}}dr \\
&\leq &CR^{2}{\left( t-s\right) ^{\gamma+1/2 }}.
\end{eqnarray*}


Finally 
\begin{eqnarray*}
\left\vert \left\vert \int_{0}^{s}\left( S_{t-s}-I\right)
S_{s-r}u_{r}^{n-1}\cdot \nabla \omega _{r}^{n-1}dr\right\vert \right\vert
_{\alpha -\gamma } &\leq &\int_{0}^{s}\left\vert \left\vert \left(
S_{t-s}-I\right) S_{s-r}u_{r}^{n-1}\cdot \nabla \omega _{r}^{n-1}\right\vert
\right\vert _{\alpha -\gamma }dr \\
&\leq &C\left( t-s\right) ^{\gamma }\int_{0}^{s}\left\vert \left\vert
S_{s-r}u_{r}^{n-1}\cdot \nabla \omega _{r}^{n-1}\right\vert \right\vert
_{\alpha }dr \\
&\leq &C\left( t-s\right) ^{\gamma }\int_{0}^{s}\left( s-r\right) ^{-\frac{1%
}{2}}\left\vert \left\vert u_{r}^{n-1}\cdot \nabla \omega
_{r}^{n-1}\right\vert \right\vert _{\alpha -\frac{1}{2}}dr \\
&\leq &CR^{2}\left( t-s\right) ^{\gamma }\int_{0}^{s}\left( s-r\right) ^{-%
\frac{1}{2}}dr \\
&\leq &CR^{2}\left( t-s\right) ^{\gamma }s^{\frac{1}{2}}.
\end{eqnarray*}%
From which we can deduce that, on any ball of radius $R$ the drift term is $ \gamma$-{H\"{o}lder continuous for }$0<s<t$%
\ with respect to the $\left\vert \left\vert \cdot \right\vert \right\vert
_{\alpha -\gamma }$   norm, i.e.
\begin{equation*}
\left\vert \left\vert \int_{0}^{t}S_{t-r}u_{r}^{n-1}\cdot \nabla \omega
_{r}^{n-1}dr-\int_{0}^{s}S_{s-r}u_{r}^{n-1}\cdot \nabla \omega
_{r}^{n-1}dr\right\vert \right\vert _{\alpha -\gamma }\leq C(R,T)[\left(
t-s\right) ^{\gamma } +{(t-s)^{\gamma+1/2}} ].
\end{equation*}%
So we obtain
\begin{align*}
    \Big\|  \int_0^t S_{t-r}u^{n-1}_r \cdot \nabla \omega^{n-1}_r dr \Big\|_{\gamma,\alpha-\gamma}\leq C R [t^{\gamma-1/2} +t^\gamma ].
\end{align*}

\subsubsection{Properties of the fractional noise term}

We use here the inequalities (\ref{itoomega7}) and (\ref{itoomega6}) proved
in the previous section.

\subsection{Weak solutions}

In the following $[W^H]_\gamma$ denotes the H\"older norm on $W^H$ on a given time interval.\\

To show the equivalence of weak and mild solutions we first need the following lemma.

\begin{lemma}\label{fubini:young}
    Let $\omega\in V^T$ and $\xi\in \cB_\infty$. Then for every $\varphi\in \cB^*_{\alpha}$ we have 
    \begin{align*}
        \int\limits_0^t \langle \int_0^s S_{r-s}(\xi \cdot \nabla \omega_r)~\txtd W^H_r ,\varphi\rangle ~\txtd s = \int_0^t \int_r^t \langle S_{r-s} (\xi \cdot \nabla \omega_r),\varphi  \rangle~\txtd s ~\txtd W^H_r.
    \end{align*}
\end{lemma}
\begin{proof}
We consider smooth approximations $(W^{H,n})_{n \geq 1}$
of the noise such that \[ [W^{H,n} - W^H]_{\gamma} \to 0 \text { as } n\to\infty. \] 
Then by the continuous dependence of the solution w.r.t. the noise (which follows from the stability of the Young integration) we can find a sequence of solutions $(\omega^n)_{n\geq 1}$ such that 
\[  \| \omega^n-\omega\|_{V^T} \to 0 \text{ as } n\to \infty.\]
We denote
\[ W^n_{t,r} := \int_r^t \langle S_{r-s} (\xi \cdot \nabla \omega^n_r),\varphi  \rangle~\txtd s \text{ respectively } W_{t,r} := \int_r^t\langle S_{r-s} (\xi \cdot \nabla \omega_r) ,\varphi\rangle~\txtd s\]
and show that the Young integrals
\[ Z^n_t:= \int_0^t W^n_{t,r}~\txtd W^{H,n}_r \text{ and  } Z_t:= \int_0^t W_{t,r}~\txtd W^{H}_r  \]
are well-defined.~Due to the smoothness of $\omega^n$ the first statement is straightforward, we only check that we have enough H\"older regularity for $W_{t,r}$ in order to define $Z_t$ as a real-valued Young integral.~The fact that $W$ also depends on $t$ is not an issue. We compute for $r_2\geq r_1$
\begin{align*}
    W_{t,r_2} -W_{t,r_1}&=\int_{r_2}^t \langle S_{r_2-s} (\xi \cdot \nabla \omega_{r_2}), \varphi \rangle~\txtd s  - \int_{r_1}^t \langle S_{r_1-s} (\xi \cdot \nabla \omega_{r_1}), \varphi \rangle~\txtd s\\
    & = \int_{r_1}^t  \langle S_{s-r_2}(\xi\cdot \nabla \omega_{r_2})- S_{s-r_1}(\nabla \cdot \omega_{r_1}),\varphi\rangle~\txtd s +\int_{r_1}^{r_2} \langle S_{s-r_2}(\xi \cdot \nabla \omega_{r_2}) ,\varphi\rangle ~\txtd s. 
\end{align*}
The second term results in 
\begin{align*}
\Big| \int_{r_1}^{r_2} \langle S(s-r_2)(\xi \cdot \nabla \omega_{r_2}) ,\varphi\rangle ~\txtd s  \Big|    &\leq \int_{r_1}^{r_2} \| S_{s-r_2}(\xi \cdot \nabla \omega_{r_2}) \|_{\cB_\alpha}  \|\varphi\|_{\cB^*_\alpha}~\txtd s\\
& \lesssim \|\varphi\|_{\cB^*_\alpha} \int_{r_1}^{r_2}(s-r_2)^{-1/2}\|\omega_{r_2}\|_{\cB_\alpha}~\txtd s\\
& \lesssim (r_2-r_1)^{1/2}\|\varphi\|_{\cB^*_\alpha} \|\omega\|_{C([0,T];\cB_\alpha)}.
\end{align*}
For the first term we have
\begin{align*}
  &  \Big| \int_{r_2}^t \langle S_{r_2-s} (\xi \cdot \nabla \omega_{r_2}), \varphi \rangle~\txtd s  - \int_{r_1}^t \langle S_{r_1-s} (\xi \cdot \nabla \omega_{r_1}), \varphi \rangle~\txtd s \Big|\\
    &\leq \|\varphi\|_{\cB^*_\alpha} \int_{r_1}^{r_2} \| S_{s-r_1}[ (S_{r_2-r_1 } - \text{I})(\xi \cdot \nabla \omega_{r_2}) -(\xi\cdot \nabla(\omega_{r_1}-\omega_{r_2})) ]  \|_{\cB_\alpha}~\txtd s\\
    & \lesssim (r_2-r_1)^{1/2}\|\omega_2\|_{C([0,T];\cB_\alpha)} +(r_2-r_1)^\gamma T^{1/2} \|\omega_1-\omega_2\|_{C^\gamma([0,T];\cB_\alpha)}.
\end{align*}
Putting these together we infer that $W_{t,\cdot}$ is $1/2$-H\"older continuous. Since $\gamma>1/2$, this means that $Z$ is well-defined as a Young integral. We further set
\[ V^n_s: = \int_0^s S_{s-r} (\xi \cdot \nabla  \omega^n_r)~\txtd W^{H,n}_r \text{ and } V_s: = \int_0^s S_{s-r} (\xi\cdot \nabla  \omega_r)~\txtd W^{H}_r. \]
Due to the smoothness of $W^{H,n}$ and by Fubini's theorem we observe that 
\begin{align*}
   \int_0^t \langle V^n_s,\varphi \rangle ~\txtd s =Z^n_t = \int_0^t \int_r^t \langle S_{s-r} (\xi\cdot \nabla \omega^n_r),\varphi\rangle~\txtd s~ \txtd W^{H,n}_r.  
\end{align*}

Based on this we obtain
\begin{align*}
&    \Big|        \int\limits_0^t \langle \int_0^s S_{r-s}(\xi \cdot \nabla \omega_r)~\txtd W^H_r ,\varphi\rangle ~\txtd s = \int_0^t \int_r^t \langle S_{r-s} (\xi \cdot \nabla \omega_r),\varphi  \rangle~\txtd s ~\txtd W^H_s \Big| =\Big| \int_0^t \langle V_s,\varphi\rangle ~\txtd s -Z_t\Big|\\
    & = \Big| \int_0^t \langle V_s,\varphi\rangle ~\txtd s -\int_0^t \langle V^n_s,\varphi\rangle~\txtd s + Z^n_t -Z_t\Big|\\
    & \leq \int_0^t|\langle V_s-V^n_s,\varphi\rangle|~\txtd s + \|Z^n-Z\|_{C([0,T])}\\
    & \leq T \|V-V^n\|_{C([0,T];\cB_\alpha)}\|\varphi\|_{\cB^*_\alpha} +\|Z^n-Z\|_{C([0,T])}\\
    & \leq T \|\varphi\|_{\cB^*_\alpha} T^\gamma \|V-V^n\|_{C^\gamma([0,T];\cB_{\alpha-\gamma})} + T^\gamma \|Z-Z^n|\|_{C^\gamma([0,T])}\\
    & \lesssim T^{1+\gamma} \|\varphi\|_{\cB^*_\alpha} \| S_{s-\cdot}(\xi \cdot \nabla (\omega^n-\omega) )\|_{C^\gamma} [ W^H-W^{H,n}]_{\gamma} + T^\gamma\|W^n_{t,r}-W_{t,r} \|_{C^{1/2}}[ W^H-W^{H,n}]_{\gamma}\\
    & \lesssim T^{1/2+\gamma} \|\varphi\|_{\cB^*_\alpha} \|\omega-\omega^n\|_{V^T}\| W^H-W^{H,n}\|_{C^\gamma} + T^\gamma\|W^n_{t,r}-W_{t,r} \|_{C^{1/2}} [W^H-W^{H,n}]_{\gamma},
\end{align*}
which tends to $0$ as $n\to \infty$. Here we used in the last line the stability of the Young integral. For the first term we also estimated $\| S_{s-\cdot} (\xi \cdot \nabla \omega_\cdot) \|_{C^\gamma}$ as follows.  For $r_2\geq r_1$ we have
\begin{align*}
    S_{s-r_2}(\xi \cdot \nabla \omega_{r_2}) - S_{s-r_1}(\xi \cdot \nabla \omega_{r_1}) = S_{s-r_2}(\xi \cdot \nabla (\omega_{r_2}-\omega_{r_1}))+ S_{s-r_1}(S_{r_2-r_1}-\text{I})(\xi \cdot \nabla \omega_{r_1}).
\end{align*}
which leads to
\begin{align*}
\|S_{s-r_2}(\xi \cdot \nabla \omega_{r_2}) - S_{s-r_1}(\xi \cdot \nabla \omega_{r_1})\|_{\cB_{\alpha-\gamma}} \lesssim (s-r_2)^{-1/2}(r_2-r_1)^\gamma \|\omega\|_{C^\gamma([0,T];\cB_{\alpha-\gamma})} + (r_2-r_1)^\gamma \|\omega\|_{C([0,T];\cB_\alpha)}.
\end{align*}
\end{proof}
\begin{theorem}
    We let $\varphi\in \cB^*_{\alpha}$. Under the assumptions of Theorem \ref{thm:mild} the mild solution
    \[ \omega_t =S_t \omega_0  - \int_0^t S_{t-r}(u_r\cdot \nabla \omega_r)~\txtd r -\int_0^t S_{t-r}(\xi\cdot \nabla \omega_r)~\txtd W^H_r \]
    is equivalent to the weak solution
\[ \langle \omega_t ,\varphi \rangle =\langle \omega_0,\varphi\rangle +\int_0^t \langle \omega_r , u_r \cdot \nabla \varphi+\Delta \varphi\rangle~\txtd r +\int_0^t \langle \omega_r, \xi_r \cdot \nabla \varphi\rangle~\txtd W^H_r.  \]
\end{theorem}

\begin{proof}
We assume w.l.o.g that $\omega_0=0$. We show that a mild solution is a weak solution, the other direction follows by standard arguments. We have using the definition of the mild solution and Lemma~\ref{fubini:young} that 
\begin{align*}
&\int_0^t \langle \mathcal{E} \omega_s, \varphi\rangle~\txtd s = \int_0^t \langle \mathcal{E} \int_0^s S_{s-r} (\xi \cdot \nabla \omega_r)~\txtd W^H_r~\txtd s + \int_0^s S_{s-r}(u_r\cdot \nabla \omega_r)~\txtd r, \varphi\rangle~\txtd s \\
&=\int_0^t \int_r^t \langle \mathcal{E} S_{s-r} (\xi \cdot \nabla \omega_r),\varphi\rangle~\txtd s~\txtd r + \int_0^t \int_r^t  \langle \mathcal{E} S_{s-r}(u_r\cdot \nabla \omega_r),\varphi\rangle ~\txtd s~\txtd r\\
&= \int_0^t \langle \int_r^t \frac{\txtd }{\txtd s} S_{s-r} (\xi \cdot \nabla \omega_r)~\txtd s,\varphi\rangle~\txtd W^H_r + \int_0^t \langle \int_r^t \frac{\txtd}{\txtd s} S_{s-r}(u_r\cdot \nabla \omega_r),\varphi\rangle ~\txtd s ~\txtd r\\
    & =\int_0^t \langle S_{t-r}(\xi \cdot \nabla \omega_r),\varphi\rangle ~\txtd W^H_r - \int_0^t \langle \xi \cdot \nabla \omega_r,\varphi\rangle~\txtd W^H_r +\int_0^t \langle S_{t-r} (u_r\cdot \nabla \omega_r),\varphi\rangle~\txtd r - \int_0^t \langle u_r\cdot\nabla \omega_r,\varphi\rangle~\txtd r\\
    & = \langle \omega_t,\varphi\rangle - \int_0^t \langle \xi \cdot \nabla \omega_r,\varphi\rangle~\txtd W^H_r - \int_0^t \langle u_r\cdot\nabla \omega_r,\varphi\rangle~\txtd r.
\end{align*}

\end{proof}


\section{Hurst parameter estimation}\label{sect:hurst}

This section is devoted to the construction of a strongly consistent
estimator for the Hurst parameter $H$ of the fractional Brownian motion
driving the transport noise in the vorticity equation. However, the results
of this section are not specific to the particular vorticity equation
considered above. In fact, the construction and consistency of the Hurst
parameter estimator rely solely on the temporal scaling properties of the
fractional Brownian motion driving the equation. Neither the precise form of
the nonlinear drift term nor the specific structure of the transport noise
enters the argument. Once local existence and sufficient regularity of
solutions are guaranteed, the drift contribution becomes asymptotically
negligible in the rescaled quadratic variation, while the stochastic
integral term fully determines the limiting behavior. Consequently, the 
\emph{same} estimation procedure applies verbatim to the more general class
of stochastic partial differential equations introduced earlier, provided
they are driven by fractional Brownian motion with Hurst parameter $H>\frac{1%
}{2}$.

Nevertheless for the clarity of the exposition, in the following we will
refer only to the vorticity equation. More precisely, let $\omega _{t}$ be
the vorticity solution and let $\varphi $ be a fixed smooth test function
(e.g.\ a Fourier mode). We assume that we can observe the scalar-valued
process $X_{\cdot }:=\langle \omega _{\cdot },\varphi \rangle $ on an
interval $\left[ 0,t\right] $. The process $X$ constitutes our observable. We
will give an estimator of the Hurst parameter $H$ only in terms of $X$. To
be able to do so, we need a (weak) solution of the vorticity equation well
defined locally and not globally in time. The increments of $X$ over small
time intervals are decomposed into a drift term and a stochastic integral
with respect to fractional Brownian motion. This decomposition isolates the
contribution of the noise. Regularity estimates for the solution imply that
the drift increments are of higher order in the mesh size than the
stochastic integral terms. After rescaling, their contribution to the
quadratic variation vanishes almost surely. We also show that the rescaled
quadratic variation of the noise term converges almost surely to a finite
limit. Crucially, the explicit form of this limit is not required for the
sequel, only its existence and non-degeneracy are used (it is this property
that allows for the application of the result and the construction to the
solution of the general equation. Exploiting the self-similar scaling of
fractional Brownian motion, a ratio-type estimator based on quadratic
variations at two successive dyadic scales is introduced. Combining the
previous steps, the estimator is shown to converge almost surely to the true
Hurst parameter $H$. Let us proceed next with the details of the
construction.

The following proposition establishes a precise description of how the
small--time increments of fractional Brownian motion behave when they are
aggregated across a fine time partition. Although fractional Brownian motion
does not admit a classical quadratic variation, the result shows that, after
applying the correct rescaling, the accumulated squared increments stabilize
around a deterministic quantity, and that this stabilization occurs with
strong probabilistic control. In particular, the deviations from this
deterministic behavior become negligible as the time discretization is
refined.

The accompanying corollary strengthens this conclusion by showing that the
stabilization holds almost surely along dyadic partitions. This pathwise
convergence is crucial, as it allows one to work with individual
realizations of the noise rather than with expectations or distributional
limits. From the point of view of statistical estimation, this ensures that
the observed time series exhibits a predictable scaling behavior that can be
exploited directly from data.

In what follows, we will use equidistant partitions $(t_{j})_{j\geq 0}$ of
the positive half-line $[0,\infty )$, with $t_{0}=0$ and $t_{j+1}-t_{j}=%
\frac{1}{n}$ and denote by

\begin{equation*}
\mathcal{A}_{t}^{n}=\left\{ j\left\vert \left[ t_{j},t_{j+1}\right] \subset %
\left[ 0,t\right] \right. \right\} .
\end{equation*}

\begin{proposition}\label{prop15}
There exists a constant $C=C\left( t\right) $ independent of $n$ and $H$ such
that 
\begin{equation}
\mathbb{E}\left[ \left( \left( \frac{1}{n}\right) ^{1-2H}\sum_{j\in \mathcal{A}%
_{t}^{n}}\left( W_{t_{j+1}}^{H}-W_{t_{j}}^{H}\right) ^{2}-t\right) ^{2}%
\right] \leq \frac{C}{n^{4-4H}}.  \label{firstineq}
\end{equation}
\end{proposition}

The proposition gives us the following immediate corollary:

\begin{corollary}\label{scaled:cov}
For a dyadic partition $(t_{j})_{j\geq 0}$ of the positive half-line $%
[0,\infty )$ with $t_{j+1}-t_{j}=\frac{1}{2^{k}}$ we have that, $P$-almost
surely 
\begin{equation*}
\lim_{k\rightarrow \infty }\left( \frac{1}{2^{k}}\right) ^{1-2H}\sum_{j\in 
\mathcal{A}_{t}^{2^{k}}}\left( W_{t_{j+1}}^{H}-W_{t_{j}}^{H}\right) ^{2}=t.
\end{equation*}
\end{corollary}

\noindent \textbf{Proof of Proposition \ref{prop15}.} We will use the fact that, for $i\neq j$, we have that 
\begin{eqnarray*}
&&\mathbb{E}\left[ \left( W_{t_{i+1}}^{H}-W_{t_{i}}^{H}\right) ^{2}\left( \left(
W_{t_{j+1}}^{H}-W_{t_{j}}^{H}\right) ^{2}\right) \right] -\left( \frac{1}{n}%
\right) ^{4H} \\
&=&\frac{(2|i-j|^{2H}-|i-j+1|^{2H}-|i-j-1|^{2H})^{2}}{2n^{4H}} \\
&=&|i-j|^{4H}\frac{(2-|1+\frac{1}{|i-j|}|^{2H}-|1-\frac{1}{|i-j|}|^{2H})^{2}%
}{2n^{4H}} \\
&=&\frac{|i-j|^{4H}}{2n^{4H}}(|1+\frac{1}{|i-j|}|^{2H}+|1-\frac{1}{|i-j|}%
|^{2H}-2)^{2} \\
&=&\frac{|i-j|^{4H}}{2n^{4H}}\left( f_{H}\left( \frac{1}{|i-j|}\right)
\right) ^{2},
\end{eqnarray*}%
where $f_{H}\left( a\right) =\left( 1-a\right) ^{2H}+\left( 1+a\right)
^{2H}-2.$ We show that there exists a constant $c_{H}$ such that 
\begin{equation*}
f_{H}\left( a\right) =\left( 1-a\right) ^{2H}+\left( 1+a\right) ^{2H}-2\leq
c_{H}a^{2}
\end{equation*}%
for $a=\frac{1}{|i-j|}=\left\{ \frac{1}{1},\frac{1}{2},\frac{1}{3}%
,....\right\} .$ Observe that $a\rightarrow \frac{f_{H}\left( a\right) }{a^{2}%
}$ is continuous on $(0,1]$ and 
\begin{equation*}
\lim_{a\rightarrow 0}\frac{\left( 1-a\right) ^{2H}+\left( 1+a\right) ^{2H}-2%
}{a^{2}}=\lim_{a\rightarrow 0}\frac{2H\left( 2H-1\right) \left( \left(
1+a\right) ^{2H-2}+\left( 1-a\right) ^{2H-2}\right) }{2}=2H\left( 2H-1\right)
\end{equation*}%
so indeed $a\rightarrow \frac{f_{H}\left( a\right) }{a^{2}}$ is bounded on $%
(0,1]$ and we define $c_{H}:=\max_{a\in (0,1]}\frac{f_{H}\left( a\right) }{%
a^{2}}$.

We are now ready to prove (\ref{firstineq}). Observe that it is enough to
prove that 
\begin{equation}
\mathbb{E}\left[ \left( \left( \frac{1}{n}\right) ^{1-2H}\sum_{j\in \mathcal{A}%
_{t}^{n}}\left( W_{t_{j+1}}^{H}-W_{t_{j}}^{H}\right) ^{2}-\frac{\left[ tn%
\right] }{n}\right) ^{2}\right] \leq \frac{C}{n^{4-4H}} \label{firstineq1}
\end{equation}
where $\left[ tn\right] $ is the integer part on $tn$. Observe that 
\begin{equation*}
\begin{aligned}
&\mathbb{E}\left[ \left( \left( \frac{1}{n}\right) ^{1-2H}\sum_{j\in \mathcal{A}%
_{t}^{n}}\left( W_{t_{j+1}}^{H}-W_{t_{j}}^{H}\right) ^{2}-\frac{\left[ tn%
\right] }{n}\right) ^{2}\right] \\
&=\mathbb{E}\left[ \left( \sum_{j\in \mathcal{A}%
_{t}^{n}}\left( \left( \frac{1}{n}\right) ^{1-2H}\left(
W_{t_{j+1}}^{H}-W_{t_{j}}^{H}\right) ^{2}-\left( t_{j+1}-t_{j}\right)
\right) \right) ^{2}\right] \\
&=A+2B
\end{aligned}
\end{equation*}
where (note that $t_{j+1}-t_{j}=\frac{1}{n}$, $\mathbb{E}\left[ \left(
W_{t_{j+1}}^{H}-W_{t_{j}}^{H}\right) ^{2}\right] =\left(
t_{j+1}-t_{j}\right) ^{2H}$, $\mathbb{E}\left[ \left(
W_{t_{j+1}}^{H}-W_{t_{j}}^{H}\right) ^{4}\right] =3\left(
t_{j+1}-t_{j}\right) ^{4H}$)
\begin{eqnarray*}
A &:&=\sum_{j\in \mathcal{A}_{t}^{n}}\mathbb{E}\left[ \left( \left( \frac{1}{n}%
\right) ^{1-2H}\left( W_{t_{j+1}}^{H}-W_{t_{j}}^{H}\right) ^{2}-\left(
t_{j+1}-t_{j}\right) \right) ^{2}\right] \\
&=&\left( \frac{1}{n}\right) ^{2-4H}\sum_{j\in \mathcal{A}_{t}^{n}}\mathbb{E}\left[
\left( \left( W_{t_{j+1}}^{H}-W_{t_{j}}^{H}\right) ^{2}-\left(
t_{j+1}-t_{j}\right) ^{2H}\right) ^{2}\right] \\
&=&\left( \frac{1}{n}\right) ^{2-4H}\sum_{j\in \mathcal{A}_{t}^{n}}\left( \mathbb{E}
\left[ \left( W_{t_{j+1}}^{H}-W_{t_{j}}^{H}\right) ^{4}\right] -\left(
t_{j+1}-t_{j}\right) ^{4H}\right) \\
&=&\left( \frac{1}{n}\right) ^{2-4H}\sum_{j\in \mathcal{A}_{t}^{n}}2\left(
t_{j+1}-t_{j}\right) ^{4H}=2\left[ tn\right] \left( \frac{1}{n}\right)
^{2}\leq \frac{2t}{n}.
\end{eqnarray*}%
Also 
\begin{equation*}
\begin{aligned}
B &:=\left( \frac{1}{n}\right) ^{2-4H}\left( \sum_{\substack{ i,j\in 
\mathcal{A}_{t}^{n}  \\ j\leq i-1}}\mathbb{E}\left[ \left( \left( \left(
W_{t_{i+1}}^{H}-W_{t_{i}}^{H}\right) ^{2}-\left( t_{i+1}-t_{i}\right)
^{2H}\right) \right) \left( \left( \left(
W_{t_{j+1}}^{H}-W_{t_{j}}^{H}\right) ^{2}-\left( t_{j+1}-t_{j}\right)
^{2H}\right) \right) \right] \right) \\
&=\left( \frac{1}{n}\right) ^{2-4H}\sum_{\substack{ i,j\in \mathcal{A}%
_{t}^{n}  \\ j\leq i-1}}\mathbb{E}\left[ \left( W_{t_{i+1}}^{H}-W_{t_{i}}^{H}\right)
^{2}\left( \left( W_{t_{j+1}}^{H}-W_{t_{j}}^{H}\right) ^{2}\right) \right]
-\left( \frac{t}{n}\right) ^{4H} \\
&\leq \left( \frac{1}{n}\right) ^{2-4H}\sum_{\substack{ i,j\in \mathcal{A}%
_{t}^{n}  \\ j\leq i-1}}t^{4H}\frac{|i-j|^{4H}}{2n^{4H}}\left( f_{H}\left( 
\frac{1}{|i-j|}\right) \right) ^{2} \\
&\leq \left( \frac{1}{n}\right) ^{2-4H}\sum_{\substack{ i,j\in \mathcal{A}%
_{t}^{n}  \\ j\leq i-1}}t^{4H}\frac{|i-j|^{4H}}{2n^{4H}}c_{H}^{2}\frac{1}{%
|i-j|^{4}}\leq \frac{c_{H}^{2}}{2}\left( \frac{1}{n}\right) ^{2}\sum 
_{\substack{ i,j\in \mathcal{A}_{t}^{n}  \\ j\leq i-1}}|i-j|^{4H-4}.
\end{aligned}
\end{equation*}
By changing the order of summation we get that 
\begin{equation*}
B\leq \frac{c_{H}^{2}}{2}\left( \frac{1}{n}\right) ^{8H-2}\sum_{1\leq i\leq %
\left[ nt\right] }\sum_{j=1}^{i-1}j^{4H-4}.
\end{equation*}%
We compare $\sum_{j=1}^{i-1}j^{4H-4}$ with the integral 
\begin{equation*}
\int_{1}^{i-1}x^{4H-4}dx.
\end{equation*}%
Since the function $x\rightarrow x^{4H-4}$ is decreasing we get that 
\begin{equation*}
\sum_{j=1}^{i-1}j^{4H-4}\leq 1+\int_{1}^{i}x^{4H-4}dx=\left\{ 
\begin{array}{ccc}
1+\frac{1}{4H-3}\left. x^{4H-3}\right\vert _{1}^{i}\leq 1+\frac{1}{3-4H} & if
& H<\frac{3}{4} \\ 
1+\ln \left( i\right) & if & H=\frac{3}{4} \\ 
1+\frac{1}{4H-3}\left. x^{4H-3}\right\vert _{1}^{i}\leq 1+\frac{2}{4H-3}%
i^{4H-3} & if & H\geq \frac{3}{4}.%
\end{array}%
\right.
\end{equation*}%
It follows that 
\begin{eqnarray*}
B &\leq &\frac{c_{H}^{2}}{2}\left( \frac{1}{n}\right) ^{2}\sum_{1\leq i\leq %
\left[ nt\right] }\left\{ 
\begin{array}{ccc}
1+\frac{1}{4H-3}\left. x^{4H-3}\right\vert _{1}^{i}\leq 1+\frac{1}{3-4H} & if
& H<\frac{3}{4} \\ 
1+\ln \left( i\right) & if & H=\frac{3}{4} \\ 
1+\frac{1}{4H-3}\left. x^{4H-3}\right\vert _{1}^{i}\leq 1+\frac{2}{4H-3}%
i^{4H-3} & if & H\geq \frac{3}{4}%
\end{array}%
\right. \\
&\leq &\frac{c_{H}^{2}}{2}\left( \frac{1}{n}\right) ^{2}\left\{ 
\begin{array}{ccc}
c\left[ nt\right] & if & H<\frac{3}{4} \\ 
\left[ nt\right] \left( 1+\ln \left[ nt\right] \right) & if & H=\frac{3}{4}
\\ 
c\left( \left[ nt\right] +\frac{1}{4H-2}\left[ nt\right] ^{4H-2}\right) & if
& H\geq \frac{3}{4}%
\end{array}%
\right. =\left\{ 
\begin{array}{ccc}
O(\frac{1}{n}) & if & H<\frac{3}{4} \\ 
O(\frac{\ln n}{n}) & if & H=\frac{3}{4} \\ 
O\left( \left( \frac{1}{n}\right) ^{4-4H}\right) & if & H\geq \frac{3}{4}.%
\end{array}%
\right.
\end{eqnarray*}%
Since 
\begin{eqnarray*}
\sum_{1\leq i\leq \left[ nt\right] }\left( 1+\frac{2}{4H-3}i^{4H-3}\right)
&\leq &c\left( \left[ nt\right] +\int_{0}^{\left[ nt\right]
}x^{4H-3}dx\right) \\
&=&c\left( \left[ nt\right] +\frac{1}{4H-2}\left. x^{4H-2}\right\vert _{0}^{%
\left[ nt\right] }\right) =c\left( \left[ nt\right] +\frac{1}{4H-2}\left[ nt%
\right] ^{4H-2}\right).
\end{eqnarray*}%
The worst case scenario is the last case where the order is $\left( \frac{1}{%
n}\right) ^{4-4H}$ which still converges to 0 as $H<1$. Since both $A$ and $%
B $ converge to 0 it follows that 
\begin{equation*}
\mathbb{E}\left[ \left( \left( \frac{1}{n}\right) ^{1-2H}\sum_{i=0}^{n-1}\left(
W_{t_{i+1}}^{H}-W_{t_{i}}^{H}\right) ^{2}-\frac{\left[ nt\right] }{n}\right)
^{2}\right]
\end{equation*}%
 is of order at most $\left( \frac{1}{n}\right) ^{4-4H}$ so the result
follows as 
\begin{equation*}
\begin{aligned}
& \mathbb{E}\left[ \left( \left( \frac{1}{n}\right) ^{1-2H}\sum_{i=0}^{n-1}\left(
W_{t_{i+1}}^{H}-W_{t_{i}}^{H}\right) ^{2}-t\right) ^{2}\right] \\
& =\mathbb{E}\left[
\left( \left( \frac{1}{n}\right) ^{1-2H}\sum_{i=0}^{n-1}\left(
W_{t_{i+1}}^{H}-W_{t_{i}}^{H}\right) ^{2}-\frac{\left[ nt\right] }{n}\right)
^{2}\right] \\
& +2\left( t-\frac{\left[ nt\right] }{n}\right) \mathbb{E}\left[ \left( \left( \frac{1%
}{n}\right) ^{1-2H}\sum_{i=0}^{n-1}\left(
W_{t_{i+1}}^{H}-W_{t_{i}}^{H}\right) ^{2}-\frac{\left[ nt\right] }{n}\right) %
\right] +\left( t-\frac{\left[ nt\right] }{n}\right) ^{2} \\
&\leq c\left[ \left( \frac{1}{n}\right) ^{4-4H}+\left( \frac{1}{n}\right)
\left( \frac{1}{n}\right) ^{2-2H}+\left( \frac{1}{n}\right) ^{2}\right] \leq
3c\left( \frac{1}{n}\right) ^{4-4H}.
\end{aligned}
\end{equation*}

The importance of these results lies not in identifying the exact value of
the limiting quantity, but in the fact that a limit exists at all and that
its dependence on the time scale is entirely governed by the Hurst
parameter. This scale-invariant behavior is the key mechanism behind the
construction of the estimator: by comparing quadratic variations computed at
different resolutions, the unknown limiting constant cancels out, leaving an
expression that depends only on the Hurst parameter.

This observation explains why the same estimation argument extends
seamlessly to the full stochastic partial differential equation studied
earlier. When the solution of the SPDE is tested against a smooth spatial
function, the resulting time-dependent scalar process inherits the same
small--scale behavior as fractional Brownian motion, up to terms that are
smoother in time. The nonlinear drift and the specific structure of the
noise only affect these smoother contributions, which become negligible
after rescaling. Consequently, the estimator remains sensitive only to the
fractional noise component, and the probabilistic scaling argument developed
for fractional Brownian motion applies without modification.

In this sense, the proposition and its corollary provide the basis for the
entire estimation procedure: they isolate the universal scaling property
that drives the estimator and explain why the method is insensitive to the
complexity of the underlying SPDE.

The next result extends the scaling properties of fractional Brownian motion
itself to the class of processes obtained by integrating deterministic
functions against the fractional Brownian motion. Crucially, the convergence
does not depend on any special structure of the integrand beyond its
temporal regularity. The limit captures only the averaged energy of the
integrand and is insensitive to finer details.

\begin{proposition}
Let $x:\left[ 0,t\right] \rightarrow \mathbb{R}$ be a $\gamma $-H\"older
continuous function $x\in C^{\gamma }([0,t],\mathbb{R})$, for any $\frac{1}{2%
}<\gamma <H$\footnote{%
In other words, $x$ has the same H\"older continuity property as the
fractional Brownian motion $W^{H}$.}. Then%
\begin{equation*}
\int_{0}^{t}x_{s}dW_{s}^{H}
\end{equation*}%
is well defined as a Young integral, as well as any of the integrals on
sub-intervals of $\left[ 0,t\right] $. Then, for a dyadic partition $%
(t_{i})_{i\geq 0}$ of the positive half-line $[0,\infty )$ with $%
t_{i+1}-t_{i}=\frac{1}{2^{k}}$ we have that, $P$-almost surely 
\begin{equation}
\lim_{k\rightarrow \infty }\left( \frac{1}{2^{k}}\right) ^{1-2H}\sum_{i\in 
\mathcal{A}_{t}^{2^{k}}}\left( \int_{t_{i}}^{t_{i+1}}x_{s}dW_{s}^{H}\right)
^{2}=\int_{0}^{t}x_{s}^{2}ds.  \label{step2}
\end{equation}
\end{proposition}

\noindent \textbf{Proof.} Since, $\mathbb{P}$-almost surely 
\begin{equation*}
\lim_{k\rightarrow \infty }\int_{\frac{\left[ t2^{k}\right] }{2^{k}}%
}^{t}x_{s}^{2}ds=0,
\end{equation*}%
it is enough to prove that $P-$almost surely 
\begin{equation*}
\lim_{k\rightarrow \infty }\sum_{i\in \mathcal{A}_{t}^{2^{k}}}\left( \frac{1%
}{2^{k}}\right) ^{1-2H}\left( \int_{t_{i}}^{t_{i+1}}x_{s}dW_{s}^{H}\right)
^{2}-\int_{t_{i}}^{t_{i+1}}x_{s}^{2}ds=0.
\end{equation*}%
Since $x:\left[ 0,t\right] \rightarrow \mathbb{R}$ is $\gamma $-Holder
continuous 
\begin{equation*}
\left\vert \sum_{i\in \mathcal{A}_{t}^{2^{k}}}%
\int_{t_{i}}^{t_{i+1}}x_{s}^{2}ds-x_{t_{j}}^{2}\left( t_{j+1}-t_{j}\right)
\right\vert =\int_{0}^{\frac{\left[ t2^{k}\right] }{2^{k}}}\left\vert
x_{s}^{2}-x_{\frac{\left[ s2^{k}\right] }{2^{k}}}^{2}\right\vert ds\leq
Ct\left( \frac{1}{2^{k}}\right) ^{\gamma }.
\end{equation*}%
Next we can estimate the real-valued Young integral as 
\begin{equation*}
\left\vert \int_{t_{i}}^{t_{i+1}}x_{s}dW_{s}^{H}-x_{t_{i}}\left(
W_{t_{i+1}}^{H}-W_{t_{i}}^{H}\right) \right\vert \leq C\left( \frac{1}{2^{k}}%
\right) ^{2\gamma }.
\end{equation*}%
So 
\begin{eqnarray*}
&&\sum_{i\in \mathcal{A}_{t}^{2^{k}}}\left( \frac{1}{2^{k}}\right)
^{1-2H}\left( \int_{t_{i}}^{t_{i+1}}x_{s}dW_{s}^{H}\right) ^{2} \\
&=&A_{k}+B_{k}+\left( \frac{1}{2^{k}}\right) ^{1-2H}\sum_{i\in \mathcal{A}%
_{t}^{2^{k}}}\left\vert x_{t_{i}}\left( W_{t_{i+1}}^{H}-W_{t_{i}}^{H}\right)
\right\vert ^{2},
\end{eqnarray*}%
where%
\begin{eqnarray*}
\left\vert A_{k}\right\vert &=&\sum_{i\in \mathcal{A}_{t}^{2^{k}}}\left( 
\frac{1}{2^{k}}\right) ^{1-2H}\left\vert
\int_{t_{i}}^{t_{i+1}}x_{s}dW_{s}^{H}-x_{t_{i}}\left(
W_{t_{i+1}}^{H}-W_{t_{i}}^{H}\right) \right\vert ^{2} \\
&\leq &C_{t}2^{k}\left( \frac{1}{2^{k}}\right) ^{1-2H}C\left( \frac{1}{2^{k}}%
\right) ^{4\gamma }=C_{t}\left( \frac{1}{2^{k}}\right) ^{4\gamma -2H} \\
\left\vert B_{k}\right\vert &=&\sum_{i\in \mathcal{A}_{t}^{2^{k}}}\left( 
\frac{1}{2^{k}}\right) ^{1-2H}2\left\vert \left(
\int_{t_{i}}^{t_{i+1}}x_{s}dW_{s}^{H}-x_{t_{i}}\left(
W_{t_{i+1}}^{H}-W_{t_{i}}^{H}\right) \right) \right\vert \left\vert
x_{t_{i}}\left( W_{t_{i+1}}^{H}-W_{t_{i}}^{H}\right) \right\vert \\
&\leq &C_{t}2^{k}\left( \frac{1}{2^{k}}\right) ^{1-2H}C\left( \frac{1}{2^{k}}%
\right) ^{3\gamma }=C_{t}\left( \frac{1}{2^{k}}\right) ^{3\gamma -2H}.
\end{eqnarray*}

By choosing $\gamma $ sufficiently close to $H$, one can deduce that both $%
A_{k}$ and $B_{k}$ vanish as $k$ tends to $\infty $. So it is enough to
prove that 
\begin{eqnarray}
\lim_{k\rightarrow \infty }\sum_{i\in \mathcal{A}_{t}^{2^{k}}}\left( \frac{1%
}{2^{k}}\right) ^{1-2H}x_{t_{i}}^{2}\left(
W_{t_{i+1}}^{H}-W_{t_{i}}^{H}\right) ^{2}-x_{t_{i}}^{2}\left( \frac{1}{2^{k}}%
\right) &=&0  \label{toprove} \\
\lim_{k\rightarrow \infty }\left( \frac{1}{2^{k}}\right) ^{1-2H}\sum_{i\in 
\mathcal{A}_{t}^{2^{k}}}x_{t_{i}}^{2}\left( \left(
W_{t_{i+1}}^{H}-W_{t_{i}}^{H}\right) ^{2}-\left(  \frac{1}{2^{k}}%
 \right) ^{2H}\right) &=&0 . \notag
\end{eqnarray}%
Choose a sufficiently fine partition $(t_{j}^{\varepsilon })_{j\geq 0}$ of
the positive half-line $[0,\infty )$, with $t_{0}=0$ and $t_{i+1}-t_{i}=%
\frac{j}{2^{n}}$ such that inside each interval $\left[ t_{j}^{\varepsilon
},t_{j+1}^{\varepsilon }\right] $ 
\begin{equation*}
\left\vert \left( x_{t_{j}^{\varepsilon }}\right) ^{2}-\left( x_{s}\right)
^{2}\right\vert \leq \varepsilon,~~s\in \left[ t_{j}^{\varepsilon
},t_{j+1}^{\varepsilon }\right].
\end{equation*}%
For any partition $\left\{ \frac{j}{2^{n+m}},j=0,.....\right\} $ more
refined that the partition $\left\{ \frac{j}{2^{n}},j=0,.....\right\} $, we
will decompose the sum 
\begin{eqnarray}
&&\sum_{j\in \mathcal{A}_{t}^{2^{n+m}}}\left( \frac{1}{2^{n+m}}\right)
^{1-2H}x_{\frac{j}{2^{n+m}}}^{2}\left( W_{\frac{j+1}{2^{n+m}}}^{H}-W_{\frac{j%
}{2^{n+m}}}^{H}\right) ^{2}  \notag \\
&=&\sum_{j^{\prime }\in \mathcal{A}_{t}^{2^{n+m}}}\sum_{\left\{ j,\frac{%
j^{\prime }2^{m}}{2^{n+m}}\leq \frac{j}{2^{n+m}}<\frac{\left( j^{\prime
}+1\right) 2^{m}}{2^{n+m}}\right\} }\left( \frac{1}{2^{n+m}}\right)
^{1-2H}x_{\frac{j}{2^{n+m}}}^{2}\left( W_{\frac{j+1}{2^{n+m}}}^{H}-W_{\frac{j%
}{2^{n+m}}}^{H}\right) ^{2}  \notag \\
&&+\sum_{\left\{ j\geq \left[ t2^{n}\right] 2^{m}\right\} }\left( \frac{1}{%
2^{n+m}}\right) ^{1-2H}x_{\frac{j}{2^{n+m}}}^{2}\left( W_{\frac{j+1}{2^{n+m}}%
}^{H}-W_{\frac{j}{2^{n+m}}}^{H}\right) ^{2}  .\label{keys}
\end{eqnarray}%
Let us control the last term first. We note that 
\begin{eqnarray*}
&&\sum_{\left\{ j\geq \left[ t2^{n}\right] 2^{m}\right\} }\left( \frac{1}{%
2^{n+m}}\right) ^{1-2H}x_{\frac{j}{2^{n+m}}}^{2}\left( W_{\frac{j+1}{2^{n+m}}%
}^{H}-W_{\frac{j}{2^{n+m}}}^{H}\right) ^{2} \\
&\leq &C\lim \sup_{m\rightarrow \infty }\sum_{\left\{ j\geq \left[ t2^{n}%
\right] 2^{m}\right\} }\left( \frac{1}{2^{n+m}}\right) ^{1-2H}\left( W_{%
\frac{j+1}{2^{n+m}}}^{H}-W_{\frac{j}{2^{n+m}}}^{H}\right) ^{2} \\
&=&C\lim \sup_{m\rightarrow \infty }\int_{\frac{\left[ t2^{n}\right] }{2^{n}}%
}^{t}ds=C\left( t-\frac{\left[ t2^{n}\right] }{2^{n}}\right),
\end{eqnarray*}%
where we applied Corollary \ref{scaled:cov} in the last step.
Now, this term can be chosen small enough by choosing $n$ sufficiently
large. For the first term on the right hand side of (\ref{keys}) we take the
difference%
\begin{eqnarray*}
&&\sum_{j^{\prime }\in \mathcal{A}_{t}^{2^{n+m}}}\sum_{\left\{ j,\frac{%
j^{\prime }2^{m}}{2^{n+m}}\leq \frac{j}{2^{n+m}}<\frac{\left( j^{\prime
}+1\right) 2^{m}}{2^{n+m}}\right\} }\left( \frac{1}{2^{n+m}}\right)
^{1-2H}x_{\frac{j}{2^{n+m}}}^{2}\left( W_{\frac{j+1}{2^{n+m}}}^{H}-W_{\frac{j%
}{2^{n+m}}}^{H}\right) ^{2}- \\
&&-\sum_{j^{\prime }\in \mathcal{A}_{t}^{2^{n+m}}}\sum_{\left\{ j,\frac{%
j^{\prime }2^{m}}{2^{n+m}}\leq \frac{j}{2^{n+m}}<\frac{\left( j^{\prime
}+1\right) 2^{m}}{2^{n+m}}\right\} }\left( \frac{1}{2^{n+m}}\right)
^{1-2H}x_{\frac{j^{\prime }}{2^{n}}}^{2}\left( W_{\frac{j+1}{2^{n+m}}%
}^{H}-W_{\frac{j}{2^{n+m}}}^{H}\right) ^{2}.
\end{eqnarray*}

By the choice of the partition, since $j$ is such that $\frac{j^{\prime
}2^{m}}{2^{n+m}}\leq \frac{j}{2^{n+m}}<\frac{\left( j^{\prime }+1\right)
2^{m}}{2^{n+m}}$ we deduce that 
\begin{eqnarray*}
&& \left\vert \sum_{\left\{ j,\frac{j^{\prime }2^{m}}{2^{n+m}}\leq \frac{j}{%
2^{n+m}}<\frac{\left( j^{\prime }+1\right) 2^{m}}{2^{n+m}}\right\} }\left( 
\frac{1}{2^{n+m}}\right) ^{1-2H}\left( x_{\frac{j}{2^{n+m}}}^{2}-x_{\frac{%
j^{\prime }}{2^{n}}}^{2}\right) \left( W_{\frac{j+1}{2^{n+m}}}^{H}-W_{\frac{j%
}{2^{n+m}}}^{H}\right) ^{2}\right\vert \\
&& \leq \varepsilon \left\vert \sum_{\left\{ j,\frac{j^{\prime }2^{m}}{%
2^{n+m}} \leq \frac{j}{2^{n+m}}<\frac{\left( j^{\prime }+1\right) 2^{m}}{%
2^{n+m}}\right\} }\left( \frac{1}{2^{n+m}}\right) ^{1-2H}\left( W_{\frac{j+1%
}{2^{n+m}}}^{H}-W_{\frac{j}{2^{n+m}}}^{H}\right) ^{2}\right\vert.
\end{eqnarray*}%
So the above difference can be controlled by 
\begin{equation*}
\varepsilon \sum_{j^{\prime }\in \mathcal{A}_{t}^{2^{n+m}}}\sum_{\left\{ j,%
\frac{j^{\prime }2^{m}}{2^{n+m}}\leq \frac{j}{2^{n+m}}<\frac{\left(
j^{\prime }+1\right) 2^{m}}{2^{n+m}}\right\} }\left( \frac{1}{2^{n+m}}%
\right) ^{1-2H}\left( W_{\frac{j+1}{2^{n+m}}}^{H}-W_{\frac{j}{2^{n+m}}%
}^{H}\right) ^{2}.
\end{equation*}

Note that 
\begin{eqnarray*}
&&\lim_{m\rightarrow \infty }\sum_{\left\{ j,\frac{j^{\prime }2^{m}}{2^{n+m}}%
\leq \frac{j}{2^{n+m}}<\frac{\left( j^{\prime }+1\right) 2^{m}}{2^{n+m}}%
\right\} }\left( \frac{1}{2^{n+m}}\right) ^{1-2H}x_{\frac{j^{\prime }}{2^{n}}%
}^{2}\left( W_{\frac{j+1}{2^{n+m}}}^{H}-W_{\frac{j}{2^{n+m}}}^{H}\right) ^{2}
\\
&=&x_{\frac{j^{\prime }}{2^{n}}}^{2}\lim_{m\rightarrow \infty }\sum_{\left\{
j,\frac{j^{\prime }2^{m}}{2^{n+m}}\leq \frac{j}{2^{n+m}}<\frac{\left(
j^{\prime }+1\right) 2^{m}}{2^{n+m}}\right\} }\left( \frac{1}{2^{n+m}}%
\right) ^{1-2H}\left( W_{\frac{j+1}{2^{n+m}}}^{H}-W_{\frac{j}{2^{n+m}}%
}^{H}\right) ^{2} \\
&=&x_{\frac{j^{\prime }}{2^{n}}}^{2}\int_{\frac{j^{\prime }2^{m}}{2^{n+m}}}^{%
\frac{\left( j^{\prime }+1\right) 2^{m}}{2^{n+m}}}ds=x_{\frac{j^{\prime }}{%
2^{n}}}^{2}\frac{1}{2^{n}}
\end{eqnarray*}%
and therefore 
\begin{equation*}
\sum_{j^{\prime }\in \mathcal{A}_{t}^{2^{n+m}}}\sum_{\left\{ j,\frac{%
j^{\prime }2^{m}}{2^{n+m}}\leq \frac{j}{2^{n+m}}<\frac{\left( j^{\prime
}+1\right) 2^{m}}{2^{n+m}}\right\} }\left( \frac{1}{2^{n+m}}\right)
^{1-2H}x_{\frac{j^{\prime }}{2^{n}}}^{2}\left( W_{\frac{j+1}{2^{n+m}}%
}^{H}-W_{\frac{j}{2^{n+m}}}^{H}\right) ^{2}
\end{equation*}%
can be chosen as close to 
\begin{equation*}
\sum_{\left\{ j^{\prime },\left[ \frac{j^{\prime }}{2^{n}},\frac{j^{\prime
}+1}{2^{n}}\right] \subset \left[ 0,t\right] \right\} }x_{\frac{j^{\prime }}{%
2^{n}}}^{2}\frac{1}{2^{n}}
\end{equation*}%
by using a sufficiently large $m$ which in turn can be chosen as close to $%
\int_{0}^{t}x_{s}^{2}ds$ by choosing a sufficiently large $n$. We deduce
from here that (\ref{toprove}) is true and, indeed, (\ref{step2}) holds
true. 

As a next step, we consider the process 
\begin{equation*}
y_{t}=y_{0}+\int_{0}^{t}a_{s}ds+\int_{0}^{t}x_{s}dW_{s}^{H}
\end{equation*}%
where $a$ is bounded and continuous on the interval $\left[ 0,t\right] $ and 
$x:\left[ 0,t\right] \rightarrow \mathbb{R}$ be a $\gamma $-Holder
continuous function $x\in C^{\gamma }([0,t],\mathbb{R})$, for any $\frac{1}{2%
}<\gamma <H.$ The following result shows that the (rescaled) quadratic
variation result extends from pure stochastic integrals to general processes
that combine drift and fractional noise. \ The process under consideration
is deliberately chosen to mirror the structure of the scalar processes
obtained by testing the full SPDE against a smooth spatial function: it
consists of a deterministic initial value, a time-integrated drift term, and
a stochastic integral driven by fractional Brownian motion. The key message
of the next proposition is that, when the process is observed at
sufficiently fine time scales, the contribution of the drift becomes
asymptotically negligible in the rescaled quadratic variation. Although the
drift may influence the macroscopic behavior of the process, it does not
affect the small scale fluctuations that determine the scaling law. As a
result, the rescaled sum of squared increments of the full process converges
almost surely to the same limit as that of the stochastic integral alone.
From the point of view of the estimation programme, this is an important
step. It shows that the quadratic variation asymptotics are entirely
governed by the fractional noise component, even in the presence of
additional deterministic dynamics. In particular, the result confirms that
the presence of lower-order terms does not interfere with the extraction of
the Hurst parameter, provided they are sufficiently regular in time.

This proposition generalizes earlier results in the literature by allowing
for an arbitrary bounded and continuous drift and a time-dependent integrand
in the noise term. Importantly, the proof does not rely on any special
structure of these terms beyond regularity. This universality is exactly
what is needed for applications to nonlinear stochastic partial differential
equations, where both the drift and the noise coefficient typically depend
on the solution itself.

\begin{proposition}\label{quadratic:sde}
For a dyadic partition $(t_{i})_{i\geq 0}$ of the positive half-line $%
[0,\infty )$ with $t_{i+1}-t_{i}=\frac{1}{2^{k}}$ we have that, $P$-almost
surely 
\begin{equation*}
\lim_{k\rightarrow \infty }\left( \frac{1}{2^{k}}\right) ^{1-2H}\sum_{i\in 
\mathcal{A}_{t}^{2^{k}}}\left( y_{t_{i+1}}-y_{t_{i}}\right)
^{2}=\lim_{k\rightarrow \infty }\left( \frac{1}{2^{k}}\right)
^{1-2H}\sum_{i\in \mathcal{A}_{t}^{2^{k}}}\left(
\int_{t_{i}}^{t_{i+1}}x_{s}dW_{s}^{H}\right) ^{2}=\int_{0}^{t}x_{s}^{2}ds.
\end{equation*}
\end{proposition}

\noindent \textbf{Proof.} We have the following: 
\begin{eqnarray*}
\lim_{k\rightarrow \infty }\left( \frac{1}{2^{k}}\right) ^{-1}\sum_{i\in 
\mathcal{A}_{t}^{2^{k}}}\left( \int_{t_{i}}^{t_{i+1}}a_{s}ds\right) ^{2}
&=&\int_{0}^{t}a_{s}^{2}ds \\
\lim_{k\rightarrow \infty }\left( \frac{1}{2^{k}}\right) ^{1-2H}\sum_{i\in 
\mathcal{A}_{t}^{2^{k}}}\left( \int_{t_{i}}^{t_{i+1}}x_{s}dW_{s}^{H}\right)
^{2} &=&\int_{0}^{t}x_{s}^{2}ds.
\end{eqnarray*}%
Then 
\begin{equation*}
\lim_{k\rightarrow \infty }\left( \frac{1}{2^{k}}\right) ^{1-2H}\sum_{i\in 
\mathcal{A}_{t}^{2^{k}}}\left( y_{t_{i+1}}-y_{t_{i}}\right) ^{2}=A+B+C,
\end{equation*}%
where 
\begin{eqnarray*}
A &:=&\lim_{k\rightarrow \infty }\left( \frac{1}{2^{k}}\right)
^{1-2H}\sum_{i\in \mathcal{A}_{t}^{2^{k}}}\left(
\int_{t_{i}}^{t_{i+1}}a_{s}ds\right) ^{2}=\lim_{k\rightarrow \infty }\left( 
\frac{1}{2^{k}}\right) ^{2-2H}\lim_{k\rightarrow \infty }\left( \frac{1}{%
2^{k}}\right) ^{-1}\sum_{i\in \mathcal{A}_{t}^{2^{k}}}\left(
\int_{t_{i}}^{t_{i+1}}a_{s}ds\right) ^{2}\\
&=& 0\times \int_{0}^{t}a_{s}^{2}ds \\
C &:=&\lim_{k\rightarrow \infty }\left( \frac{1}{2^{k}}\right)
^{1-2H}\sum_{i\in \mathcal{A}_{t}^{2^{k}}}\left(
\int_{t_{i}}^{t_{i+1}}x_{s}dW_{s}^{H}\right) ^{2}=\int_{0}^{t}x_{s}^{2}ds \\
B &:=&\lim_{k\rightarrow \infty }\left( \frac{1}{2^{k}}\right)
^{1-2H}\sum_{i\in \mathcal{A}_{t}^{2^{k}}}\left(
\int_{t_{i}}^{t_{i+1}}a_{s}ds\right) \left(
\int_{t_{i}}^{t_{i+1}}x_{s}dW_{s}^{H}\right).
\end{eqnarray*}%
Note that 
\begin{eqnarray*}
&&\left( \frac{1}{2^{k}}\right) ^{1-2H}\left\vert \sum_{i\in \mathcal{A}%
_{t}^{2^{k}}}\left( \int_{t_{i}}^{t_{i+1}}a_{s}ds\right) \left(
\int_{t_{i}}^{t_{i+1}}x_{s}dW_{s}^{H}\right) \right\vert \\
&\leq &\left( \frac{1}{2^{k}}\right) ^{1-H}\sqrt{\left( \frac{1}{2^{k}}\right)
^{-1}\sum_{i=0}^{2^{k}-1}\left( \int_{t_{i}}^{t_{i+1}}a_{s}ds\right) ^{2}\left( 
\frac{1}{2^{k}}\right) ^{1-2H}\sum_{i=0}^{2^{k}-1}\left(
\int_{t_{i}}^{t_{i+1}}x_{s}dW_{s}^{H}\right) ^{2}}\\
&\rightarrow& 0\times \sqrt{%
\int_{0}^{t}a_{s}^{2}ds\int_{0}^{t}x_{s}^{2}ds}=0
\end{eqnarray*}%
and it follows that also $B=0$. 

We are ready now to apply the result to our framework, i.e.~to~\eqref{eq2}. To this aim we choose a
smooth test function $\varphi,$ for example an exponential function 
\begin{equation*}
\varphi \left( x\right) =\exp \left( ik\cdot x\right), ~~  x \in \mathbb{T}^2
\end{equation*}%
and we have using the weak formulation (recall Theorem \ref{thm:mild}) that 
\begin{equation*}
\left( \omega _{t},\varphi \right) =\left( \omega _{0},\varphi \right)
+\int_{0}^{t}\left( \omega _{s},u_{s}\cdot \nabla \varphi +\Delta \varphi
\right) ds+\int_{0}^{t}\left( \omega _{s},\xi _{s}\cdot \nabla \varphi
\right) dW_{s}^{H}
\end{equation*}%
where 
\begin{equation*}
s\rightarrow a_{s}:=\left( \omega _{s},u_{s}\cdot \nabla \varphi +\Delta
\varphi \right)
\end{equation*}%
is bounded and continuous on $\left[ 0,t\right] $ and 
\begin{equation*}
s\rightarrow x_{s}:=\left( \omega _{s},\xi _{s}\cdot \nabla \varphi \right)
\end{equation*}%
is $\gamma $-Holder continuous function for any $\frac{1}{2}<\gamma <H$. Now, Proposition \ref{quadratic:sde} gives us the following immediate corollary:

\begin{corollary}
For a dyadic partition $(t_{j})_{j\geq 0}$ of the positive half-line $%
[0,\infty )$ with $t_{j+1}-t_{j}=\frac{1}{2^{k}}$ we have that, $P$-almost
surely 
\begin{equation*}
\lim_{k\rightarrow \infty }\left( \frac{1}{2^{k}}\right) ^{1-2H}\sum_{j\in 
\mathcal{A}_{t}^{2^{k}}}\left( \left( \omega _{t_{j+1}},\varphi \right)
-\left( \omega _{t_{j}},\varphi \right) \right) ^{2}=\int_{0}^{t}\left(
\omega _{s},\xi _{s}\cdot \nabla \varphi \right) ^{2}ds.
\end{equation*}
\end{corollary}

We can now give the estimator for the Hurst parameter. Following from \cite{Tudor}, we define
define%
\begin{equation}\label{est:h}
H_{k}:=\frac{1}{2}\left( \frac{1}{\log 2}\log \frac{\sum_{j\in \mathcal{A}%
_{t}^{2^{k}}}\left( \left( \omega _{t_{j+1}},\varphi \right) -\left( \omega
_{t_{j}},\varphi \right) \right) ^{2}}{\sum_{j\in \mathcal{A}%
_{t}^{2^{k+1}}}\left( \left( \omega _{t_{j+1}},\varphi \right) -\left(
\omega _{t_{j}},\varphi \right) \right) ^{2}}-1\right).
\end{equation}

\begin{proposition}\label{prop:est:h}
We have that, $P$-almost surely 
\begin{equation*}
\lim_{k\rightarrow \infty }H_{k}=H.
\end{equation*}
\end{proposition}

\noindent \textbf{Proof.} Define%
\begin{equation*}
V_{k}:=\left( \frac{1}{2^{k}}\right) ^{1-2H}\sum_{j\in \mathcal{A}%
_{t}^{2^{k}}}\left( \left( \omega _{t_{j+1}},\varphi \right) -\left( \omega
_{t_{j}},\varphi \right) \right) ^{2}.
\end{equation*}%
Then from Proposition \ref{quadratic:sde} we deduce that 
\begin{equation*}
\lim_{k\rightarrow \infty }V_{k}=\int_{0}^{t}\left( \omega _{s},\xi
_{s}\cdot \nabla \varphi \right) ^{2}ds.
\end{equation*}%
Then observe that 
\begin{equation*}
\frac{\sum_{j\in \mathcal{A}_{t}^{2^{k}}}\left( \left( \omega
_{t_{j+1}},\varphi \right) -\left( \omega _{t_{j}},\varphi \right) \right)
^{2}}{\sum_{j\in \mathcal{A}_{t}^{2^{k+1}}}\left( \left( \omega
_{t_{j+1}},\varphi \right) -\left( \omega _{t_{j}},\varphi \right) \right)
^{2}}=\frac{\left( \frac{1}{2^{k+1}}\right) ^{1-2H}V_{k}}{\left( \frac{1}{%
2^{k}}\right) ^{1-2H}V_{k+1}}\rightarrow \frac{1}{2^{1-2H}}.
\end{equation*}%
So%
\begin{equation*}
\log \frac{\sum_{j\in \mathcal{A}_{t}^{2^{k}}}\left( \left( \omega
_{t_{j+1}},\varphi \right) -\left( \omega _{t_{j}},\varphi \right) \right)
^{2}}{\sum_{j\in \mathcal{A}_{t}^{2^{k+1}}}\left( \left( \omega
_{t_{j+1}},\varphi \right) -\left( \omega _{t_{j}},\varphi \right) \right)
^{2}}\rightarrow -\left( 1-2H\right) \log 2=\left( 2H-1\right) \log 2.
\end{equation*}%
Finally we deduce that: 
\begin{equation*}
H_{k}=\frac{1}{2}\left( \frac{1}{\log 2}\log \frac{\sum_{j\in \mathcal{A}%
_{t}^{2^{k}}}\left( \left( \omega _{t_{j+1}},\varphi \right) -\left( \omega
_{t_{j}},\varphi \right) \right) ^{2}}{\sum_{j\in \mathcal{A}%
_{t}^{2^{k+1}}}\left( \left( \omega _{t_{j+1}},\varphi \right) -\left(
\omega _{t_{j}},\varphi \right) \right) ^{2}}-1\right) \rightarrow H.
\end{equation*}

\section{Fixed point argument for the general case}\label{sect:genproofs}
In this section we extend the fixed point argument developed in Section \ref{sect:vorticityeqn} to a more general class of stochastic evolution equations. The framework introduced in Section \ref{sect:sewinglemma}, in particular the sewing lemma for general integrands, allows us to treat a broader range of nonlinearities and noise structures beyond the transport-type case. Since the arguments closely follow those of Section \ref{sect:vorticityeqn}, we only highlight the main steps and the necessary modifications, and keep the exposition deliberately concise.

Let $V^T=C([0,T];\cB_\alpha)\cap C^\gamma([0,T];\cB_{\alpha-\gamma})$ be as before.
\begin{theorem}
   Under the Assumptions~\ref{ass:d} and~\ref{ass:f}, the equation \eqref{general} has a unique solution in $V^T$. 
\end{theorem}
\begin{proof}
As before, the proof relies on a fixed point argument. To this aim we define the map 
\begin{equation*}
\Lambda :V^T\rightarrow V^T
\end{equation*}%
and show that there exists $\omega $ such that 
\begin{equation*}
\Lambda (\omega )=\omega 
\end{equation*}%
where%
\begin{equation}
\Lambda _{t}(\omega )=S_{t}\omega _{0}-\int_{0}^{t}S_{t-s}(\mathcal{D}\omega
_{s})ds-\int_{0}^{t}S_{t-s}(\mathcal{F}\omega _{s})dW_{s}^{H}.
\label{eq:mainmild1}
\end{equation}%
Observe that
\begin{equation*}
\left\vert \left\vert S_{t}\omega _{0}\right\vert \right\vert _{\mathcal{B}%
_{\alpha }}\leq \left\vert \left\vert \omega _{0}\right\vert \right\vert _{%
\mathcal{B}_{\alpha }}
\end{equation*}%
\begin{eqnarray}
\left\vert \left\vert \int_{0}^{t}S_{t-s}(\mathcal{D}\omega
_{s})ds\right\vert \right\vert _{\mathcal{B}_{\alpha }} &\leq
&\int_{0}^{t}\left\vert \left\vert S_{t-s}(\mathcal{D}\omega
_{s})\right\vert \right\vert _{\mathcal{B}_{\alpha }}ds  
\leq C\int_{0}^{t}\frac{1}{\left( t-s\right) ^{\beta }}\left\vert
\left\vert (\mathcal{D}\omega _{s})\right\vert \right\vert _{\mathcal{B}%
_{\alpha -\beta }}ds  \notag \\
&\leq &C\int_{0}^{t}\frac{1}{\left( t-s\right) ^{\beta }}\left\vert
\left\vert \omega _{s}\right\vert \right\vert _{\mathcal{B}_{\alpha
}}^{q}ds\leq C\sup_{s\in \left[ 0,t\right] }\left\vert \left\vert \omega
_{s}\right\vert \right\vert _{\mathcal{B}_{\alpha }}^{q}\frac{t^{1-\beta }}{%
1-\beta }.  \label{norm1}
\end{eqnarray}%
The above computation gives that control on the nonlinear term. We also
need to justify that it is continuous in time with values in $\cB_\alpha$. For this observe that 
\begin{equation*}
\left\vert \left\vert \int_{0}^{t_{1}}S_{t_{1}-s}(\mathcal{D}\omega
_{s})ds-\int_{0}^{t_{2}}S_{t_{1}-s}(\mathcal{D}\omega _{s})ds\right\vert
\right\vert _{\mathcal{B}_{\alpha }}=A+B
\end{equation*}%
where:%
\begin{eqnarray*}
A &:&=\left\vert \left\vert \int_{t_{2}}^{t_{1}}S_{t_{1}-s}(\mathcal{D}%
\omega _{s})ds\right\vert \right\vert _{\mathcal{B}_{\alpha }}\leq
\int_{t_{2}}^{t_{1}}\left\vert \left\vert S_{t_{1}-s}(\mathcal{D}\omega
_{s})\right\vert \right\vert _{\mathcal{B}_{\alpha }}ds \leq \int_{t_{2}}^{t_{1}}\frac{1}{\left( t_{1}-s\right) ^{\beta }}%
\left\vert \left\vert (\mathcal{D}\omega _{s})\right\vert \right\vert _{%
\mathcal{B}_{\alpha -\beta }}ds \\
&\leq &\int_{t_{2}}^{t_{1}}\frac{1}{\left( t_{1}-s\right) ^{\beta }}%
\left\vert \left\vert \omega _{s}\right\vert \right\vert _{\mathcal{B}%
_{\alpha }}^{q}ds\leq C\sup_{s\in \left[ 0,t\right] }\left\vert \left\vert
\omega _{s}\right\vert \right\vert _{\mathcal{B}_{\alpha }}^{q}\frac{\left(
t_{1}-t_{2}\right) ^{1-\beta }}{1-\beta }
\end{eqnarray*}%
and, for $\sigma $ such that $\beta +\sigma <1$, say $\sigma =\frac{1-\beta 
}{2}$ 
\begin{eqnarray*}
B &:&=\left\vert \left\vert \int_{0}^{t_{2}}\left(
S_{t_{1}-s}-S_{t_{2}-s}\right) (\mathcal{D}\omega _{s})ds\right\vert
\right\vert _{\mathcal{B}_{\alpha }} \leq \int_{0}^{t_{2}}\left\vert \left\vert \left( S_{t_{1}-t_{2}}-I\right)
S_{t_{2}-s}(\mathcal{D}\omega _{s})\right\vert \right\vert _{\mathcal{B}%
_{\alpha }}ds \\
&\leq &\int_{0}^{t_{2}}\left( t_{1}-t_{2}\right) ^{\sigma }\left\vert
\left\vert S_{t_{2}-s}(\mathcal{D}\omega _{s})\right\vert \right\vert _{%
\mathcal{B}_{\alpha +\sigma }}ds \leq \left( t_{1}-t_{2}\right) ^{\sigma }\int_{0}^{t_{2}}\frac{1}{\left(
t_{2}-s\right) ^{\beta +\sigma }}\left\vert \left\vert (\mathcal{D}\omega
_{s})\right\vert \right\vert _{\mathcal{B}_{\alpha -\beta }}ds \\
&\leq &\left( t_{1}-t_{2}\right) ^{\sigma }\int_{0}^{t_{2}}\frac{1}{\left(
t_{2}-s\right) ^{\beta +\sigma }}\left\vert \left\vert \omega
_{s}\right\vert \right\vert _{\mathcal{B}_{\alpha }}^{q}ds \leq C\left( t_{1}-t_{2}\right) ^{\sigma }\sup_{s\in \left[ 0,t\right]
}\left\vert \left\vert \omega _{s}\right\vert \right\vert _{\mathcal{B}%
_{\alpha }}^{q}\frac{t_{2}^{1-\beta -\sigma }}{1-\beta -\sigma }.
\end{eqnarray*}%
From here we deduce that 
\begin{equation*}
\left\vert \left\vert \int_{0}^{t_{1}}S_{t_{1}-s}(\mathcal{D}\omega
_{s})ds-\int_{0}^{t_{2}}S_{t_{1}-s}(\mathcal{D}\omega _{s})ds\right\vert
\right\vert _{\mathcal{B}_{\alpha }}\leq C\left( T\right) \sup_{s\in \left[
0,t\right] }\left\vert \left\vert \omega _{s}\right\vert \right\vert _{%
\mathcal{B}_{\alpha }}^{q}\left( \left( t_{1}-t_{2}\right) ^{\frac{1-\beta }{%
2}}+\left( t_{1}-t_{2}\right) ^{1-\beta }\right) 
\end{equation*}%
so indeed the nonlinear term is continuous.
We need to prove now the H\"older continuity in the norm $\left\vert
\left\vert \cdot \right\vert \right\vert _{\mathcal{B}_{\alpha -\gamma }}$.
For this observe that 
\begin{equation*}
\left\vert \left\vert \int_{0}^{t_{1}}S_{t_{1}-s}(\mathcal{D}\omega
_{s})ds-\int_{0}^{t_{2}}S_{t_{1}-s}(\mathcal{D}\omega _{s})ds\right\vert
\right\vert _{\mathcal{B}_{\alpha -\gamma }}=A+B
\end{equation*}%
where:%
\begin{eqnarray*}
A &:&=\left\vert \left\vert \int_{t_{2}}^{t_{1}}S_{t_{1}-s}(\mathcal{D}%
\omega _{s})ds\right\vert \right\vert _{\mathcal{B}_{\alpha -\gamma }} \leq \int_{t_{2}}^{t_{1}}\left\vert \left\vert S_{t_{1}-s}(\mathcal{D}%
\omega _{s})\right\vert \right\vert _{\mathcal{B}_{\alpha -\gamma }}ds \leq C\int_{t_{2}}^{t_{1}} (t_1-s)^{-\max\{0,\beta-\gamma\}} \| \mathcal{D}\omega
_{s}\| _{\mathcal{B}_{\alpha -\beta }}ds \\
&\leq &C\sup_{s\in \left[ 0,t\right] }\left\vert \left\vert \omega
_{s}\right\vert \right\vert _{\mathcal{B}_{\alpha }}^{q}\left(
t_{1}-t_{2}\right)^{\min\{ 1,1+\gamma-\beta \}} \\
&\leq &CT^{1-\beta }\sup_{s\in \left[ 0,t\right] }\left\vert \left\vert
\omega _{s}\right\vert \right\vert _{\mathcal{B}_{\alpha }}^{q}\left(
t_{1}-t_{2}\right) ^{\gamma }
\end{eqnarray*}
and
\begin{eqnarray*}
B &:&=\left\vert \left\vert \int_{0}^{t_{2}}\left(
S_{t_{1}-s}-S_{t_{2}-s}\right) (\mathcal{D}\omega _{s})ds\right\vert
\right\vert _{\mathcal{B}_{\alpha -\gamma }} \leq \int_{0}^{t_{2}}\left\vert \left\vert \left(
S_{t_{1}-s}-S_{t_{2}-s}\right) (\mathcal{D}\omega _{s})ds\right\vert
\right\vert _{\mathcal{B}_{\alpha -\gamma }}ds \\
&\leq &\int_{0}^{t_{2}}\left\vert \left\vert \left( S_{t_{1}-t_{2}}-I\right)
S_{t_{2}-s}(\mathcal{D}\omega _{s})\right\vert \right\vert _{\mathcal{B}%
_{\alpha -\gamma }}ds \leq C\left( t_{2}-t_{1}\right) ^{\gamma }\int_{0}^{t_{2}}\left\vert
\left\vert S_{t_{2}-s}(\mathcal{D}\omega _{s})\right\vert \right\vert _{%
\mathcal{B}_{\alpha }}ds \\
&\leq &C\left( t_{2}-t_{1}\right) ^{\gamma }\int_{0}^{t_{2}}\frac{1}{\left(
t_{2}-s\right) ^{-\beta }}\left\vert \left\vert \mathcal{D}\omega
_{s}\right\vert \right\vert _{\mathcal{B}_{\alpha -\gamma }}ds \leq C\sup_{s\in \left[ 0,t\right] }\left\vert \left\vert \omega
_{s}\right\vert \right\vert _{\mathcal{B}_{\alpha }}^{q}t_{2}^{1-\beta
}\left( t_{2}-t_{1}\right) ^{\gamma } \\
&\leq &C\sup_{s\in \left[ 0,t\right] }\left\vert \left\vert \omega
_{s}\right\vert \right\vert _{\mathcal{B}_{\alpha }}^{q}T^{1-\beta }\left(
t_{1}-t_{2}\right) ^{\gamma}.
\end{eqnarray*}%
It follows that 
\begin{equation}
\left\vert \left\vert \int_{0}^{t_{1}}S_{t_{1}-s}(\mathcal{D}\omega
_{s})ds-\int_{0}^{t_{2}}S_{t_{1}-s}(\mathcal{D}\omega _{s})ds\right\vert
\right\vert _{\mathcal{B}_{\alpha -\gamma }}\leq CT^{1-\beta }\sup_{s\in %
\left[ 0,T\right] }\left\vert \left\vert \omega _{s}\right\vert \right\vert
_{\mathcal{B}_{\alpha }}^{q}\left( t_{1}-t_{2}\right)^{\gamma}.
\label{norm2}
\end{equation}%
In conclusion, from (\ref{norm1}) and (\ref{norm2}) it follows that 
\begin{equation*}
\left\Vert \int_{0}^{\cdot }S_{t-s}(\mathcal{D}\omega _{s})ds\right\Vert
_{V^{T}}\leq CT^{1-\beta }\sup_{s\in \left[ 0,T\right] }\left\vert
\left\vert \omega _{s}\right\vert \right\vert _{\mathcal{B}_{\alpha }}^{q}.
\end{equation*}
In particular one can choose $T$ sufficiently small such that 
\begin{equation*}
\left\Vert \int_{0}^{\cdot }S_{t-s}(\mathcal{D}\omega _{s})ds\right\Vert
_{V^{T}}\leq R/2.
\end{equation*}
We now show that $\Lambda :V^T\rightarrow V^T$ is a contraction. Therefore, we need to estimate

\begin{eqnarray*}
\left\vert \left\vert \int_{0}^{t}S_{t-s}(\mathcal{D}\omega
_{s}^{1})ds-\int_{0}^{t}S_{t-s}(\mathcal{D}\omega _{s}^{2})ds\right\vert
\right\vert _{\mathcal{B}_{\alpha }} &\leq &\int_{0}^{t}\frac{1}{\left(
t-s\right) ^{\beta }}\left\vert \left\vert \mathcal{D\omega }^{1}\mathcal{%
-D\omega }^{2}\right\vert \right\vert _{\mathcal{B}_{\alpha -\beta }}ds \\
&\leq &Ct^{1-\beta }\max \left( \left\vert \left\vert \mathcal{\omega }%
^{1}\right\vert \right\vert _{\mathcal{B}_{\alpha }}^{p},\left\vert
\left\vert \mathcal{\omega }^{2}\right\vert \right\vert _{\mathcal{B}%
_{\alpha }}^{p}\right) \sup_{s\in \left[ 0,t\right] }\left\vert \left\vert 
\mathcal{\omega }_{s}^{1}\mathcal{-\omega }_{s}^{2}\right\vert \right\vert _{%
\mathcal{B}_{\alpha }}.
\end{eqnarray*}%
Hence 
\begin{equation}
\left\Vert \int_{0}^{\cdot }S_{\cdot -s}(\mathcal{D}\omega
_{s}^{1})ds-\int_{0}^{\cdot }S_{\cdot -s}(\mathcal{D}\omega
_{s}^{2})ds\right\Vert _{\left. C(0,T];\mathcal{B}_{a}\right) }\leq
CT^{1-\beta }\left( \left\Vert \mathcal{\omega }^{1}\right\Vert _{C\left(
(0,T];\mathcal{B}_{a}\right) }+\left\Vert \mathcal{\omega }^{2}\right\Vert
_{C\left( (0,T];\mathcal{B}_{a}\right) }\right) \left\Vert \mathcal{\omega }%
^{1}\mathcal{-\omega }^{2}\right\Vert _{\left. C(0,T];\mathcal{B}_{a}\right).
}  \label{111}
\end{equation}%
Denote%
\begin{equation*}
\Upsilon _{t}=\int_{0}^{t}S_{t-s}(\mathcal{D}\omega
_{s}^{1})ds-\int_{0}^{t}S_{t-s}(\mathcal{D}\omega
_{s}^{2})ds=\int_{0}^{t}S_{t-s}(\mathcal{D}\omega _{s}^{1}-\mathcal{D}\omega
_{s}^{2})ds.
\end{equation*}%
Then 
\begin{equation*}
\left\vert \left\vert \Upsilon _{t_{1}}-\Upsilon _{t_{2}}\right\vert
\right\vert _{\mathcal{B}_{\alpha -\beta }}=A^{\Upsilon }+B^{\Upsilon },
\end{equation*}%
where%
\begin{eqnarray*}
A^{\Upsilon } &:&=\left\vert \left\vert \int_{t_{2}}^{t_{1}}S_{t_{1}-s}(%
\mathcal{D}\omega _{s}^{1}-\mathcal{D}\omega _{s}^{2})ds\right\vert
\right\vert _{\mathcal{B}_{\alpha -\gamma }} \\
&\leq &\int_{t_{2}}^{t_{1}}\left\vert \left\vert S_{t_{1}-s}(\mathcal{D}%
\omega _{s}^{1}-\mathcal{D}\omega _{s}^{2})\right\vert \right\vert _{%
\mathcal{B}_{\alpha -\gamma }}ds \\
&\leq &C\int_{t_{2}}^{t_{1}} (t_1-s)^{-\max\{0,\beta-\gamma\}} \left\vert \left\vert \mathcal{D}\omega _{s}^{1}-%
\mathcal{D}\omega _{s}^{2}\right\vert \right\vert _{\mathcal{B}_{\alpha
-\beta }}ds \\
&\leq &C\left( t_{1}-t_{2}\right)^{\min\{1,1+\gamma-\beta\}} \max \left( \left\vert \left\vert \mathcal{%
\omega }^{1}\right\vert \right\vert _{\mathcal{B}_{\alpha }}^{p},\left\vert
\left\vert \mathcal{\omega }^{2}\right\vert \right\vert _{\mathcal{B}%
_{\alpha }}^{p}\right) \left\vert \left\vert \mathcal{\omega }^{1}\mathcal{%
-\omega }^{2}\right\vert \right\vert _{\mathcal{B}_{\alpha }} \\
&\leq &C\left( t_{1}-t_{2}\right) ^{ 1-\beta  }\left(
t_{1}-t_{2}\right) ^{\gamma }\max \left( \left\vert \left\vert \mathcal{%
\omega }^{1}\right\vert \right\vert _{\mathcal{B}_{\alpha }}^{p},\left\vert
\left\vert \mathcal{\omega }^{2}\right\vert \right\vert _{\mathcal{B}%
_{\alpha }}^{p}\right) \left\vert \left\vert \mathcal{\omega }^{1}\mathcal{%
-\omega }^{2}\right\vert \right\vert _{\mathcal{B}_{\alpha }} \\
&\leq &CT^{1-\beta }\left( t_{1}-t_{2}\right) ^{\gamma }\max \left(
\left\vert \left\vert \mathcal{\omega }^{1}\right\vert \right\vert _{%
\mathcal{B}_{\alpha }}^{p},\left\vert \left\vert \mathcal{\omega }%
^{2}\right\vert \right\vert _{\mathcal{B}_{\alpha }}^{p}\right) \left\vert
\left\vert \mathcal{\omega }^{1}\mathcal{-\omega }^{2}\right\vert
\right\vert _{\mathcal{B}_{\alpha }}.
\end{eqnarray*}%
Similarly, we get 
\begin{eqnarray*}
B^{\Upsilon } &:&=\left\vert \left\vert \int_{0}^{t_{2}}\left(
S_{t_{1}-s}-S_{t_{2}-s}\right) (\mathcal{D}\omega _{s}^{1}-\mathcal{D}\omega
_{s}^{2})ds\right\vert \right\vert _{\mathcal{B}_{\alpha -\gamma }} \\
&\leq &\int_{0}^{t_{2}}\left\vert \left\vert \left(
S_{t_{1}-s}-S_{t_{2}-s}\right) (\mathcal{D}\omega _{s}^{1}-\mathcal{D}\omega
_{s}^{2})ds\right\vert \right\vert _{\mathcal{B}_{\alpha -\gamma }}ds \\
&\leq &\int_{0}^{t_{2}}\left\vert \left\vert \left( S_{t_{1}-t_{2}}-I\right)
S_{t_{2}-s}(\mathcal{D}\omega _{s}^{1}-\mathcal{D}\omega
_{s}^{2})\right\vert \right\vert _{\mathcal{B}_{\alpha -\gamma }}ds \\
&\leq &C\left( t_{2}-t_{1}\right) ^{\gamma }\int_{0}^{t_{2}}\left\vert
\left\vert S_{t_{2}-s}(\mathcal{D}\omega _{s}^{1}-\mathcal{D}\omega
_{s}^{2})\right\vert \right\vert _{\mathcal{B}_{\alpha }}ds \\
&\leq &C\left( t_{2}-t_{1}\right) ^{\gamma }\int_{0}^{t_{2}}\frac{1}{\left(
t_{2}-s\right) ^{\beta }}\left\vert \left\vert \mathcal{D}\omega _{s}^{1}-%
\mathcal{D}\omega _{s}^{2}\right\vert \right\vert _{\mathcal{B}_{\alpha
-\beta }}ds \\
&\leq &Ct_{2}^{1-\beta }\left( t_{2}-t_{1}\right) ^{\gamma }\max \left(
\left\vert \left\vert \mathcal{\omega }^{1}\right\vert \right\vert _{%
\mathcal{B}_{\alpha }}^{p},\left\vert \left\vert \mathcal{\omega }%
^{2}\right\vert \right\vert _{\mathcal{B}_{\alpha }}^{p}\right) \left\vert
\left\vert \mathcal{\omega }^{1}\mathcal{-\omega }^{2}\right\vert
\right\vert _{\mathcal{B}_{\alpha }} \\
&\leq &CT^{1-\beta }\left( t_{1}-t_{2}\right) ^{\gamma }\max \left(
\left\vert \left\vert \mathcal{\omega }^{1}\right\vert \right\vert _{%
\mathcal{B}_{\alpha }}^{p},\left\vert \left\vert \mathcal{\omega }%
^{2}\right\vert \right\vert _{\mathcal{B}_{\alpha }}^{p}\right) \left\vert
\left\vert \mathcal{\omega }^{1}\mathcal{-\omega }^{2}\right\vert
\right\vert _{\mathcal{B}_{\alpha }}.
\end{eqnarray*}%
It follows that 
\begin{equation}
\begin{aligned}
& \left\Vert \int_{0}^{\cdot }S_{\cdot -s}(\mathcal{D}\omega
_{s}^{1})ds-\int_{0}^{\cdot }S_{\cdot -s}(\mathcal{D}\omega
_{s}^{2})ds\right\Vert _{C^{\gamma }\left( [0,T];\mathcal{B}_{a-\gamma
}\right) }\\
& \leq CT^{1-\beta }\left( \left\Vert \mathcal{\omega }%
^{1}\right\Vert _{C\left( (0,T];\mathcal{B}_{a}\right) }+\left\Vert \mathcal{%
\omega }^{2}\right\Vert _{C\left( (0,T];\mathcal{B}_{a}\right) }\right)
\left\Vert \mathcal{\omega }^{1}\mathcal{-\omega }^{2}\right\Vert _{\left.
C(0,T];\mathcal{B}_{a}\right) } . \label{222}
\end{aligned}
\end{equation}
Therefore it follows that 
\begin{eqnarray*}
\left\vert \left\vert \int_{0}^{\cdot }S_{\cdot -s}(\mathcal{D}\omega
_{s}^{1})ds-\int_{0}^{\cdot }S_{\cdot -s}(\mathcal{D}\omega
_{s}^{2})ds\right\vert \right\vert _{V^{T}} &\leq &C T^{1-\beta
} R\left\Vert \mathcal{\omega }^{1}\mathcal{-\omega }%
^{2}\right\Vert _{\left. C(0,T];\mathcal{B}_{a}\right) } \\
&\leq &C T^{1-\beta } R\left\vert \left\vert \mathcal{\omega }%
^{1}\mathcal{-\omega }^{2}\right\vert \right\vert _{V^{T}}.
\end{eqnarray*}%
Similarly we have that, see Remark \ref{rem:sewing} and Corollary~\ref{i:reg}
\begin{equation*}
\left\vert \left\vert \int_{0}^{\cdot }S_{\cdot -s}(\mathcal{F}\omega
_{s}^{1})dW_{s}^{H}-\int_{0}^{\cdot }S_{\cdot -s}(\mathcal{F}\omega
_{s}^{2})dW_{s}^{H}\right\vert \right\vert _{V^{T}}\leq C T^{\gamma-\theta
} R \left\vert \left\vert \mathcal{\omega }^{1}\mathcal{-\omega }%
^{2}\right\vert \right\vert _{V^{T}}.
\end{equation*}
Taking $L:=2C\left( T^{1-\beta }+T^{\gamma-\theta}\right) (R+1)$ we deduce the result for $T$
sufficiently small. 
\end{proof}

\begin{remark} The result can be generalised in a more specific way, applicable especially to three-dimensional models which contain a stretching term. 
Consider the following nonlinear equation
\begin{equation}\label{eq:abstract}
dX_t + B(\tilde{dt}, X_t) = \Delta X_tdt
\end{equation}
with
\begin{equation}
\begin{aligned}
B(\tilde{dt}, X_t) &= f(X_t)\cdot\nabla X_t dt   -  X_t \cdot \nabla f(X_t)dt+ (\xi \cdot \nabla X_t  - X_t \cdot \nabla \xi)dW_t^H\\
& = [f(X_t), X_t]dt + [\xi, X_t]dW_t^H
\end{aligned}
\end{equation}
$H\in \left(\frac{1}{2}, 1 \right), \quad f(X_t) = curl^{-1}X_t$,
and initial condition $X_0 \in \mathcal{B}_{\alpha}$. 
Equation \eqref{eq:abstract} admits a unique mild solution in the function space $V^T$.
To prove the equivalent of Proposition \ref{prop:contraction} in the general case a space $D$ is required such that $\Lambda : D \rightarrow D$ and a distance $d_D$ on $D$ with 
\begin{equation*}
d_D(\Lambda(x^1), \Lambda(x^2)) \leq Kd_D(x^1, x^2)
\end{equation*}
and $K<1$. To correctly define $D$ and $d_D$ one needs to choose properly the space(s) in which the integral 
\begin{equation*}
\displaystyle\int_0^{\cdot} B(\tilde{ds}, X_s)
\end{equation*}
is well-defined for $t\in [0, T]$ or for $t\in[0,\tau]$ with suitably-chosen $\tau$. 
We have
\begin{equation*}
\begin{aligned}
\displaystyle\int_0^{\cdot} B(\tilde{ds}, X_s) & = \displaystyle\int_0^{\cdot} B_1(X_s)ds + \displaystyle\int_0^{\cdot} B_2(X_s)dW_s^H
\end{aligned}
\end{equation*}
We actually have the map
\begin{equation}
X \rightarrow \Lambda(X)_t := S(t)X_0 + \Gamma\left( \displaystyle\int_0^{\cdot}B(\tilde{ds},X_s)\right)_t
\end{equation}
with
\begin{equation}
\begin{aligned}
\Gamma\left( \displaystyle\int_0^{\cdot}B(\tilde{ds},X_s)\right)_t &= \displaystyle\int_0^t S(t-s)B(\tilde{ds}, X_s) \\
& = \displaystyle\int_0^t S(t-s)\left(f(X_s)\cdot\nabla X_s ds + \xi \cdot \nabla X_s dW_s^H\right)\\
& = \displaystyle\int_0^t S(t-s)B_1(X_s)ds + \displaystyle\int_0^t S(t-s)B_2(X_s)dW_s^H. 
\end{aligned}
\end{equation}

\end{remark}

\begin{remark}
For equation~\eqref{general} we can define the  estimator~\eqref{est:h} for the Hurst parameter. 
Similar to the proof of Proposition \ref{prop:est:h}, we infer that
\[ \lim\limits_{k\to\infty} V_k = \int_0^t (\mathcal{F}(\omega_s),\varphi)^2~\txtd s.\]
\end{remark}

\section{Applications}\label{sect:applications}
The following are examples of nonlinear operators $\mathcal{D}$ satisfying
the conditions we propose above

\begin{itemize}
\item $\mathcal{B}_{\alpha }=H^{2\alpha}(\mathbb{T}^{2})$ with the norm $%
\left\vert \left\vert x\right\vert \right\vert _{\alpha }:=\left\vert
\left\vert x\right\vert \right\vert _{H^{2\alpha}(\mathbb{T}^{2})}$ and 
\begin{equation*}
\mathcal{D\omega =}u\cdot \nabla \omega
\end{equation*}%
where $u=$curl $^{-1}\omega $. Here $u$ is the velocity of the fluid and $%
\omega $ is the vorticity of the fluid. This is the nonlinear operator
appearing in the equation for 2D\ ideal incompresible fluids (in vorticity
form).
\item $\mathcal{B}_{\alpha }=H^{2\alpha}(\mathbb{T}^{3})$ with the norm $%
\left\vert \left\vert x\right\vert \right\vert _{\alpha }:=\left\vert
\left\vert x\right\vert \right\vert _{H^{2\alpha}(\mathbb{T}^{3})}$ and 
\begin{equation*}
\mathcal{D\omega =}u\cdot \nabla \omega -\omega \cdot \nabla u=\nabla \cdot
\left( u\times \omega \right)
\end{equation*}%
where $u=$curl $^{-1}\omega $. Here $u$ is the velocity of the fluid and $%
\omega $ is the vorticity of the fluid. This is the nonlinear operator
appearing in the equation for 3D\ ideal incompresible fluids (in vorticity
form).
\item $\mathcal{B}_{\alpha }=H_{b}^{2\alpha}(\mathbb{T}^{2})$ with the
norm 
\begin{equation*}
\left\vert \left\vert x\right\vert \right\vert _{\alpha }:=\left\vert
\left\vert x\right\vert \right\vert _{H_{b}^{2\alpha}(\mathbb{T}%
^{2})}=\int_{\mathbb{T}^{2}}x\left( a\right) ^{2}b\left( a\right) da.
\end{equation*}%
These are weighted Sobolev spaces. 
\begin{equation*}
\mathcal{D\omega =}u\cdot \nabla \omega
\end{equation*}%
where 
\begin{equation*}
u+\frac{1}{6}\delta ^{2}b^{2}\nabla (\nabla \cdot u)=\mathrm{curl}%
^{-1}\left( b\omega \right)
\end{equation*}%
This is the nonlinear operator appearing in the great lake equation ($b$ is
the bottom topography)
\item $\mathcal{B}_{\alpha }=H^{2\alpha}(\mathbb{T}^{2}\times \mathbb{T}%
^{2})$ with the norm $\left\vert \left\vert x\right\vert \right\vert
_{\alpha }^{2}:=\left\vert \left\vert x^{1}\right\vert \right\vert
_{H^{2\alpha}(\mathbb{T}^{2})}^{2}+\left\vert \left\vert x^{2}\right\vert
\right\vert _{H^{2\alpha}(\mathbb{T}^{2})}^{2}$%
\begin{equation*}
\mathcal{D}\left( \mathcal{\omega }_{1},\omega _{2}\right) \mathcal{=}\left(
u_{1}\cdot \nabla \omega _{1}-\beta \frac{\partial \psi _{1}}{\partial x}%
,u_{2}\cdot \nabla \omega _{2}-\mu \Delta \omega _{2}-\beta \frac{\partial
\psi _{2}}{\partial x}\right)
\end{equation*}%
where $\psi _{i}$ is the stream function, $\beta $ is the planetary
vorticity gradient, $\mu $ is the bottom friction parameter, $u_{i}$ is the
velocity vector and $\omega _{i}$ is the vorticity of the fluid. The
computational domain $\Omega =\mathbb{T}^{2}\times \lbrack 0,\mathcal{H}]$  is a
horizontally periodic flat-bottom channel of depth $\mathcal{H}=\mathcal{H}_{1}+\mathcal{H}_{2}$ given by
two stacked isopycnal fluid layers of depth $\mathcal{H}_{1}$ and $\mathcal{H}_{2}$.
\end{itemize}
The two layers are related through two elliptic equations: 
\begin{subequations}
\begin{align*}
\omega _{1}& =\Delta \psi _{1}+s_{1}(\psi _{2}-\psi _{1}), \\
\omega _{2}& =\Delta \psi _{2}+s_{2}(\psi _{1}-\psi _{2}),
\end{align*}
with stratification parameters $s_{1}$, $s_{2}$. This is the nonlinear
operator appearing in the two layer quasi-geostrophic equation.
Assumption \ref{ass:d} on the drift term can be verified by Lemma \ref{sobolev}.
\end{subequations}
\vspace{10mm}

\paragraph{Acknowledgements}
A. Blessing acknowledges support by the Deutsche Forschungsgemeinschaft (DFG, German Research Foundation) - CRC/TRR 388 "Rough Analysis, Stochastic Dynamics and Related Fields" - Project ID 516748464 and from DFG CRC 1432 " Fluctuations and Nonlinearities in Classical and Quantum Matter beyond Equilibrium" - Project ID 425217212.

D. Crisan has been supported by the European Research Council (ERC) under the European Union’s Horizon 2020 Research and Innovation Programme, (ERC) Grant Agreement No 856408: Stochastic Transport in Upper Ocean Dynamics (STUOD). 

O. Lang has been partially supported by the European Research Council (ERC) under the European Union’s Horizon 2020 Research and Innovation Programme (ERC), Grant Agreement No 856408: Stochastic Transport in Upper Ocean Dynamics (STUOD). 

\vspace{3mm}
\noindent\textbf{Data availability statement} \\
\noindent Data sharing not applicable to this article as no datasets were generated or analysed during the current study.

\vspace{3mm}
\noindent\textbf{Conflict of interest statement} \\
On behalf of the authors, the corresponding author states that there is no conflict of interest.

\appendix
\section{An alternative proof for the construction of the Young integral}\label{a:young}

We provide an alternative construction of the Young integral based on the sewing lemma similar to \cite[Theorem 2.4]{GH} and \cite[Theorem 4.1]{GHN} tailored to the Young case and transport-type noise.
To this aim we consider the scale of Banach spaces $(\cB_\delta)_{\delta\in \R}$, where we will set $\delta=\alpha-1/2$ to incorporate transport-type noise, as seen in Section \ref{sect:sewinglemma}.~We first introduce some notations.
 
 \begin{itemize}
     \item  $[W^H]_\gamma$ denotes the $\gamma$-H\"older norm of the noise on an arbitrary time interval.
          \item We set $\Delta_n:=\{ 0\leq t_1\leq t_n\leq T \}$ and consider the space $C^{\gamma}_2([0,T];\cB_\delta)$ of functions $g=(g_{s,t}):\Delta_2 \to \cB_\delta$ for which 
    \[ \sup\limits_{0\leq s <t\leq T} \frac{\|g_{s,t}\|_{\cB_\delta}}{|t-s|^\gamma}<\infty. \]
    \item We introduce the increment operators $(\delta f)_{s,t} =f_t-f_s$, $(\delta^S f)_{s,t}=f_t-S_{t-s}f_s$, $(\delta f)_{s,u,t}=f_{s,t}-f_{s,u}-f_{u,t}$. Further, the space $C^{\gamma_1,\gamma_2}([0,T];\cB_\delta)$ consists of functions $h=(h_{s,u,t}):\Delta_3 \to \cB_\delta$ such that
    \[ \sup\limits_{(s,u,t)\in \Delta_3} \frac{\|h_{s,u,t}\|_{\cB_\delta}}{|t-u|^{\gamma_1}|u-s|^{\gamma_2} }<\infty. \]
    \item For our aims, in order to define the Young integral we consider the space $\mathcal{Y}$ consisting of two-index elements $\xi\in C^\gamma_2([0,T];\cB_{\delta})$ such that $\delta \xi \in C^{\gamma,\gamma}_3(\cB_{\delta-\gamma})$. We endow $\mathcal{Y}$ with the norm 
    \[ \|\xi\|_{\mathcal{Y}} =\|\xi\|_{C^\gamma_2([0,T];\cB_{\delta})} +\| \delta \xi\|_{C^{\gamma,\gamma}([0,T];\cB_{\delta-\gamma})}.\]

 \end{itemize}

\begin{theorem}{\em(Young integral)}\label{young:a}
    Let  $\xi\in\mathcal{Y}$. Then 
     there exists a map $\mathcal{I}: \mathcal{Y} \to C([0,T];\cB_\delta)\cap C^\gamma([0,T];\cB_{\delta-\gamma}) $ such that $\mathcal{I}_0=0$ which satisfies for every $0\leq s \leq t\leq T$ and $0\leq \theta<2\gamma$ the estimate:
    \begin{align}\label{est:young} 
    \| (\delta^S \mathcal{I}(\xi))_{s,t} - S_{t-s}\xi_{s,t}\|_{\delta-\gamma+\theta} \lesssim [W^H]_\gamma \|\xi\|_{\mathcal{Y}} |t-s|^{2\gamma-\theta}.
    \end{align}
    In particular, the convolution 
   \begin{align*}
      \mathcal{I}_t(\xi) := \lim\limits_{\pi\in\cP([0,t]),|\pi|\to 0} \sum\limits_{[u,v]\in\pi} S_{t-u}\xi_{u,v} 
   \end{align*}
   exists in $\cB_{\delta-\gamma}$.

\end{theorem}
\begin{proof}
We prove the statement for dyadic partitions of the interval $[s,t]$. We denote by $\pi_k$ the $k$-th dyadic partition of $[s,t]$, i.e.~$\pi_k=\{ s\leq t_0 <t_1,\ldots <t_{2^k}=t  \}$, so $t_i = s+\frac{i(t-s)}{2^k}$ for $i={0,\ldots ,2^{k}-1}$.
We define for $\pi\in\cP([0,t])$ the integral
\[I^{\pi}_t :=\sum\limits_{[u,v]\in \pi}  S_{t-u}Y_u W^H_{v,u}= \sum\limits_{[u,v]\in \pi}  S_{t-u}\xi_{u,v}\] 
Then we get for $m=(u+v)/2$ that
\begin{align*} I^{\pi_k}_{s,t}- I^{\pi_{k+1}}_{s,t}
&=\sum\limits_{[u,v]\in \pi_k }S_{t-u}\xi_{u,v} -\sum\limits_{[u,v]\in \pi_k} S_{t-u} \xi_{u,m} -S_{t-m}\xi_{m,v}\\
&=\sum\limits_{[u,v]\in\pi_k} S_{t-u} \delta\xi_{u,m,v} +S_{t-m} (S_{m-u}-I)\xi_{m,v}. 
\end{align*}
  We show that $(I^{\pi_k}_{s,t})$ is a Cauchy sequence in $\cB_{\delta-\gamma}$. \\
  Using regularizing properties of analytic semigroups we get that 
\begin{align*}
    \Big\| I^{\pi_k}_{s,t}-I^{\pi_{k+1}}_{s,t}\Big\|_{\delta-\gamma+\theta}& \lesssim \sum\limits_{[u,v]\in\pi_k} \| S_{t-u} \|_{\cL(\cB_{\delta-\gamma},\cB_{\delta-\gamma+\theta})} \|(\delta \xi)_{u,m,v}\|_{\cB_{\delta-\gamma}} \\& \hspace*{15 mm}+ \|S_{t-m}\|_{\cL(\cB_{\delta-\gamma},\cB_{\delta-\gamma+\theta})}  \| S_{m-u}-I\|_{\cL(\cB_\delta,\cB_{\delta-\gamma}) }\|\xi_{m,v}\|_{\cB_\delta}  \\
    & \lesssim \|\xi\|_{\mathcal{V}} \sum\limits_{[u,v]\in\pi_k} (t-m)^{-\theta} (v-m)^\gamma(m-u)^\gamma\\
    & \lesssim \|\xi\|_{\mathcal{V}} \sum\limits_{[u,v]\in\pi_k} |t-m|^{-\theta} |v-m|^{2\gamma-1}|m-u|\\
    & \lesssim \|\xi\|_{\mathcal{V}}  \sum\limits_{[u,v]\in\pi_k} |t-m|^{\theta'-\theta} |v-m|^{2\gamma-1-\theta'}|m-u|\\
    & \lesssim \|\xi\|_{\mathcal{V}} 2^{-k(2\gamma-1-\theta')} |t-s|^{2\gamma-1-\theta'} \sum\limits_{[u,v]\in\pi_k}|t-m|^{\theta'-\theta}|m-u|\\
    & \lesssim  \|\xi\|_{\mathcal{V}} 2^{-k(2\gamma-1-\theta')} |t-s|^{2\gamma-1-\theta'} \int_s^t |t-r|^{\theta'-\theta}~\txtd r\\
    & \lesssim  \|\xi\|_{\mathcal{V}} 2^{-k(2\gamma-1-\theta')}|t-s|^{2\gamma-\theta}.
\end{align*}
Choosing $\theta'$ such that $2\gamma-1-\theta'>0$ and summing over $k\in\mathbb{N}$ proves~\eqref{remainder:integral2}.
\end{proof}

\begin{corollary}{\em(Young integral for transport-type noise)}\label{young:t}
        Let $Y\in C([0,T];\cB_{\alpha-1/2})\cap C^\gamma([0,T];\cB_{\alpha-\gamma-1/2})$.\\Then 
     there exists a map $\mathcal{I}: C([0,T];\cB_{\alpha-1/2})\cap C^\gamma([0,T];\cB_{\alpha-\gamma-1/2})  \to C([0,T];\cB_\alpha)\cap C^\gamma([0,T];\cB_{\alpha-\gamma}) $ such that $\mathcal{I}_0=0$, 
   \begin{align*}
\mathcal{I}_t =\int_0^t S_{t -s}Y_s dW^H_s       =\lim\limits_{\pi\in\cP([0,t]),|\pi|\to 0} \sum\limits_{[u,v]\in\pi} S_{t-u} Y_u W^H_{u,v}, 
   \end{align*}
which satisfies the following estimates for all $0\leq s \leq t \leq T$:
\begin{align}\label{remainder:integral2}
    \Big\| \int_s^t S_{t-r}Y_r~d W^H_r -S_{t-s}Y_s W^H_{s,t} \Big\|_{\alpha-\gamma}\lesssim [W^H]_\gamma \max\{ \|Y\|_{0,\alpha-1/2}, \|Y\|_{\gamma,\alpha-1/2-\gamma}\} (t-s)^{2\gamma-1/2}
    \end{align}
    and
    \begin{align}\label{remainder:integral}
    \Big\| \int_s^t S_{t-r}Y_r~d W^H_r -S_{t-s}Y_s W^H_{s,t} \Big\|_{\alpha}\lesssim [W^H]_\gamma \max\{ \|Y\|_{0,\alpha-1/2}, \|Y\|_{\gamma,\alpha-1/2-\gamma}\} (t-s)^{\gamma-1/2}.
\end{align}

\end{corollary} 
\begin{proof}
    The statement follows from Theorem \ref{young:a} setting $\delta=\alpha-1/2$ and using the approximation term $\xi_{s,t}=Y_s W^{H}_{s,t}$.~In order to obtain~\eqref{remainder:integral} we set  $\theta=1/2$ in~\eqref{est:young},  respectively $\theta=\gamma+1/2$ for~\eqref{remainder:integral}.
    \end{proof}

For the sake of completeness, we show that the convolution improves the spatial regularity by a parameter $\sigma<\gamma$. This justifies the choice of Young's integral in the context of transport type noise.

\begin{corollary}\label{i:reg}
    Let $Y\in C([0,T];\cB_\delta)\cap C^\gamma([0,T];\cB_{\delta-\gamma})$ and $0\leq \sigma<\gamma$. Then the integral map constructed in Theorem~\ref{young:a} is continuous from $C([0,T];\cB_\delta)\cap C^\gamma([0,T];\cB_{\delta-\gamma})$ to $C([0,T];\cB_{\delta+\sigma})\cap C^\gamma([0,T];\cB_{\delta-\gamma+\sigma})$.
\end{corollary}
\begin{proof}    
We first show the H\"older continuity. To this aim, we compute for $0\leq s<t\leq T$
   \begin{align*}
       &\int_0^t S_{t-r}Y_r~d W^H_r -\int_0^s S_{s-r}Y_r~d W^H_r\\
       & = \int_s^t S_{t-r} Y_r~\txtd W^H_r +(S_{t-s}-I)\int_0^s S_{s-r}Y_r~d W^H_r.
   \end{align*} 
We set $A:= \max\{ \|Y\|_{0,\delta} ,\| Y\|_{\gamma,\delta-\gamma}\} $. The first term gives due to~\eqref{remainder:integral}  
   \begin{align*}
       \Big\| \int_s^t S_{t-r} Y_r~d W^H_r\Big\|_{\delta-\gamma+\sigma} &\lesssim A (t-s)^{2\gamma-\sigma} + \|S_{t-s}Y_s W^H_{s,t}\|_{\delta-\gamma+\sigma}\\
       & \lesssim A (t-s)^{2\gamma-\sigma} +(t-s)^{\gamma+\varepsilon} [W^H]_{\gamma+\varepsilon} ,
   \end{align*}
   where we used that $W^H$ is $(\gamma+\varepsilon)$-H\"older continuous for $\varepsilon<H-\gamma$. 
Furthermore
\begin{align*}
\Big\| (S_{t-s}-I) \int_0^s S_{s-r} Y_r \txtd W^H_r  \Big\|_{\delta-\gamma+\sigma}& \leq \|S_{t-s}-I\|_{\cL(\cB_{\delta+\sigma},\cB_{\delta-\gamma+\sigma})} \Big\| \int_0^s S_{s-r} Y_r~\txtd W^H_r \Big\|_{\delta+\sigma}\\
& \lesssim (t-s)^\gamma [ A s^{\gamma-\sigma} +\| S_s y_0 W^H_{0,s} \|_{\delta+\sigma} ] \\
& \lesssim (t-s)^\gamma [A s^{\gamma-\sigma} +\|S_s\|_{\cL(\cB_{\delta},\cB_{\delta+\sigma})} \|Y\|_{\delta} s^\gamma [W^H]_\gamma ]\\
&\lesssim A (t-s)^\gamma  s^{\gamma-\sigma} .
\end{align*}
   Putting these estimates together, we get 
   \begin{align}\label{hoelder:integral}
       \Big\| \int_0^\cdot S_{\cdot-r}Y_r~\txtd W^H_r \Big\|_{\gamma,\alpha-\gamma} \lesssim A T^{(\gamma-\sigma) \wedge\varepsilon }. 
   \end{align}
Based on~\eqref{remainder:integral2} we get the following estimate for the stochastic convolution in $C([0,T],\cB_\alpha)$. We get 
\begin{align*}
    \Big\| \int_s^t S_{t-r}y_r~\txtd W^H_r \Big\|_{\delta+\sigma} &\lesssim [W^H]_\gamma A (t-s)^{\gamma-\sigma}\\
    &+ \|S_{t-s}\|_{L(\cB_\delta,\cB_{\delta+\sigma})} \|Y\|_{0,\delta} [W^H]_\gamma (t-s)^\gamma\\
    & \lesssim A [W^H]_\gamma (t-s)^{\gamma-\sigma}.
\end{align*}
Therefore 
\[ \Big\| \int_0^\cdot S_{\cdot-r} Y_r~\txtd W^H_r \Big\|_{0,\delta+\sigma} \lesssim A [W^H]_\gamma   T^{\gamma-\sigma}. \]
This proves the statement.
\end{proof}

\begin{remark}
\begin{itemize}
\item Setting $\delta=\alpha-1/2$, $\sigma=0$ and $\theta$ as in the proof of Corollary~\ref{young:t}, we get the regularity of the convolution in $C([0,T];\cB_{\alpha-1/2})\cap C^\gamma([0,T];\cB_{\alpha-1/2-\gamma})$, as justified in Theorem \ref{s}.
\item Choosing $\delta=\alpha$ we are in the setting of section \ref{sect:genproofs}.
\end{itemize}
\end{remark}

\end{document}